\documentclass{article}[utf8, 12pt]%
\usepackage{amsmath}
\usepackage{amsfonts}
\usepackage{amssymb}
\usepackage{mathrsfs}
\usepackage{graphicx}
\usepackage{hyperref}
\usepackage{stmaryrd}
\usepackage{authblk}
\usepackage{enumerate}
\usepackage{bbm}
\usepackage{color}%
\usepackage{esint}
\usepackage{mathtools}
\usepackage{etoolbox}
\let\origpartial\partial
\usepackage{mathabx}
\apptocmd{\thebibliography}{\renewcommand{\sc}{}}{}{}
\usepackage[title,titletoc]{appendix}
\usepackage{xpatch,amsthm}
\makeatletter
\xpatchcmd{\@thm}{\fontseries\mddefault\upshape}{}{}{} % same font as thm-header
\makeatother
\providecommand{\U}[1]{\protect\rule{.1in}{.1in}}
\newtheorem{theorem}{Theorem}

\newtheorem{claim}[theorem]{Claim}

\newtheorem{definition}[theorem]{Definition}

\newtheorem{lemma}[theorem]{Lemma}

\newtheorem*{question*}{Question}
\newtheorem{proposition}[theorem]{Proposition}
\newtheorem{remark}[theorem]{Remark}

\newcommand{\X}{\mathfrak X}
\newcommand{\Xu}{\mathfrak X u}

\begin{document}
	\author[a]{Arka Mallick\thanks{arkamallick@iisc.ac.in}}
	\author[b]{Swarnendu Sil\thanks{swarnendusil@iisc.ac.in}\thanks{Supported by the ANRF-SERB MATRICS Project grant MTR/2023/000885}}
	
	\affil[a,b]{Department of Mathematics \\ Indian Institute of Science\\ Bangalore, India}
	
	\title{Continuity estimates for variable growth variational problems in the Heisenberg group}	
	\date{}
	%	\title{Continuity estimates for variable exponent quasilinear equations in Heisenberg groups}
	\maketitle 
	\begin{abstract}
		We study regularity results for local minimizers of variable growth variational problem in Heisenberg groups under suitable integrability assumption on the horizontal gradient of the exponent function. More precisely, our main focus is on the continuity properties of the horizontal gradient $\mathfrak{X} u$, where $u \in HW_{\text{loc}}^{1,1}$ is a local minimizer of the functional   
		\begin{align*}
			I [u]:= \int_{\Omega} \frac{1}{p(x)}\left\lvert \mathfrak{X} u \right\rvert^{p(x)}\ \mathrm{d}x
		\end{align*} in a 
		domain of $\Omega \subset \mathbb{H}_{n},$ where $\mathbb{H}_{n}$ is the Heisenberg group with homogeneous dimension $Q=2n+2,$ where $p \in HW^{1,1}\left( \Omega\right)$ and we assume suitable integrability hypothesis on $\mathfrak{X} p.$ We prove 
		\begin{enumerate}[(a)]
			\item if $\mathfrak{X} p \in L^{q}\left( \Omega; \mathbb{R}^{2n}\right)$ with $q>Q,$ then $\X u$ is H\"{o}lder continuous and  
			\item if $\mathfrak{X} p \in L^{(Q,1)}\log L \left( \Omega; \mathbb{R}^{2n}\right),$ then $\mathfrak{X} u$ is continuous. 
		\end{enumerate} 
		In fact, in the non-borderline case $(a)$, we prove H\"{o}lder continuity of the horizontal gradient for the minima of more general variational problems, assuming $p$ to be H\"{o}lder continuous, i.e. without any assumption on the weak derivative of $p.$ To the best of our knowledge, the present work is the first regularity result for minimizers of variable growth variational problems in the setting of Heisenberg groups.  
	\end{abstract}
	\noindent\textbf{Keywords:} Heisenberg group, $p$-Laplacian, quasilinear subelliptic equations, variable exponent, $p(\cdot)$-Laplacian, continuity of horizontal gradient. \smallskip 
	
	\noindent\textbf{MSC 2020:} Primary 35B65, 35H20, 35J62; Secondary 35H10, 35J70, 35J75.
	\tableofcontents 
	
	\section{Introduction and main results}
	Let $n \geq 1 $  be an integer and let $\mathbb{H}_{n}$ denote the $n$-th Heisenberg group over the reals. Let $Q:=2n+2$ stands for the homogeneous dimension of $\mathbb{H}_{n}.$ We are interested in the continuity properties of the horizontal gradient $\X u$ for a local minimizer $u$ of the following functional 
	\begin{align}
		I\left[w\right] = \int_{\Omega} \frac{1}{p \left(x\right)}\left\lvert \mathfrak{X} w \right\rvert^{p(x)}\ \mathrm{d}x, \label{px-energy}
	\end{align}
	where $\Omega \subset \mathbb{H}_{n}$ is an open, bounded subset. The study of `variable growth' problems in the Euclidean case started gaining attention after Zhikov investigated them in \cite{Zhikov_higherintegrability}, \cite{Zhikov_Lavrentiev}, where he studied the variational problem 
	\begin{align*}
		\inf \left\lbrace \int_{U} \frac{1}{p(x)}\left\lvert \nabla u \right\rvert^{p(x)}\ \mathrm{d}x: u \in u_{0} + W_{0}^{1, p\left(\cdot\right)}\left( U \right)\right\rbrace,  
	\end{align*}
	where $U \subset \mathbb{R}^{n}$ is an open bounded subset. The Euler-Lagrange equation is the so-called $p(x)$-Laplacian, i.e. 
	\begin{align*}
		\operatorname{div} \left(\left\lvert  \nabla u \right\rvert^{p(x)-2} \nabla u \right) &= 0 &&\text{ in } U. 
	\end{align*} Apart from being mathematically interesting, these type of variable growth problems arise in a variety of applications, e.g. in the theory of electrorheological fluids, thermistor problem, fluid flow in porous media, magnetostatics etc. Zhikov \cite{Zhikov_Lavrentiev} showed that a Lavrentiev gap can occur and the minimizers need not have any additional regularity if the exponent function $p$ is allowed to be too irregular (see \cite{Balci_Diening_Surnachev_LavrentievGap} and the references therein for the state-of-the-art at present). On the other side, when $p$ is regular enough, the local H\"{o}lder continuity of $\nabla u$ follows from Acerbi-Mingione \cite{Acerbi_Mingione_C1alpha}. 
	
	On the subelliptic side, regularity questions for standard $p$-growth variational problems are being investigated in a number of important works, namely Capogna-Danielli-Garofalo \cite{Capogna_Danielli_Garofalo}, Capogna \cite{Capogna_regularity}, Capogna \cite{Capogna_quasiconformal}, Capogna-Garofalo \cite{Capogna_Garofalo_partialregularity_step2} (see also the books Capogna-Danielli \cite{Capogna_Danielli_book}, Riccotti \cite{Ricciotti_pLaplaceHeisenberg}), Mingione-Zatorska-Goldstein-Zhong \cite{Mingione_ZatorskaGoldstein_Zhong}, Zhong \cite{zhong2018regularityvariationalproblemsheisenberg}, Mukherjee-Zhong \cite{Mukherjee_Zhong}, Mukherjee-Sire \cite{Mukherjee_Sire}, Mukherjee \cite{Mukherjee_Lipschitz}, Mukherjee \cite{Mukherjee_C1alpha}, Citti-Mukherjee \cite{Citti_Mukherjee_Step2}, \cite{Mallick_Sil_ExcessDecay}. A  relatively recent survey can be found in Capogna-Citti-Zhong \cite{Capogna_Citti_Zhong}.  However, to the best of our knowledge, variable growth variational problem in the subelliptic setting has not been explored yet.

	Our first result can be viewed as a subelliptic analogue of the results in \cite{Acerbi_Mingione_C1alpha}, where we prove local H\"{o}lder continuity (in the sense of Folland-Stein, but this also implies H\"{o}lder continuity with respect to the Euclidean metric with half the H\"{o}lder exponent) for the horizontal gradient for a slightly more general functional. Throughout the article, we will use the notation $\gamma'$ to denote $\gamma/(\gamma-1)$ for any $\gamma \in (1,\infty)$.
	\begin{theorem}\label{Theorem holder continuity}
		Let $0 < \lambda \leq \Lambda < \infty$ and $1 < \gamma_{1} \leq \gamma_{2} < \infty$ be real numbers and let $\gamma := \min\left\lbrace \gamma_{1}, \gamma'_2,  2Q'. \right\rbrace.$ Let $\Omega \subset \mathbb{H}_{n}$ be open and bounded. Let $a:\Omega \rightarrow [\nu, \Lambda]$ be measurable and $p:\Omega \rightarrow [\gamma_{1}, \gamma_{2}]$ be continuous. Suppose $F \in L^{\gamma^{'}}\left(\Omega; \mathbb{R}^{2n}\right)$ and let $u$ be a local minimizer to the functional 
		\begin{align*}
			J\left[u\right] = \int_{\Omega} \frac{a\left(x\right)}{p \left(x\right)}\left\lvert \mathfrak{X} u \right\rvert^{p(x)}\ \mathrm{d}x - \int_{\Omega} \left\langle F,  \mathfrak{X} u \right\rangle\ \mathrm{d}x. 
		\end{align*}
		Then the following holds. 
		\begin{enumerate}[(a)]
			\item If $F \in L^{q\gamma^{'}}_{\rm{loc}}\left(\Omega; \mathbb{R}^{2n}\right)$ for some $q>1$ and 
			\begin{align}\label{log holder condition thm}
				\lim\limits_{R \rightarrow 0} \omega_p\left(R\right)\log \left(\frac{1}{R}\right):= L < +\infty, 
			\end{align}
			then there exists $\sigma >0$ such that $\X u \in L^{p\left(\cdot\right)\left(1 + \sigma\right)}_{\rm{loc}}\left(\Omega; \mathbb{R}^{2n}\right).$
			
			\item If $a$ is continuous, $F \in {\rm{BMO}}_{\rm{loc}}\left(\Omega; \mathbb{R}^{2n}\right)$ and 
			\begin{align}\label{vanishing log holder condition thm}
				\lim\limits_{R \rightarrow 0} \omega_p\left(R\right)\log \left(\frac{1}{R}\right) = 0 , 
			\end{align}
			then $u \in C^{0, \mu}_{\rm{loc}}\left(\Omega\right)$	for every $0 \leq \mu < 1.$ 
			\item If $a \in C^{0, \alpha_{1}}_{\rm{loc}}\left(\Omega\right)$, $F \in C^{0, \alpha_{2}}_{\rm{loc}}\left(\Omega; \mathbb{R}^{2n}\right)$ and $p \in C^{0, \alpha_{3}}\left(\Omega\right)$ for some exponents $\alpha_{1}, \alpha_{2}, \alpha_{3} \in (0, 1),$ then there exists $\beta \in (0,1)$ such that $\X u \in C^{0, \beta}_{\rm{loc}}\left(\Omega; \mathbb{R}^{2n}\right).$
		\end{enumerate}
	\end{theorem}
	We remark that the result in \cite{Acerbi_Mingione_C1alpha} applies to far more general, even non-differentiable functionals, for which no Euler-Lagrange equation is available. While this can also be done, in this work, we chose to restrict ourselves to the functional $J$ only and use the Euler-Lagrange equations, as our main interest here is the functional $I$ as defined in \eqref{px-energy}. Observe that for the functional $I$, i.e. when $a\equiv 1$ and $F \equiv 0,$ conclusion $(c)$ of Theorem \ref{Theorem holder continuity} and Sobolev embedding (see e.g. Proposition \ref{Morrey's theorem} below) implies that if $\X p \in L^{q}$ with $q>Q,$ then $\Xu$ is H\"{o}lder continuous. Our primary goal in the present work, however lies in deducing the following borderline continuity estimates for $\X u$.
	\begin{theorem}\label{main theorem}
		Let $1 < \gamma_{1} \leq \gamma_{2} < \infty$ be real numbers. Let $\Omega \subset \mathbb{H}_{n}$ be open and bounded. Let $p:\Omega \rightarrow [\gamma_{1}, \gamma_{2}]$ be a $HW^{1,1}\left(\Omega\right)$ function. Let $w$ be a local minimizer to the functional 
		\begin{align*}
			I\left[w\right] = \int_{\Omega} \frac{1}{p \left(x\right)}\left\lvert \mathfrak{X} w \right\rvert^{p(x)}\ \mathrm{d}x. 
		\end{align*}
		If $\X p \in L^{\left(Q,1\right)}\log L \left( \Omega; \mathbb{R}^{2n}\right), $ then $\X w$ is continuous in $\Omega.$ Moreover, 
		\begin{enumerate}[(i)]
			\item there exists a constant $c_{3} = c_{3}\left( \texttt{data} \right) \geq 1$ and a positive radius $R_{1} = R_{1}\left( \texttt{data}, \left\lVert \left\lvert \X w \right\rvert^{p\left(\cdot\right)}\right\rVert_{L^{1}\left( \Omega\right)} \right) >0$ such that if $0 < R \leq R_{1},$ then we have 
			\begin{align}\label{sup bound homogeneous}
				\sup\limits_{B\left(x, R/2\right)} \left\lvert \X w \right\rvert \leq c_{3}\fint_{B\left(x, R\right)}\left( 1 +  \left\lvert \X w \right\rvert \right)  
			\end{align}
			holds whenever $B(x, R) \subset \Omega$ is a metric ball and  
			\item for any $\varepsilon \in (0,1)$ and any $A \geq 1,$ there exists a positive constant $\tau_{1} = \tau_{1}\left( \texttt{data}, \left\lVert \left\lvert \X w \right\rvert^{p\left(\cdot\right)}\right\rVert_{L^{1}\left( \Omega\right)}, \varepsilon, A \right) \in (0, 1/4)$ such that  
			\begin{align}\label{osc bound homogeneous}
				\sup\limits_{B\left(x_{0}, R/2\right)} \left\lvert \X w \right\rvert \leq A\lambda  \implies \sup\limits_{x, y \in B\left(x_{0}, \tau R\right)} \left\lvert \X w \left( x\right) - \X w \left( y\right) \right\rvert \leq \varepsilon \lambda, 
			\end{align}
			for all $0 < \tau \leq \tau_{1},$ whenever  $B(x_{0}, R) \subset \Omega$ is a metric ball and $\lambda >0$ is a positive constant,    
		\end{enumerate}
		where $\texttt{data} = Q, \gamma_{1}, \gamma_{2}, \left\lVert p\right\rVert_{L^{1}\left( \Omega\right)} + \left\lVert \X p \right\rVert_{L^{\left(Q,1\right)}\log L\left( \Omega; \mathbb{R}^{2n}\right)}.$
	\end{theorem}
	Note that we have 
	\begin{align*}
		L^{Q} \subsetneq L^{\left(Q,1\right)} \subsetneq L^{\left(Q,1\right)}\log L \subsetneq L^q \qquad \text{ for all } q>Q. 
	\end{align*}
	So our assumptions here can be viewed as the limiting case of the H\"{o}lder continuity result. 
	In the Euclidean setting, Ok \cite{Ok_pxStein}  proved the continuity of $\nabla u,$ where $u$ is a weak solution to 
	\begin{align}\label{intro pxlaplacian euclidean}
		\operatorname{div} \left(a\left(x\right)\left\lvert  \nabla u \right\rvert^{p(x)-2} \nabla u \right) &= f &&\text{ in } U, 
	\end{align}
	assuming $a$ to be Dini continuous, $p$ to be log-Dini continuous and $f \in L^{(n,1)}$ (see also Ok \cite{Ok_C1minima_px}).  This result generalized the `nonlinear Stein theorem' of Kuusi-Mingione \cite{KuusiMingione_Steinequationslinearpotentials}, \cite{KuusiMingione_nonlinearStein} to the variable exponent setting. Baroni \cite{Baroni_pxStein} also proved the continuity of $\nabla u$, for \eqref{intro pxlaplacian euclidean} when $f=0,$ under Lorentz and Lorentz-Zygmund type integrability assumptions on the derivatives of $p$ and $a,$ namely that $\nabla a \in L^{(n,1)}$ and $\nabla p \in L^{(n,1)}\log L.$ We remark that our assumption on $p$ in Theorem \ref{main theorem} \emph{does not} imply that $p$ is log-Dini continuous. So Theorem \ref{main theorem} falls outside Ok's framework and is more in the spirit of Baroni's results. However, there seems to be an error in \cite{Baroni_pxStein} (see Remark \ref{Baroni error}). Our arguments can easily be adapted to the easier Euclidean setup to furnish a correct proof of Baroni's results in \cite{Baroni_pxStein} (see Remark \ref{Baroni error}). \smallskip

	The borderline continuity result for $\X u,$ when $u$ is a weak solution to the system  	
	\begin{align}\label{intro pxlaplacian Heisenberg}
		\operatorname{div}_{\mathbb{H}} \left(a\left(x\right)\left\lvert  \X u \right\rvert^{p(x)-2} \X u \right) &= f &&\text{ in } \Omega, 
	\end{align}
	is investigated in the upcoming work \cite{MallickSil_SubellipticStein}. While Theorem \ref{main theorem} here is a special case of the main result there, the estimates \eqref{sup bound homogeneous}, \eqref{osc bound homogeneous} are crucially used to derive the general result. Thus, to avoid circularity, one needs a direct, independent proof of these estimates and this is the main motivation for the present article. The conclusions of Theorem \ref{main theorem} also holds without any weak differentiability assumption on $p$ if we assume $p$ to be log-Dini continuous. However, as we do not need this intermediate step to prove the general result under the log-Dini assumption on $p$, we neither state nor prove this result here.  \smallskip

	There are two main ingredients in the proof of Theorem \ref{main theorem}. The first is an `excess-to-excess' decay estimate, which is proved for the setting of Heisenberg group for the first time in Mallick-Sil \cite{Mallick_Sil_ExcessDecay}. The other is a delicate linearization argument, carried out in Section \ref{comparison with bounds}, which is also the main novelty in the present work. Once we have these two ingredients, the pointwise bound and continuity can be established by an iteration argument. \smallskip 
	
	The rest of the article is organized as follows. Section \ref{notations} gathers the notations and Section \ref{prelim} records the some preliminary facts and results that we would use in the article. Section \ref{higher integrability section} proves the crucial higher integrability result that we would need. Section \ref{Holder continuity section} is devoted to proving Theorem \ref{Theorem holder continuity} and Section \ref{continuity section} proves Theorem \ref{main theorem}.  
	\section{Notations}\label{notations} 
	We shall follow the notations in \cite{Mallick_Sil_ExcessDecay}. 
	\begin{itemize}
		\item For $n\geq 1$, the \textit{Heisenberg Group} denoted by $\mathbb{H}_{n}$, is identified with the Euclidean space 
		$\mathbb{R}^{2n+1}$ with the group operation 
		\begin{equation}\label{eq:group op}
			x\cdot y\, := \Big(x_1+y_1,\ \dots,\ x_{2n}+y_{2n},\ t+s+\frac{1}{2}
			\sum_{i=1}^n (x_iy_{n+i}-x_{n+i}y_i)\Big)
		\end{equation}
		for every $x=(x_1,\ldots,x_{2n},t),\, y=(y_1,\ldots,y_{2n},s)\in \mathbb{H}_{n}$. It is a non-Abelian simply connected stratified nilpotent Lie group.
		\item The left invariant vector fields 
		\[ X_i=  \partial_{x_i}-\frac{x_{n+i}}{2}\partial_t, \quad
		X_{n+i}=  \partial_{x_{n+i}}+\frac{x_i}{2}\partial_t,\] 
		for every $1\leq i\leq n$ and the only
		non zero commutator $T = \partial_t,$ constitutes the
		canonical basis of the Lie algebra. We call $X_1,\ldots, X_{2n}$ as \textit{horizontal
			vector fields} and $T$ as the \textit{vertical vector field}.
		\item The \textit{Horizontal gradient} of a scalar function $ f: \mathbb{H}_{n} \to \mathbb{R}$ is denoted by 
		$\X f  = (X_1f,\ldots, X_{2n}f).$ For a vector valued function 
		$F = (f_1,\ldots,f_{2n}) : \mathbb{H}_{n}\to \mathbb{R}^{2n}$, the 
		\textit{Horizontal divergence} is defined as 
		$$ \operatorname{div}_{\mathbb{H}} (F)  =  \sum_{i=1}^{2n} X_i f_i .$$
		\item The bi-invariant Haar measure of $\mathbb{H}_{n}$ is just the Lebesgue 
		measure of $\mathbb{R}^{2n+1}$. $\left\lvert A \right\rvert$ will denote the $(2n+1)$- dimensional Lebesgue measure of a measurable set $A \subset \mathbb{H}_{n}.$ For $A \subset \mathbb{H}_{n}$ measurable with $\left\lvert A \right\rvert >0$ and a measurable function $g:A \rightarrow \mathbb{R}^{N}$, taking values in any Euclidean space $\mathbb{R}^{N},$ $N \geq 1, $ its integral average will be denoted by the notation
		\begin{align*}
			\left(g\right)_{A} := \fint_{A} g\left(x\right)\ \mathrm{d}x := \frac{1}{\left\lvert A \right\rvert} \int_{A} g\left(x\right)\ \mathrm{d}x . 
		\end{align*} 
		If $A=B(x_0, R)$ for $x_0\in \mathbb{H}_n$ and $R>0$, then $\left(g\right)_{x_0, R}$ and $\left(g\right)_A$ denotes the same quantity. If $x_0$ is fixed in the context, then $\left(g\right)_{x_0, R}$ and $\left(g\right)_{R}$ denotes the same quantity. An important property of the integral averages that we would use throughout the rest 
		is 
		\begin{align}\label{minimality of mean}
			\left( \fint_{A} \left\lvert g - \left(g\right)_{A}\right\rvert^{q}\ \mathrm{d}x \right)^{\frac{1}{q}} \leq 2 	\left( \fint_{A} \left\lvert g - \xi \right\rvert^{q}\ \mathrm{d}x \right)^{\frac{1}{q}}, 
		\end{align} 
		for any $ \xi \in \mathbb{R}^{N} $ and any $ 1 \leq q < \infty.$ 
		
		\item The \textit{Carnot-Carath\`eodory metric} (CC-metric) is a left-invariant metric on $\mathbb{H}_{n},$ denoted by $d_{\text{CC}},$ is defined as the length of the shortest horizontal curves, connecting two points. This metric is equivalent to the \textit{Kor\`anyi metric}
		Throughout this article we use CC-metric balls, which are defined to be 
		\begin{align*}
			B_r(x) := \left\lbrace y\in\mathbb{H}_{n}: d_{\text{CC}}(x,y)<r \right\rbrace \qquad \text{ for }r>0 \text{ and } x \in \mathbb{H}_{n}.
		\end{align*} However, by virtue of the equivalence of the metrics, all assertions in this article for CC-balls can be restated to Kor\`anyi balls.  For any CC-metric ball $B_r \subset \mathbb{H}_{n}$ with $r>0,$ there exists a constant $c_{n} = c_{n}(n)>0$ such that 
		\begin{align}\label{measure of balls}
			\left\lvert B_{r} \right\rvert = c_n r^Q. 
		\end{align}
		Note that, there exist constants  $A_1, A_2>0$ such that the Euclidean distance and CC distance satisfy the following relation
		\begin{align}\label{euclidean and cc distance}
			A_1 \left \lvert x-y \right \rvert \leq d_{\text{CC}} \left(x, y\right) \leq A_2 \left \lvert x-y \right \rvert^{\frac{1}{2}},
		\end{align}
		for any $x, y \in \mathbb{H}_n$. 
		\item For us, $\Omega \subset \mathbb{H}_{n}$ would always denote an open, bounded subset, where by bounded we mean $\sup\limits_{x, y \in \Omega} d_{\text{CC}}(x,y) < \infty.$

		\item 	By a slight abuse of notation, we would use the notation $C^{\,0,\alpha}(\Omega)$ to denote the Folland-Stein classes (cf. \cite{Folland-Stein_book}), which are H\"{o}lder spaces with respect to the CC-metric, defined as   
		\begin{equation}\label{def:holderspace}
			C^{\,0,\alpha}(\Omega) =
			\left\lbrace u\in L^{\infty}(\Omega): \sup_{x, y\in \Omega, x\neq y} \frac{|u(x)-u(y)|}{d_{\text{CC}}(x,y)^\alpha}< +\infty\right\rbrace 
		\end{equation} 
		for $0<\alpha \leq 1$, where $\Omega \subset \mathbb{H}_{n}.$
		These are Banach spaces with the norm 
		\begin{equation*}
			\|u\|_{C^{\,0,\alpha}(\Omega)}
			=  \|u\|_{L^\infty(\Omega)}+ \left[ u\right]_{C^{0, \alpha}\left( \Omega\right)},
		\end{equation*}
		where the Folland-Stein seminorm is defined as 
		\begin{align*}
			\left[ u\right]_{C^{0, \alpha}\left( \Omega\right)} := \sup_{x,y\in\Omega, x\neq y} 
			\frac{|u(x)-u(y)|}{d_{\text{CC}}(x,y)^\alpha}. 
		\end{align*}
		
		\item For any uniformly continuous function $a:\Omega \rightarrow \mathbb{R},$ the modulus of continuity of $a$, denoted $\omega_{a, \Omega}$ is a concave increasing function $\omega_{a, \Omega}: [0, \infty) \rightarrow [0, \infty),$ defined by 
		\begin{align*}
			\omega_{a, \Omega}\left( r \right):=  \sup\limits_{\substack{x,y \in \Omega,\\ d_{\text{CC}}(x,y) \leq r }} \left\lvert a\left(x\right) - a\left( y\right)\right\rvert \qquad \text{ for all  } r>0,
		\end{align*}
		and we set $\omega_{a, \Omega}\left(0\right)= 0$ by convention. We would often denote the modulus of continuity by $\omega_{a}$ when the domain is clear and also just $\omega$ when the function is also clear from the context.  It is easy to see that by definition, $a \in C^{0, \alpha} \left( \Omega\right)$ if and only if $\omega_{a, \Omega} \left(r\right) \leq C r^{\alpha}$ for some constant $C>0$ and the smallest such constant is equal to the seminorm $\left[ a\right]_{C^{0, \alpha}\left( \Omega\right)}.$
	\end{itemize}
	\section{Preliminaries}\label{prelim}
	\subsection{Function spaces}
	\subsubsection{Morrey-Campanato spaces}
	Now, assume $\Omega \subset \mathbb{H}_{n}$ is any open subset such that there exists a constant $A = A\left( \Omega\right)>0$ such that for any $x_{0} \in \Omega$ and any $r < \operatorname{diam} \Omega,$ we have  
	\begin{align*}
		\left\lvert B_{r}\left(x_{0}\right) \cap \Omega \right\rvert \geq Ar^{Q}. 
	\end{align*} 
	This property is clearly satisfied if $\Omega$ itself is a metric ball. 
	\begin{definition}[Morrey and Campanato spaces]
		For $1 \leq p <  \infty$ and $\lambda \geq 0,$  $\mathrm{L}^{p,\lambda}\left(\Omega\right) $ stands for the Morrey space of all $u \in L^{p}\left(\Omega \right)$ such that 
		$$ \lVert u \rVert_{\mathrm{L}^{p,\lambda}\left(\Omega\right)}^{p} := \sup_{\substack{ x_{0} \in \Omega,\\ \rho >0 }} 
		\rho^{-\lambda} \int_{B_{\rho}(x_{0}) \cap \Omega} \lvert u \rvert^{p} < \infty, $$ endowed with the norm 
		$ \lVert u \rVert_{\mathrm{L}^{p,\lambda}}$ and $\mathcal{L}^{p,\lambda}\left(\Omega\right) $ denotes the Campanato space of all $u \in L^{p}\left(\Omega\right)$ such that 
		$$ [u ]_{\mathcal{L}^{p,\lambda}\left(\Omega\right)}^{p} := \sup_{\substack{ x_{0} \in \Omega,\\  \rho >0 }} 
		\rho^{-\lambda} \int_{B_{\rho}(x_{0}) \cap \Omega} \lvert u  - (u)_{ \rho , x_{0},\Omega}\rvert^{p} < \infty, $$ endowed with the norm 
		$ \lVert u \rVert_{\mathcal{L}^{p,\lambda}} := \lVert  u \rVert_{L^{p}} +  
		[u ]_{\mathcal{L}^{p,\lambda}}.$ Here 
		\begin{align*}
			(u)_{ \rho , x_{0},\Omega} = \frac{1}{ \left\lvert \left( B_{\rho}(x_{0}) \cap \Omega \right) \right\rvert}\int_{B_{\rho}(x_{0}) \cap \Omega} u  = \fint_{B_{\rho}(x_{0}) \cap \Omega} u .
		\end{align*}
	\end{definition}
	We define $\rm{BMO}\left(\Omega\right):= \mathcal{L}^{1,Q}\left(\Omega\right).$ These definitions are extended componentwise for vector-valued functions. We record the following facts, which are proved exactly in the same manner as their Euclidean counterparts in view of \eqref{measure of balls}. 
	\begin{proposition}\label{Morrey-Campanato theorem}
		For $0 \leq \lambda < Q,$ we have 
		\begin{align*}
			\mathrm{L}^{p,\lambda}\left(\Omega\right) \simeq	\mathcal{L}^{p,\lambda}\left(\Omega\right) \qquad \text{ with equivalent norms}.\end{align*}
		For all $1 \leq p < \infty,$ we have 
		\begin{align*}
			\mathcal{L}^{p,Q}\left(\Omega\right) \simeq \rm{BMO}\left(\Omega\right) \qquad \text{ with equivalent seminorms}.
		\end{align*}For $Q< \lambda \leq Q+p, $ we have 
		\begin{align*}
			C^{0, \frac{\lambda -Q}{p}}\left(\Omega\right) \simeq	\mathcal{L}^{p,\lambda}\left(\Omega\right) \qquad \text{ with equivalent seminorms}.\end{align*}
	\end{proposition}
	For details on classical Morrey and Campanato spaces, we refer to 
	\cite{giaquinta-martinazzi-regularity} and for the sub-elliptic setting we refer to \cite{Capogna_Danielli_book}.
	\subsubsection{Horizontal Sobolev spaces and Sobolev inequalities}
	Let $\Omega \subset \mathbb{H}_{n}$ be open. For $1 \leq p \leq \infty,$ the usual $L^{p}$ spaces $L^{p}\left(\Omega\right)$ is defined in the usual manner and extended componentwise to vector-valued functions. 
	\begin{definition}[Horizontal Sobolev spaces]
		For $ 1\leq p < \infty$, the \textit{Horizontal Sobolev space} $HW^{1,p}(\Omega)$ consists
		of functions $u\in L^p(\Omega)$ such that the distributional horizontal gradient $\X u$ is in $L^p(\Omega\,,\mathbb{R}^{2n})$.
		$HW^{1,p}(\Omega)$ is a Banach space with respect to the norm
		\begin{equation}\label{eq:sob norm}
			\| u\|_{HW^{1,p}(\Omega)}= \ \| u\|_{L^p(\Omega)}+\| \X u\|_{L^p(\Omega,\mathbb{R}^{2n})}.
		\end{equation}
		We define $HW^{1,p}_{\text{loc}}(\Omega)$ as its local variant and 
		$HW^{1,p}_0(\Omega)$ as the closure of $C^\infty_0(\Omega)$ in 
		$HW^{1,p}(\Omega)$ with respect to the norm in \eqref{eq:sob norm}.
	\end{definition}
	The definition is extended componentwise to vector valued functions as well. We would need the following Poincar\'{e} and Poincar\'{e}-Sobolev type inequalities for Heisenberg groups, which are written in a scale-invariant form. 
	\begin{proposition}[Poincar\'{e} inequality with means]\label{poincarewithmeans}
		Let $B_{R} \subset \mathbb{H}_{n}$ be any ball of radius $R>0.$ Let $1 \leq s <  \infty.$ Then  there exists a constant $c >0,$ depending only on $n$ and $s$ such that for any $u \in HW^{1,s}\left( B_{R} \right),$ we have
		\begin{equation}\label{poincareineqwithmeans}
			\left( \int_{B_{R}} \lvert u - \left( u\right)_{B_{R}} \rvert^{s} \right)^{\frac{1}{s}} \leq c R \left( \int_{B_{R}} \lvert \X u \rvert^{s} \right)^{\frac{1}{s}}. 
		\end{equation}
	\end{proposition} 
	\begin{proposition}[Poincar\'{e}-Sobolev inequality]\label{poincaresobolev}
		Let $B_{R} \subset \mathbb{H}_{n}$ be any ball of radius $R>0.$ Let $1 \leq s <  Q.$ Then there exists a constant $c >0,$ depending only on $n$ and $s$ such that for any $u \in HW^{1,s}_{0}\left( B_{R} \right),$ we have
		\begin{equation}\label{poincaresobolevineq}
			\left( \fint_{B_{R}} \lvert u \rvert^{\frac{Qs}{Q-s}} \right)^{\frac{Q-s}{Qs}} \leq c R \left( \fint_{B_{R}} \lvert \X u \rvert^{s} \right)^{\frac{1}{s}}. 
		\end{equation}
	\end{proposition}
	\begin{proposition}[Poincar\'{e}-Sobolev inequality with means]\label{poincaresobolevwithmeans}
		Let $B_{R} \subset \mathbb{H}_{n}$ be any ball of radius $R>0.$ Let $1 \leq s <  Q.$ Then there exists a constant $c >0,$ depending only on $n$ and $s$ such that for any $u \in HW^{1,s}\left( B_{R} \right),$ we have
		\begin{equation}\label{poincaresobolevineqwithmeans}
			\left( \fint_{B_{R}} \lvert u - \left( u\right)_{B_{R}} \rvert^{\frac{Qs}{Q-s}} \right)^{\frac{Q-s}{Qs}} \leq c R \left( \fint_{B_{R}} \lvert \X u \rvert^{s} \right)^{\frac{1}{s}}. 
		\end{equation}
	\end{proposition}
	The scale invariant forms in \eqref{poincareineqwithmeans}, \eqref{poincaresobolevineq}, \eqref{poincaresobolevineqwithmeans} can be derived by using their corresponding versions when $R=1.$ For the case $R=1$, we can use the Poincar\'e inequality established in \cite{Jerison_Poincare} and the global density result in \cite[Theorem 1.7]{Garofalo-Nhieu-Global-Approximation} to derive \eqref{poincareineqwithmeans}. Additionally, in this case, by using \cite[Theorem 2.1]{lu-poincare-sobolev} and \cite[Theorem 1.7]{Garofalo-Nhieu-Global-Approximation} we can obtain \eqref{poincaresobolevineq} and \eqref{poincaresobolevineqwithmeans}. When $s > Q,$ we also have the following analogue of Morrey's theorem. 
	\begin{proposition}\label{Morrey's theorem}
		Let $\Omega \subset \mathbb{H}_{n}$ be open and let $u \in HW^{1,s}_{\text{loc}}\left( \Omega\right)$ with $s > Q.$ Then $u \in C^{0, \frac{s-Q}{s}}\left( \Omega\right)$ and for any metric ball $B_{2R} \subset \Omega,$ we have the estimate 
		\begin{align*}
			\left\lVert u \right\rVert_{C^{0,\frac{s-Q}{s} }\left( \overline{B_{R/2}}\right)} \leq C(R, Q, s) \left\lVert u \right\rVert_{HW^{1,s}\left(\Omega\right)}. 
		\end{align*} 
	\end{proposition}
	\begin{proof}
		By the Poincar\'{e} inequality \eqref{poincarewithmeans}, we have 
		\begin{align*}
			\int_{B_{\rho}} \left\lvert u - (u)_{\rho} \right\rvert^{s} \leq c\rho^{s} \int_{B_{\rho}} \left\lvert \X u \right\rvert^{s} \leq c\rho^{s} \int_{B_{R}} \left\lvert \X u \right\rvert^{s},
		\end{align*}
		for any $0 < \rho < R,$ whenever the ball  $B_{2R} \subset \Omega.$ This estimate implies, in the standard way,  that $u \in \mathcal{L}^{s,s}_{\text{loc}}\left(\Omega\right)$. Since $s >Q,$ the second conclusion of Proposition \ref{Morrey-Campanato theorem} implies the result. 
	\end{proof}
	\begin{remark}\label{p uniform constants Sobolev inequalities}
		The constants in each of the propositions above in this subsection depends on $s.$ However, each of them can be proved in a constructive manner (i.e. without using any contradiction-compactness argument) and thus the constants are explicitly quantifiable. Thus, if $1 < \gamma_{1} \leq s \leq \gamma_{2}< \infty,$ the constants can all be chosen (by replacing them with larger or smaller constants if necessary) to depend only $\gamma_{1}, \gamma_{2}$ such that the estimates are uniform in $s \in [\gamma_{1}, \gamma_{2}].$ This is extremely crucial for us in this variable exponent setting.
	\end{remark}
	\subsubsection{Variable exponent spaces on \texorpdfstring{$\mathbb{H}_{n}$}{Hn}}\label{variable exponent spaces}
	Let $\Omega \subset \mathbb{H}_{n}$ be open and bounded and let $1 < \gamma_{1} \leq \gamma_{2} < \infty.$ 
	\begin{definition}
		For any uniformly continuous (and thus Lebesgue measurable) function $p:\Omega \rightarrow [\gamma_{1}, \gamma_{2}], $ we say 
		\begin{enumerate}[(i)]
			\item $p$ is \emph{log-H\"{o}lder continuous},  denoted by $p \in \mathcal{P}^{\log}_{[\gamma_{1}, \gamma_{2}]}\left(\Omega \right),$ if  \begin{align}\label{log holder def}
				\lim\limits_{R \rightarrow 0} \omega_p\left(R\right)\log \left(\frac{1}{R}\right) < +\infty.  
			\end{align}
			\item $p$ is \emph{vanishing log-H\"{o}lder continuous} if 
			\begin{align}\label{log holder vanishing}
				\lim\limits_{R \rightarrow 0} \omega_p\left(R\right)\log \left(\frac{1}{R}\right) = 0 . 
			\end{align}
			\item  $p$ is \emph{log-Dini continuous} if there exists $R>0$ such that we have 
			\begin{align}\label{log Dini}
				\int_{0}^{R}\omega_p\left(\rho\right)\log \left(\frac{1}{\rho}\right)\ \frac{\mathrm{d}\rho}{\rho} < \infty.  
			\end{align}	
			\item $p$ is assumed to be H\"{o}lder continuous (with respect to the CC-matric) with exponent $\alpha \in (0,1),$ if there exists a constant $M>0$ such that 
			\begin{align}\label{holder exponent}
				\omega_p\left(R\right) \leq M R^{\alpha} \qquad \text{ for all } 0 < R \leq \operatorname{diam} \Omega.  
			\end{align} 
		\end{enumerate} 
	\end{definition}
	\begin{remark}
		Since $\Omega$ is bounded and $\gamma_{2}<\infty,$ the condition \eqref{log holder def} is equivalent to the usual definition of log-H\"{o}lder continuity (cf. \cite{Diening_et_al_variable_exponent} for the Euclidean case). Similarly, due to \eqref{euclidean and cc distance}, the conditions \eqref{log holder vanishing} and \eqref{log Dini} are equivalent to the  usual definitions of vanishing of log-H\"older continuity and log-Dini continuity.
	\end{remark} 
	\begin{remark}
		Note that the conditions \eqref{log holder def}, \eqref{log holder vanishing}, \eqref{log Dini} and \eqref{holder exponent} are progressively stronger requirements. 
	\end{remark}
	\begin{definition}
		Given any $p \in \mathcal{P}^{\log}_{[\gamma_{1}, \gamma_{2}]}\left(\Omega \right),$ we define the variable exponent Lebesgue space $L^{p(\cdot)}\left( \Omega\right)$ with exponent function $p$ as 
		\begin{align*}
			L^{p(\cdot)}\left( \Omega\right):= \left\lbrace u:\Omega \rightarrow \mathbb{R} \text{ measurable}: \int_{\Omega} \left\lvert u\left(x\right)\right\rvert^{p\left(x\right)}\ \mathrm{d}x < \infty\right\rbrace. 
		\end{align*}
		We equip this space with the Luxemburg norm 
		\begin{align*}
			\left\lVert u \right\rVert_{	L^{p(\cdot)}\left( \Omega\right)} := \inf \left\lbrace \lambda >0: \int_{\Omega} \left( \frac{\left\lvert u\left(x\right)\right\rvert}{\lambda}\right)^{p\left(x\right)}\ \mathrm{d}x \leq 1\right\rbrace.  
		\end{align*}
	\end{definition}
	This definition is exactly the same as the Euclidean one and exactly in the same manner and it follows form \cite{Diening_et_al_variable_exponent}) that for $p \in \mathcal{P}^{\log}_{[\gamma_{1}, \gamma_{2}]}\left(\Omega \right),$ the space $L^{p(\cdot)}\left( \Omega\right)$, equipped with the above norm is a reflexive, separable Banach space. 
	The definition extends in the obvious manner to vector-valued functions by requiring each component to be in $L^{p(\cdot)}\left( \Omega\right).$
	When $U \subset \mathbb{R}^{n},$ the usual variable exponent Sobolev spaces are defined as 
	\begin{align*}
		W^{1, p(\cdot)}\left( U\right):= \left\lbrace u \in L^{p(\cdot)}\left( U\right): \left\lVert u \right\rVert_{L^{p(\cdot)}\left( U\right)} + \left\lVert \nabla u \right\rVert_{L^{p(\cdot)}\left(U; \mathbb{R}^{n}\right)}  < \infty \right\rbrace,
	\end{align*}
	where $u$ is implicitly assumed to be weakly differentiable in $U$ and $\nabla u$ stands for the weak gradient of $u.$
	In the case of Heisenberg groups, we define the horizontal variable exponent Sobolev space as 
	\begin{align*}
		HW^{1, p(\cdot)}\left( \Omega \right):= \left\lbrace u \in L^{p(\cdot)}\left( \Omega\right): \left\lVert u \right\rVert_{L^{p(\cdot)}\left( \Omega\right)} + \left\lVert \X u \right\rVert_{L^{p(\cdot)}\left(\Omega; \mathbb{R}^{2n}\right)}  < \infty \right\rbrace,
	\end{align*}
	where $\X u$ denotes the weak horizontal gradient. This space is equipped with the norm 
	\begin{align}\label{HW1px norm}
		\left\lVert u \right\rVert_{HW^{1, p(\cdot)}\left( \Omega \right)}:= \left\lVert u \right\rVert_{L^{p(\cdot)}\left( \Omega\right)} + \left\lVert \X u \right\rVert_{L^{p(\cdot)}\left(\Omega; \mathbb{R}^{2n}\right)}.
	\end{align}The space $	HW^{1, p(\cdot)}_{0}\left( \Omega \right)$ denotes the closure of $C_{c}^{\infty}\left( \Omega\right)$ under the norm \eqref{HW1px norm}.

	\subsubsection{Lorentz-Zygmund spaces}
	\paragraph*{\texorpdfstring{$L\log^{\beta}L$}{LlogbetaL} spaces}
	For $\beta \geq 0 ,$ the Orlicz space $L\log^{\beta}L \left( \Omega; \mathbb{R}^{m}\right)$ is defined as 
	\begin{align*}
		L\log^{\beta}L \left( \Omega; \mathbb{R}^{m}\right): = \left\lbrace f \in L^{1}\left( \Omega; \mathbb{R}^{m}\right): \int_{\Omega} \left\lvert f \right\rvert \log^{\beta}\left( e + \left\lvert f \right\rvert \right)\ \mathrm{d}x < \infty \right\rbrace.  
	\end{align*}
	This is a Banach space equipped with the Luxemburg norm
	\begin{align*}
		\left\lVert f \right\rVert_{L\log^{\beta}L \left( \Omega\right)}:= \inf \left\lbrace \lambda >0: \fint_{\Omega} \left\lvert \frac{f}{\lambda} \right\rvert \log^{\beta}\left( e + \left\lvert \frac{f}{\lambda} \right\rvert \right)\ \mathrm{d}x \leq 1\right\rbrace.  
	\end{align*}
	We also define 
	\begin{align*}
		\left[ f\right]_{L\log^{\beta}L \left( \Omega\right)}:= \fint_{\Omega} \left\lvert f \right\rvert \log^{\beta}\left( e + \frac{\left\lvert f \right\rvert}{\fint_{\Omega}\left\lvert f \right\rvert\ \mathrm{d}x} \right)\ \mathrm{d}x. 
	\end{align*}
	Iwaniec (\cite{IwaniecPharmonic}, \cite{Iwaniec_et_al_MappingsBMOboundeddistortion}, \cite{IwaniecVerde_LlogL}) showed that there exists a constant $c = c\left(\beta\right) \geq 1,$ independent of $\Omega$ and $f$, such that for all $ f \in L\log^{\beta}L \left( \Omega; \mathbb{R}^{m}\right),$ we have 
	\begin{align*}
		\frac{1}{c}\left\lVert f \right\rVert_{L\log^{\beta}L \left( \Omega\right)} \leq 	\left[ f\right]_{L\log^{\beta}L \left( \Omega\right)} \leq c\left\lVert f \right\rVert_{L\log^{\beta}L \left( \Omega\right)} 
	\end{align*}
	For any $p >1,$ $L^{p}\left(\Omega; \mathbb{R}^{m}\right)$ embeds into $L\log^{\beta}L \left( \Omega; \mathbb{R}^{m}\right)$ and we have the following inequality (cf. estimate $(28)$ in \cite{Acerbi_Mingione_CZforpxLaplacian}) 
	\begin{align}\label{LlogbetaL estimate}
		\fint_{\Omega} \left\lvert f \right\rvert \log^{\beta}\left( e + \frac{\left\lvert f \right\rvert}{\fint_{\Omega}\left\lvert f \right\rvert\ \mathrm{d}x} \right)\ \mathrm{d}x \leq c\left(p, \beta\right) \left( \fint_{\Omega} \left\lvert f \right\rvert^{p}\ \mathrm{d}x\right)^{\frac{1}{p}}. 
	\end{align}
	In the same spirit, we also have the following lemma (see Lemma 2.1 in \cite{Baroni_Coscia}). 
	\begin{lemma}\label{Llog beta L lemma}
		Let $\zeta >1$, $\delta, \beta, \tau \geq 0$ and $f$ be a nonnegative function in $L^{\zeta}\left(B_{R}\right)$ for some ball $B_{R}$ with $R \leq 1/e.$ Then there exists a constant $c= c(n, \beta, \delta, \tau, \zeta) >0$ such that 
		\begin{align*}
			\fint_{B_{R}}  f \log^{\beta}\left( e + f^{\delta} \right)\ \mathrm{d}x \leq c \left( 1 + R^{\tau} \left\lVert f \right\rVert_{L^{1}\left(B_{R}\right)}\right)^{\beta} \log^{\beta}\left(\frac{1}{R}\right)\left( \fint_{B_{R}} \left\lvert f \right\rvert^{\zeta}\ \mathrm{d}x\right)^{\frac{1}{\zeta}}. 
		\end{align*}
	\end{lemma}
	\paragraph*{\texorpdfstring{$L^{p,q}\log L$}{LpqlogL} spaces}
	A good reference for these materials are Chapter 9, \cite{Pick_Kufner_Fucik_Function_spaces}. 	Let $\Omega \subset \mathbb{H}_{n}$ be a measurable set and let $g: \Omega \rightarrow \mathbb{R}$ be a measurable function. The distribution function of $f$, denoted $\mu_{f}: [0, +\infty) \rightarrow [0, +\infty),$ and the nonincreasing rearrangement of $f$, denoted $f^{\ast}: [0, +\infty) \rightarrow [0, +\infty],$ are defined the usual way with 
	\begin{align*}
		\mu_{f}\left(t\right) = \left\lvert \left\lbrace x \in \Omega: \left\lvert f\left(x\right)\right\rvert > t\right\rbrace \right\rvert \quad \text{ and } \quad f^{\ast}\left(r\right)= \inf \left\lbrace t \geq 0: 	\mu_{f}\left(t\right) \leq r\right\rbrace.    
	\end{align*}
	\begin{definition} Let $1\leq s < \infty,$ $\mu \in (0, \infty]$ and $\beta \in \mathbb{R}.$ Given a measurable function $f:\Omega \rightarrow \mathbb{R},$ we define its  $L^{(s, \mu)}\log^{\beta}L \left(\Omega\right)$ quasinorm as 
		\begin{align*}
			\left\vvvert f \right\vvvert_{L^{(s, \mu)}\log^{\beta}L \left(\Omega\right)} := \left\lbrace \begin{aligned}
				&\left( \int_{0}^{\left\lvert \Omega \right\rvert} \left[ \rho^{\frac{1}{s}}\left( 1+ \left\lvert \log \rho \right\rvert\right)^{\beta}f^{\ast}\left(\rho\right)\right]^{\mu}\ \frac{\mathrm{d}\rho}{\rho}\right)^{\frac{1}{\mu}} &&\text{ if } \mu < \infty, \\
				&\sup\limits_{\rho \in \left(0, \left\lvert \Omega \right\rvert\right)}  \rho^{\frac{1}{s}}\left( 1+ \left\lvert \log \rho \right\rvert\right)^{\beta}f^{\ast}\left(\rho\right) &&\text{ if } \mu = \infty. 
			\end{aligned}\right. 
		\end{align*}
		We say $f \in L^{(s, \mu)}\log^{\beta}L \left(\Omega\right)$ if $	\left\vvvert f \right\vvvert_{L^{(s, \mu)}\log^{\beta}L \left(\Omega\right)}  < \infty.$
	\end{definition}
	The quasinorm is defined above is indeed, in general, only a quasinorm and not a norm. However, we have the following. 
	\begin{proposition}
		When $s>1$ and $\mu < \infty,$ the $L^{(s, \mu)}\log^{\beta}L \left(\Omega\right)$ quasinorm is equivalent to a norm
		\begin{align*}
			\left\lVert f \right\rVert_{L^{(s, \mu)}\log^{\beta}L \left(\Omega\right)} := \left( \int_{0}^{\left\lvert \Omega \right\rvert} \left[ \rho^{\frac{1}{s}}\left( 1+ \left\lvert \log \rho \right\rvert\right)^{\beta}f^{\ast\ast}\left(\rho\right)\right]^{\mu}\ \frac{\mathrm{d}\rho}{\rho}\right)^{\frac{1}{\mu}},
		\end{align*}
		where $f^{\ast\ast}$ is essentially the maximal function of $f^{\ast},$ defined by 
		\begin{align*}
			f^{\ast\ast} \left( r\right) = \frac{1}{r} \int_{0}^{r} f^{\ast}\left(t\right)\ \mathrm{d}t \qquad \text{ for } r \in (0, \infty). 
		\end{align*} 
		Moreover, the norm is absolutely continuous, i.e. for any $f \in L^{(s, \mu)}\log^{\beta}L \left(\Omega\right)$ and $E \subset \Omega$ measurable,  we have 
		\begin{align}\label{abs cont Lorentz zygmund norm}
			\left\lVert f \right\rVert_{L^{(s, \mu)}\log^{\beta}L \left(E\right)} \rightarrow 0 \qquad \text{ as } \qquad \left\lvert E \right\rvert \rightarrow 0. 
		\end{align}
	\end{proposition}
	Now for any $\beta \in \mathbb{R}$ and any $g \in L^{(Q, 1)}\log^{\beta}L \left(\Omega\right),$ we define 
	\begin{align}\label{def of Lorentz Zygmund sum}
		S_{(x_0, r, \theta, \tau, \beta )}\left( g \right): = \sum\limits_{j=0}^{\infty} r_{j} \log^{\beta} \left(\frac{1}{r_{j}}\right)\left(  \fint_{B_{r_{j}}\left(x_{0}\right)}\left\lvert   g \right\rvert^{\theta} \right)^{\frac{1}{\theta}}, 	
	\end{align}
	where $0 < r < 1/2e ,$ $\tau \in (0,1),$ $r_{j}:= \tau^{j}r$ and $1 < \theta < Q.$ 
	\begin{proposition}
		For any $g \in L^{(Q, 1)}\log^{\beta}L \left(\Omega\right),$ we have 
		\begin{align}\label{estimate on infinite sum Lorentz Zygmund}
			S_{(x_0, r, \theta, \tau, \beta )}\left( g \right) \leq c 	\left\lVert g \right\rVert_{L^{(Q, 1)}\log^{\beta}L \left(B_{2r}\left(x_{0}\right)\right)} . 
		\end{align}
		Consequently, $S_{(x_0, r, \theta, \tau, \beta )}\left( g \right) \rightarrow 0$ as $r \rightarrow 0.$
	\end{proposition}
	For functions with derivatives in Lorentz-Zygmund spaces, we record the following result, which can be proved by adapting the Euclidean case proved in \cite[Theorem 1.3]{Cianchi_Pick_embeddingContinuous} (with $Q$ playing the role of $n$ here), by using \cite[Lemma 4.2]{Francu_polyaSzego_Carnot}, combined with \cite[Theorem 1.18]{Garofalo_Nhieu_isoperimatricMinimalsurface}.  
	\begin{proposition}\label{Lorenz-Zygmund-Sobolev_implies_logholder}
		For any $p \in L^{1}\left(\Omega\right)$ with $\X p \in L^{(Q, 1)}\log^{\beta}L \left(\Omega\right),$ we have 
		\begin{align}\label{Sobolev estimate Lorentz Zygmund}
			\lim\limits_{R \rightarrow 0} \omega_p\left(R\right)\log \left(\frac{1}{R}\right) \leq c\left(Q\right) \left[ \left\lVert p\right\rVert_{L^{1}\left( \Omega\right)} + \left\lVert \X p \right\rVert_{L^{\left(Q,1\right)}\log L\left( \Omega; \mathbb{R}^{2n}\right)}\right]. 
		\end{align}
		In particular, any such function $p$ is log-H\"{o}lder continuous, with the log-H\"{o}lder constant $L$ (as defined in \eqref{log holder condition thm}) being bounded by a constant depending only on $Q$ and $\left\lVert p\right\rVert_{L^{1}\left( \Omega\right)} + \left\lVert \X p \right\rVert_{L^{\left(Q,1\right)}\log L\left( \Omega; \mathbb{R}^{2n}\right)}.$ 
	\end{proposition}
	\subsection{Algebraic lemmas}
	In this subsection, we shall record some elementary algebraic facts. As is standard in the literature, given any real number $1 < p < \infty,$ we shall use the following auxiliary mapping $V_{p}: \mathbb{R}^{2n} \rightarrow \mathbb{R}^{2n}$ 
	defined by 
	\begin{equation}\label{definition V}
		V_{p}(z) : = \left\lvert z \right\rvert^{\frac{p-2}{2}} z , 
	\end{equation}
	which is a locally Lipschitz bijection from $\mathbb{R}^{2n}$ into itself. 
	We summarize the relevant properties of the map in the following. 
	\begin{lemma}\label{prop of V}
		For any $p >1$ there exists a constant $c_{V} \equiv c_{V}(n,p) > 0$ such that 
		\begin{equation}\label{constant cv}
			\frac{\left\lvert z_{1} - z_{2}\right\rvert}{c_{V}} \leq  \frac{\left\lvert V_{p}(z_{1}) - V_{p}(z_{2})\right\rvert}{\left( \left\lvert z_{1} \right\rvert + 
				\left\lvert z_{2}\right\rvert \right)^{\frac{p-2}{2}}} \leq c_{V}\left\lvert z_{1} - z_{2}\right\rvert,
		\end{equation}
		for any $z_{1}, z_{2} \in \mathbb{R}^{2n},$ not both zero. This implies the classical monotonicity estimate 
		\begin{equation}\label{monotonicity}
			\left( \left\lvert z_{1} \right\rvert + \left\lvert z_{2}\right\rvert \right)^{p-2} \left\lvert z_{1} - z_{2}\right\rvert^{2} \leq c(n,p) 
			\left\langle  \left\lvert z_{1} \right\rvert^{p-2} z_{1} - \left\lvert z_{2} \right\rvert^{p-2} z_{2}, z_{1} - z_{2} \right\rangle ,  
		\end{equation}
		with a constant $c(n,p) > 0$ for all $p >1$ and all $z_{1}, z_{2} \in \mathbb{R}^{2n}.$
		Moreover, if $1 < p \leq 2,$ there exists a constant $c \equiv c(n,p) > 0$ such that for any  $z_{1}, z_{2} \in \mathbb{R}^{2n},$
		\begin{equation}\label{v estimate p less 2}
			\left\lvert z_{1} - z_{2}\right\rvert \leq c \left\lvert V(z_{1}) - V(z_{2})\right\rvert^{\frac{2}{p}} + c \left\lvert V(z_{1}) - V(z_{2})\right\rvert 
			\left\lvert z_{2}\right\rvert^{\frac{2-p}{2}}.
		\end{equation}
	\end{lemma}
	The estimates \eqref{constant cv} and \eqref{monotonicity} are classical (cf. Lemma 2.1, \cite{hamburgerregularity}). The estimate \eqref{v estimate p less 2} follows from this 
	(cf. Lemma 2, \cite{KuusiMingione_nonlinearStein}).
	\begin{lemma}
		Let $m \in \mathbb{N}.$ For any $1 < p < \infty,$ there exists a constant $c=c\left(m, p\right) >0$ such that for any $\xi, \eta \in \mathbb{R}^{m},$ we have 
		\begin{align}\label{comparison algebraic ineq}
			\left\lvert \xi \right\rvert^{p} \leq c 	\left\lvert \eta \right\rvert^{p} + c 	\left( 	\left\lvert \xi \right\rvert^{2} + 	\left\lvert \eta \right\rvert^{2} \right)^{\frac{\left(p-2\right)}{2}}\left\lvert \xi - \eta \right\rvert^{2}.
		\end{align}
	\end{lemma}
	\begin{lemma}
		Let $1 < \gamma_{1} \leq \gamma_{2} < \infty$ be any two real numbers and let $\beta$ satisfy 
		\begin{align*}
			\frac{\gamma_{2}}{\gamma_{2}-1} \leq \beta \leq \frac{\gamma_{1}}{\gamma_{1} -1}. 
		\end{align*}
		Then the following holds. 
		\begin{enumerate}[(i)]
			\item For any $0 < \sigma < 1,$ we have the inequality 
			\begin{align}\label{LpintoLlogL}
				\left( e + t \right)\log^{\beta} \left( e + t\right) \leq c\left(\gamma_{1}, \gamma_{2} \right) \sigma^{-\beta}\left( e + t \right)^{1 + \sigma/4} \qquad \text{ for all } t \geq 0. 
			\end{align}
			\item We have the following inequality 
			\begin{align}\label{less than e}
				t^{\frac{\beta}{\beta -1}}\left\lvert \log \ t\right\rvert^{\beta}  \leq c \left(\gamma_{1}, \gamma_{2}\right) \qquad \text{ for all } 0< t \leq e. 
			\end{align}
			\item For any two nonnegative real numbers $a, b \geq 0,$ we have 
			\begin{align}\label{concavity of log}
				\log^{\beta}\left( e + ab \right) \leq 2^{\left(\frac{\gamma_{1}}{\gamma_{1}-1} -1\right)}\left[ 	\log^{\beta}\left( e + a \right) + 	\log^{\beta}\left( e + b \right)\right]. 
			\end{align}
		\end{enumerate}
	\end{lemma}
	\begin{proof}
		$(i)$ and $(ii)$ follow from the elementary inequality valid for every $\beta >0,$  
		\begin{align}\label{growth of log}
			t \log^{\beta}t \leq \left[ \frac{\beta}{e\left( p -1\right)}\right]^{\beta}t^{p} \qquad \text{ for all } t \geq 1, 
		\end{align}
		valid for any $p>1.$ $(iii)$ follows from the concavity of the logarithm. 
	\end{proof}
	\subsection{Subelliptic \texorpdfstring{$p$}{p}-harmonic functions}
	We record here some known results about subelliptic $p$-Laplacian. 
	\begin{proposition}\label{existence of minimizers for constant exponent}
		Let $\Omega \subset \mathbb{H}^{n}$ be a smooth, bounded domain. Let $a:\Omega \rightarrow [\gamma, L]$ be measurable, where $ 0 < \gamma < L < \infty.$ Let $1 < p < \infty,$ $F\in L^{p'}\left( \Omega; \mathbb{R}^{2n}\right)$ and $u_{0} \in HW^{1,p}\left( \Omega\right).$ Then the following minimization problem 
		\begin{align}
			m:= \inf \left\lbrace \int_{\Omega} \left[ a\left( x\right)\left\lvert \X u \right\rvert^{p} - \left\langle F, \X u \right\rangle\right]: u \in u_{0} + HW^{1,p}_{0}\left( \Omega\right)\right\rbrace 
		\end{align}
		admits a unique minimizer. Moreover, the minimizer $\bar{u} \in HW^{1,p}\left( \Omega\right)$ is the unique weak solution to the system 
		\begin{align*}
			\left\lbrace \begin{aligned}
				\operatorname{div}_{\mathbb{H}}\left( a\left(x\right)\left\lvert \X u\right\rvert^{p-2}\X u\right) &= \operatorname{div}_{\mathbb{H}}F &&\text{ in } \Omega, \\
				u &=u_{0} &&\text{ on } \origpartial\Omega. 
			\end{aligned}\right. 
		\end{align*} 
	\end{proposition} 
	This can be proved using standard variational methods. For constant coefficient case homogeneous case, i.e. when $a\left(x\right) \equiv a$ for some constant $a>0$ and $F \equiv 0,$ we have the following fundamental estimate due to \cite{zhong2018regularityvariationalproblemsheisenberg} and \cite{Mukherjee_Zhong}. 
	\begin{theorem}\label{Uhlenbeck estimate}
		Let $\Omega \subset \mathbb{H}^{n}$ be open. Let $a>0$ and $1<p < \infty.$ Let $u \in HW^{1,p}_{\text{loc}}\left( \Omega\right)$ be a weak solution to 
		\begin{align*}
			\operatorname{div}_{\mathbb{H}}\left( a \left\lvert \X u\right\rvert^{p-2}\X u\right) &=0 &&\text{ in } \Omega.
		\end{align*}
		Then $\X u$ is locally H\"{o}lder continuous in $\Omega$ and for any ball $B_{R} \subset \Omega,$ we have the estimates 
		\begin{align}\label{sup estimate constant homogeneous}
			\sup\limits_{B_{R/2}} \left\lvert \X u \right\rvert &\leq c_{1}  \fint_{B_{R}} \left\lvert \X u \right\rvert \\ \intertext{ and }
			\label{oscillation estimate constant homogeneous}
			\sup\limits_{x, y \in B_{\rho}} \left\lvert \X u\left(x\right) - \X u\left(y\right) \right\rvert &\leq c_{2}\left(\frac{\rho}{R}\right)^{\beta} \left( \fint_{B_{R}} \left\lvert \X u \right\rvert^{p}\right)^{\frac{1}{p}},  
		\end{align}
		for any $0 < \rho < R/2,$ for some constant $c_{1}, c_{2}>0$ and $0 < \beta < 1$ depending only on $n$ and $p$.
	\end{theorem}
	The estimate \eqref{oscillation estimate constant homogeneous} implies a `excess-to-energy decay estimate'. However, for our purposes, the following fundamental `excess-to-excess decay estimate', proved in \cite{Mallick_Sil_ExcessDecay}, is crucial. 
	\begin{theorem}\label{MallickSil estimate}
		Let $\Omega \subset \mathbb{H}^{n}$ be open. Let $a>0$ and $1<p < \infty.$ Let $u \in HW^{1,p}_{\text{loc}}\left( \Omega\right)$ be a weak solution to 
		\begin{align*}
			\operatorname{div}_{\mathbb{H}}\left( a \left\lvert \X u\right\rvert^{p-2}\X u\right) &=0 &&\text{ in } \Omega.
		\end{align*}
		Then for any $1 \leq q \leq 2, $ we have the estimates 
		\begin{align}\label{excess decay}
			\left( \fint_{B_{\rho}(x_{0})} \left\lvert \X u  - \left( \X u\right)_{B_{\rho}(x_{0})}\right\rvert^{q}\right)^{\frac{1}{q}}&\leq C_{q}\left(\frac{\rho}{R}\right)^{\beta} \left( \fint_{B_{R}(x_{0})} \left\lvert \X u - \left( \X u\right)_{B_{R}(x_{0})}\right\rvert^{q}\right)^{\frac{1}{q}},  
		\end{align}
		for any $x_{0} \in \Omega$ and any $0 < \rho < R$ such that $B_{R}(x_{0}) \subset \Omega,$  where $0 < \beta < 1$ depends only on $n$ and $p$ and for some constant $C_{q}= C_{q}\left( n, p , q\right)>0$.
	\end{theorem}
	\begin{remark}\label{p uniform constant decay estimates}
		Note that while the constants in the estimates \eqref{sup estimate constant homogeneous}, \eqref{oscillation estimate constant homogeneous} and \eqref{excess decay} depend on $p$, but tracking the constants in the proofs reveal that the constants are explicitly quantifiable. Thus, if $1 < \gamma_{1} \leq p \leq \gamma_{2}< \infty,$ the constants can all be chosen (by replacing them with larger or smaller constants if necessary) to depend only $\gamma_{1}, \gamma_{2}$ such that the estimates are uniform in $p \in [\gamma_{1}, \gamma_{2}].$ See also Remark \ref{p uniform constants Sobolev inequalities}. Observe that though the proof of Theorem \ref{MallickSil estimate} in \cite{Mallick_Sil_ExcessDecay} uses a contradiction argument, it does not use a `contradiction-compactness argument' and thus the `contradiction argument' is, in effect, merely a case by case analysis, where the constants are quantifiable for each case. 
	\end{remark}	
	\subsection{Subelliptic \texorpdfstring{$p\left(x\right)$}{p(x)}-Laplacian}
	Consider the quasilinear equation 
	\begin{align}\label{main equation}
		\operatorname{div}_{\mathbb{H}} \left( a\left(x\right)\left\lvert \X u\right\rvert^{p\left(x\right)-2}\X u\right) &= \operatorname{div}_{\mathbb{H}} F &&\text{ in } \Omega,  
	\end{align}
	where $F \in L^{\gamma_1'}\left(\Omega; \mathbb{R}^{2n}\right).$ We begin by defining the notion of weak solutions. 
	\begin{definition} $u \in HW^{1,1}\left( \Omega \right)$ is called a \textbf{weak solution} of \eqref{main equation} if $u \in HW^{1,p\left(\cdot\right)}\left( \Omega \right)$ and satisfies 
		\begin{align}\label{weak formulation in WTdpx}
			\int_{\Omega} \left\langle a\left(x\right)\left\lvert \X u \right\rvert^{p\left(x\right)-2} \Xu, \X \phi \right\rangle = \int_{\Omega}\left\langle F, \X \phi \right\rangle \qquad \text{ for every } \phi \in HW_{0}^{1,p\left(\cdot\right)}\left(\Omega\right).
		\end{align}
	\end{definition}                                                                               
	\begin{proposition}\label{existence of minimizers for variable exponent}
		Let $\Omega \subset \mathbb{H}^{n}$ be open and bounded. Let $a:\Omega \rightarrow [\nu, L]$ be measurable, where $ 0 < \nu < L < \infty.$ Let $1 < p < \infty,$ $F\in L^{\gamma_1'}\left( \Omega; \mathbb{R}^{2n}\right)$ and $u_{0} \in HW^{1,p(\cdot)}\left( \Omega\right).$ Then the following minimization problem 
		\begin{align}
			m:= \inf \left\lbrace \int_{\Omega} \left[ a\left( x\right)\left\lvert \X u \right\rvert^{p(x)} - \left\langle F, \X u \right\rangle\right]\ \mathrm{d}x: u \in u_{0} + HW^{1,p(\cdot)}_{0}\left( \Omega\right)\right\rbrace 
		\end{align}
		admits a unique minimizer. Moreover, the minimizer $\bar{u} \in HW^{1,p(\cdot)}\left( \Omega\right)$ is the unique weak solution to the system 
		\begin{align*}
			\left\lbrace \begin{aligned}
				\operatorname{div}_{\mathbb{H}}\left( a\left(x\right)\left\lvert \X u\right\rvert^{p(x)-2}\X u\right) &= \operatorname{div}_{\mathbb{H}}F &&\text{ in } \Omega, \\
				u &=u_{0} &&\text{ on } \origpartial\Omega. 
			\end{aligned}\right. 
		\end{align*} 
	\end{proposition} 
	This can again be proved using standard variational methods.
	\section{Higher integrability}\label{higher integrability section}
	We begin by proving higher integrability estimates. 
	\begin{theorem}[Higher integrability]\label{higher integrability}
		Let $q>1$ and $F \in L^{\frac{q\gamma_{1}}{\gamma_{1}-1}} \left(\Omega; \mathbb{R}^{2n}\right)$. Let $u \in HW_{\text{loc}}^{1,p\left(\cdot\right)} \left(\Omega\right)$ be a local weak solution to the system 
		\begin{align}\label{p(x) Laplace homogeneous diveregnce form}
			\operatorname{div}_{\mathbb{H}}( a(x) \lvert \X u \rvert^{p\left(\cdot\right)-2} \Xu) )  = \operatorname{div}_{\mathbb{H}}F   &&\text{ in } \Omega,
		\end{align} where $a: \Omega \rightarrow [\nu, \Lambda]$ is a measurable function and $p:\Omega \rightarrow [\gamma_{1}, \gamma_{2}]$ satisfies \eqref{log holder condition thm} i.e. log-H\"older continuous. Let $\tilde{\Omega} \subset \subset \Omega$ be an open subset and set 
		\begin{align}\label{energy bound}
			K_{0}:= \int_{\tilde{\Omega}} \left\lvert \X u \right\rvert^{p\left(x\right)}\ \mathrm{d}x +1.
		\end{align}
		Then there exist constants $c \equiv c \left(Q, \gamma_1, \gamma_2, \nu, \Lambda, L \right)>0,$ an exponent $\sigma_{0} \equiv \sigma_{0}\left(Q, \gamma_1, \gamma_2, \nu, \Lambda, L, q \right) \in (0, 1)$ with $1+\sigma_0<q$ and a positive radius $R_{0} \equiv  R_{0}\left( Q, \gamma_1, \gamma_2, \nu, \Lambda, L, K_{0}\right) <1/K_0$ such that for any ball $B_{R}\subset \subset \tilde{\Omega}$ with $0 < R \leq R_{0},$ any $\xi \in \mathbb{R}^{2n}$ and $0\leq \sigma \leq \sigma_{0},$ we have 
		\begin{multline}\label{higher integrability estimate}
			\left( \fint_{B_{R/2}}\left\lvert \X u \right\rvert^{p\left(x\right)\left(1 + \sigma\right)}\ \mathrm{d}x\right)^{\frac{1}{1+\sigma}}
			\\ \leq c \left( \fint_{B_{R}}\left\lvert \X u \right\rvert^{p\left(x\right)}\ \mathrm{d}x + 1\right)  + c \left( 1+  \fint_{B_{R}}\left\lvert F - \xi \right\rvert^{\frac{\gamma_{1}\left( 1+\sigma\right)}{\gamma_{1}-1}}\ \mathrm{d}x \right)^{\frac{1}{1+\sigma}}. 
		\end{multline} 
	\end{theorem}
	\begin{proof}
		In the proof, $C$ or $c$ will denote a generic positive constant, depending only on $Q, \gamma_1, \gamma_2, \nu, \Lambda, L,$ whose value can change from line to line, even within a single display. Note that since $\Lambda$ and $L$ are both finite constants, we can assume $\Lambda=L,$ since otherwise we can replace both of them with $\max \left\lbrace L, \Lambda\right\rbrace.$ Also, we may assume $\gamma_1\leq\frac{2Q}{2Q-1}$.  We set $s:=\gamma_1$. Now fix $x_{0} \in \tilde{\Omega}$ and choose a radius $0<R_{0}<1/K_0$ small enough such that $B_{4R_{0}}\left(x_{0}\right) \subset \subset \tilde{\Omega}$ and we have 
		\begin{align}\label{choice of radius}
			\left\lbrace \begin{aligned}
				&\omega_p \left(8R_{0}\right) \leq \frac{s}{2Q}, \\
				&0 < \omega_p\left( R\right)\log \left(\frac{1}{R}\right) \leq L \qquad \text{ for all } 0 < R \leq 8R_{0}. 
			\end{aligned}\right. 
		\end{align}
		Note that this implies 
		\begin{align}\label{bound on p1 by p2}
			\frac{p_{1}}{p_{2}} \geq 1-\frac{s}{2Q},
		\end{align}
		where $p_{2}:= p^{+}_{B_{R}\left(x_{0}\right) }$ and $p_{1}:= p^{-}_{B_{R}\left(x_{0}\right) }.$ We set
		\begin{align}\label{p*1/s}
			\left(\frac{p_1}{s}\right)^*:= \begin{cases}
				\frac{Qp_1}{Qs-p_1} \mbox{ if } p_1<Qs, \\ 
				p_2+1 \mbox{ else .}
			\end{cases}
		\end{align}
		Then clearly, $\left(\frac{p_1}{s}\right)^*\geq p_2$. 
		%	Indeed, in the case $p_1<Qs$ this follows from the computation below. 
		%	\begin{align*}
			%		\left(\frac{p_1}{s}\right)^* \stackrel{\eqref{bound on p1 by p2}}{\geq}  p_2\frac{Q\left(1-\frac{s}{2Q}\right)}{Qs-s} \geq p_2,
			%	\end{align*}
		%	where to derive the last inequality we have used the bound $s=\gamma_1 \leq \frac{2Q}{2Q-1}$.
		From now on, all balls would be centered at $x_{0}$ and we omit writing the center of the balls. 
		Now we choose a cut-off function $\eta \in C_{c}^{\infty}\left(B_{R}\right)$ with 
		\begin{align*}
			0 \leq \eta \leq 1 \text{ in  } B_{R}, \qquad \eta \equiv 1 \text{ in } B_{R/2} \qquad \text{ and } \left\lvert \X \eta \right\rvert \leq \frac{C}{R}.   
		\end{align*}
		Note that $\phi = \eta^{p_{2}} \left( u - \left(u\right)_{R}\right)  \in HW_{0}^{1, p\left(\cdot\right)}\left( B_{R}\right)$ and thus we can plug $\phi$ as the test function in the weak formulation of  \eqref{p(x) Laplace homogeneous diveregnce form}. This yields 
		\begin{align*}
			\int_{B_{R}} \left\langle  a\left(x\right) \left\lvert \X u \right\rvert^{\left(p\left(x\right)-2\right)} \X u, \X\left( \eta^{p_{2}}  \left( u - \left(u\right)_{R}\right)  \right) \right\rangle = \int_{B_{R}} \left\langle  F, \X \left( \eta^{p_{2}}  \left( u - \left(u\right)_{R}\right)  \right) \right\rangle. 
		\end{align*}
		Integrating by parts, we have 
		\begin{align*}
			\int_{B_{R}} &\left\langle  a\left(x\right) \left\lvert \X u \right\rvert^{p\left(x\right)-2} \X u, \X\left( \eta^{p_{2}} \left( u - \left(u\right)_{R}\right)  \right) \right\rangle \\&= \int_{B_{R}}\left\langle F, \X\left(\eta^{p_{2}}  \left( u - \left(u\right)_{R}\right) \right) \right\rangle  =  \int_{B_{R}}\left\langle F - \xi, \X\left(\eta^{p_{2}}  \left( u - \left(u\right)_{R}\right) \right) \right\rangle, 
		\end{align*}
		for any $\xi \in \mathbb{R}^{2n}.$	Now note that  
		\begin{align*}
			\X \left( \eta^{p_{2}}  \left( u - \left(u\right)_{R}\right)  \right) = p_{2}\eta^{p_{2}-1} \left( u - \left(u\right)_{R}\right) \X \eta  + \eta^{p_{2}}\X u.
		\end{align*}
		Plugging this, we obtain 
		\begin{align}
			\gamma &\int_{B_{R}}\eta^{p_{2}}\left\lvert \X u \right\rvert^{p\left(x\right)} \notag \\&\leq \int_{B_{R}} \eta^{p_{2}}\left\langle  a\left(x\right) \left\lvert \X u \right\rvert^{p\left(x\right)-2} \X u, \X u \right\rangle \notag \\&= \begin{multlined}[t] \int_{B_{R}} p_{2}\eta^{p_{2}-1} \left( u - \left(u\right)_{R}\right)\left\langle F - \xi, \X \eta \right\rangle + \int_{B_{R}}\eta^{p_{2}}\left\langle F - \xi, \Xu \right\rangle \notag \\- \int_{B_{R}}p_{2}\eta^{p_{2}-1} \left( u - \left(u\right)_{R}\right)\left\langle a\left(x\right) \left\lvert \X u \right\rvert^{p\left(x\right)-2} \X u,  \X \eta \right\rangle \end{multlined} \notag \\
			&:= I_{1} + I_{2} +I_{3}.  \label{p(x) energy}
		\end{align}
		Now, using Young's inequality  and the fact $p_2' \leq \left(p(x)\right)' \leq \gamma_1'$, for all $x\in B_R$, we have 
		\begin{align}
			\left\lvert I_{1} \right\rvert &\leq c\int_{B_{R}} \eta^{\frac{p\left(x\right)\left(p_{2}-1\right)}{p\left(x\right)-1}} \left\lvert F - \xi \right\rvert^{\frac{p\left(x\right)}{p\left(x\right)-1}} + c \int_{B_{R}}\left\lvert \X  \eta \right\rvert^{p_{2}}\left\lvert  u - \left(u\right)_{R}\right\rvert^{p_{2}} + c \left\lvert B_{R} \right\rvert \notag \\
			&\leq c\int_{B_{R}} \left\lvert F - \xi \right\rvert^{\frac{p\left(x\right)}{p\left(x\right)-1}} + \frac{c}{R^{p_{2}}} \int_{B_{R}}\left\lvert  u - \left(u\right)_{R} \right\rvert^{p_{2}} + c \left\lvert B_{R} \right\rvert, \label{estimate of I1}
		\end{align}
		Again using Young's inequality with $\varepsilon>0,$ we have 
		\begin{align}
			\left\lvert I_{2} \right\rvert &\leq \varepsilon \int_{B_{R}}\eta^{p_{2}}\left\lvert \X u \right\rvert^{p\left(x\right)} + C_{\varepsilon} \int_{B_{R}} \eta^{p_{2}} \left\lvert F - \xi \right\rvert^{\frac{p\left(x\right)}{p\left(x\right)-1}}\notag \\
			&\leq \varepsilon \int_{B_{R}}\eta^{p_{2}}\left\lvert \X u \right\rvert^{p\left(x\right)} + C_{\varepsilon} \int_{B_{R}} \left\lvert F - \xi \right\rvert^{\frac{p\left(x\right)}{p\left(x\right)-1}} \label{estimate of I2} \\ \intertext{ and } 
			\left\lvert I_{3} \right\rvert 		 &\leq L\gamma_{2}\int_{B_{R}}\eta^{p_{2}-1}\left\lvert \X u \right\rvert^{p\left(x\right)-1}\left\lvert  \X \eta \right\rvert \left\lvert   u - \left(u\right)_{R} \right\rvert \notag	\\&\leq \varepsilon \int_{B_{R}}\eta^{\frac{p\left(x\right)\left(p_{2}-1\right)}{p\left(x\right)-1}}\left\lvert \X u \right\rvert^{p\left(x\right)} + C_{\varepsilon}\int_{B_{R}}\left\lvert \X  \eta \right\rvert^{p_{2}}\left\lvert   u - \left(u\right)_{R} \right\rvert^{p_{2}}  + c \left\lvert B_{R} \right\rvert \notag \\
			&\leq \varepsilon \int_{B_{R}}\eta^{p_{2}}\left\lvert \X u  \right\rvert^{p\left(x\right)} + \frac{C_{\varepsilon}}{R^{p_{2}}}\int_{B_{R}}\left\lvert   u - \left(u\right)_{R} \right\rvert^{p_{2}}  + c \left\lvert B_{R} \right\rvert, \label{estimate of I3}
		\end{align}
		where in the last line, we have used the fact that $\frac{p\left(x\right)\left(p_{2}-1\right)}{p\left(x\right)-1} > p_{2}$ and $0 \leq \eta \leq 1,$ implying $$ \eta^{\frac{p\left(x\right)\left(p_{2}-1\right)}{p\left(x\right)-1}} \leq \eta^{p_{2}}.$$
		Combining the estimates \eqref{estimate of I1}, \eqref{estimate of I2}, \eqref{estimate of I3} with \eqref{p(x) energy}, choosing $\varepsilon>0$ small enough and then dividing by $\left\lvert B_{R}\right\rvert $  we deduce 
		\begin{align}\label{before poincare sobolev}
			\fint_{B_{R}}\eta^{p_{2}}\left\lvert \X u \right\rvert^{p\left(x\right)}  &\leq c\fint_{B_{R}} \left\lvert F - \xi \right\rvert^{\frac{p\left(x\right)}{p\left(x\right)-1}} + \frac{c}{R^{p_{2}}} \fint_{B_{R}}\left\lvert   u - \left(u\right)_{R} \right\rvert^{p_{2}} + c.
		\end{align}
		Now note that by the fact $s=\gamma_1$, $p_1/s\geq 1$ and by \eqref{p*1/s}, we have , $\left(\frac{p_1}{s}\right)^*\geq p_2$. So using H\"older inequality and the Poincar\'e-Sobolev inequality \eqref{poincaresobolevineqwithmeans} we derive
		\begin{align*}
			\frac{1}{R} \left(\fint_{B_{R}}\left\lvert   u - \left(u\right)_{R} \right\rvert^{p_{2}}\right)^\frac{1}{p_2} &\leq \frac{1}{R} \left(\fint_{B_{R}}\left\lvert   u - \left(u\right)_{R} \right\rvert^{\left(\frac{p_1}{s}\right)^*}\right)^{\frac{1}{\left(\frac{p_1}{s}\right)^*}} \\
			&\leq \left( \fint_{B_R} \left \lvert \X u \right \rvert ^{\frac{p_1}{s}} \right)^{\frac{s}{p_1}}.
		\end{align*}
		
		Combining this estimate with \eqref{before poincare sobolev} we deduce 
		\begin{align*}
			\fint_{B_{R}}\eta^{p_{2}}\left\lvert \X u \right\rvert^{p\left(x\right)} \leq  \fint_{B_{R}} \left\lvert F - \xi \right\rvert^{\frac{p\left(x\right)}{p\left(x\right)-1}} + c \left( 	\fint_{B_{R}} \left\lvert \X u\right\rvert^{\frac{p_{1}}{s}}\right)^{\frac{sp_{2}}{p_{1}}} + c.
		\end{align*}
		Since $\eta \equiv 1$ on $B_{R/2},$ this implies 
		\begin{align}\label{prelim rev holder}
			\fint_{B_{R/2}}\left\lvert \X u \right\rvert^{p\left(x\right)} \leq c \left( 	\fint_{B_{R}} \left\lvert \X u\right\rvert^{\frac{p_{1}}{s}}\right)^{\frac{sp_{2}}{p_{1}}} + c \left( 1 +  \fint_{B_{R}} \left\lvert F - \xi \right\rvert^{\frac{p\left(x\right)}{p\left(x\right)-1}} \right). 
		\end{align}
		Since $ p\left(x\right) \geq p_{1}$ and consequently $p^{'}\left(x\right) \leq p_{1}^{'} \leq \gamma_{1}^{'},$ and we have   
		\begin{align}
			\fint_{B_{R}} \left\lvert \X u\right\rvert^{\frac{p_{1}}{s}} \leq  c\fint_{B_{R}} \left\lvert \X u\right\rvert^{\frac{p\left(x\right)}{s}} +  c \label{estimate of du by p}
		\end{align}
		and 
		\begin{align}
			\fint_{B_{R}} \left\lvert F - \xi \right\rvert^{\frac{p\left(x\right)}{p\left(x\right)-1}} \leq c \fint_{B_{R}} \left\lvert F - \xi \right\rvert^{\frac{\gamma_{1}}{\gamma_{1}-1}} +  c .\label{estimate of F by p_{1}}
		\end{align}
		Using \eqref{estimate of du by p} and as $sp_{2}/p_{1} > 1,$ we have   
		\begin{align}\label{estimate of p1bys}
			\left( 	\fint_{B_{R}} \left\lvert \X u\right\rvert^{\frac{p_{1}}{s}}\right)^{\frac{sp_{2}}{p_{1}}} &\leq \left( \fint_{B_{R}} \left\lvert \X u\right\rvert^{\frac{p\left(x\right)}{s}} +  1\right)^{\frac{sp_{2}}{p_{1}}} \notag\\&\leq c \left[ \left( \fint_{B_{R}} \left\lvert \X u\right\rvert^{\frac{p\left(x\right)}{s}} \right)^{\frac{sp_{2}}{p_{1}}} + 1\right] \notag\\&=c\left( \fint_{B_{R}} \left\lvert \X u\right\rvert^{\frac{p\left(x\right)}{s}} \right)^{s}\left( \fint_{B_{R}} \left\lvert \X u\right\rvert^{\frac{p\left(x\right)}{s}} \right)^{\frac{s\left( p_{2}-p_{1}\right)}{p_{1}}} + c. 
		\end{align}
		Now we estimate the term 
		$$I := \left( 	\fint_{B_{R}} \left\lvert \X u\right\rvert^{\frac{p\left(x\right)}{s}}\right)^{\frac{s\left( p_{2}-p_{1}\right)}{p_{1}}}. $$
		Since $s>1,$ by H\"{o}lder inequality, we deduce 
		\begin{align*}
			I &\leq \left[ \left(\fint_{B_{R}} \left\lvert \X u\right\rvert^{p\left(x\right)} \right)^{\frac{1}{s}}\left( \fint_{B_{R}} 1 \right)^{\frac{s-1}{s}}\right]^{\frac{s\left( p_{2}-p_{1}\right)}{p_{1}}} 
			\\&= \left(\fint_{B_{R}} \left\lvert \X u\right\rvert^{p\left(x\right)} \right)^{\frac{ p_{2}-p_{1}}{p_{1}}} 
			\\&= cR^{-\frac{Q\left( p_{2}-p_{1}\right)}{p_{1}}}\left(\int_{B_{R}} \left\lvert \X u\right\rvert^{p\left(x\right)} \right)^{\frac{ p_{2}-p_{1}}{p_{1}}} 
			\\&\leq cR^{-\frac{Q\left( p_{2}-p_{1}\right)}{p_{1}}}\left(\int_{B_{R}} \left\lvert \X u\right\rvert^{p\left(x\right)} +1 \right)^{\frac{ p_{2}-p_{1}}{p_{1}}} 
			\\&\leq cR^{-Q\gamma_2^{-1}\omega_p\left(R\right)}\left(\int_{B_{R}} \left\lvert \X u\right\rvert^{p\left(x\right)} +1 \right)^{\frac{\omega_p\left(R\right)}{p_{1}}} \\
			&\leq c \left(\frac{K_0}{R^{Q\gamma_2^{-1}}}\right)^{\omega_p(R)} \leq c \left(\frac{1}{R^{1+Q\gamma_2^{-1}}}\right)^{\omega_p(R)} \stackrel{\eqref{log holder condition thm}}{\leq}  ce^{L(1+Q\gamma_2^{-1})}, 
		\end{align*}
		where in the last line we have used the fact $K_0<R_0^{-1}< R^{-1}.$
		Combining this with \eqref{prelim rev holder}, \eqref{estimate of p1bys} and \eqref{estimate of F by p_{1}}, we arrive at 
		\begin{align*}
			\fint_{B_{R/2}}\left\lvert \X u \right\rvert^{p\left(x\right)} \leq c\left( \fint_{B_{R}} \left\lvert \X u\right\rvert^{\frac{p\left(x\right)}{s}} \right)^{s}  + c \left( 1 +  \fint_{B_{R}} \left\lvert F - \xi \right\rvert^{\frac{\gamma_{1}}{\gamma_{1}-1}} \right). 
		\end{align*} Since $s>1,$ this immediately implies the following reverse H\"{o}lder inequality with increasing support 
		\begin{align*}
			\left( \fint_{B_{R/2}} \left\lvert \X u \right\rvert^{p\left(x\right)} \right)^{\frac{1}{s}}\leq c \fint_{B_{R}} \left\lvert \X u\right\rvert^{\frac{p\left(x\right)}{s}}  + c \left( 1 +  \fint_{B_{R}} \left\lvert F - \xi \right\rvert^{\frac{\gamma_{1}}{\gamma_{1}-1}} \right)^{\frac{1}{s}}. 
		\end{align*}
		Now a standard Gehring lemma type argument (c.f. Theorem 3.3 \cite{GheringZatorska-Goldstein} ) implies the existence of a constant $\sigma_{0}$ satisfying $1+\sigma_0<q$, such that for any $0 < \sigma \leq \sigma_{0},$ we have 
		\begin{multline*}
			\left( \fint_{B_{R/2}}\left\lvert \X u \right\rvert^{p\left(x\right)\left(1 + \sigma\right)}\ \mathrm{d}x\right)^{\frac{1}{1+\sigma}}
			\\ \leq c \fint_{B_{R}}\left\lvert \X u \right\rvert^{p\left(x\right)}\ \mathrm{d}x  + c \left( 1+  \fint_{B_{R}}\left\lvert F - \xi \right\rvert^{\frac{\gamma_{1}\left( 1+\sigma\right)}{\gamma_{1}-1}}\ \mathrm{d}x \right)^{\frac{1}{1+\sigma}}.
		\end{multline*} 
		This completes the proof. 
	\end{proof}
	When $F \equiv 0$ and $a\equiv 1,$ the same proof coupled with a standard interpolation argument implies the following (see also \cite[Theorem 3.1]{Ok_pxStein}, \cite[Theorem 2.4]{Baroni_pxStein}). 
	\begin{theorem}\label{higher integrability 2}
		Let $w \in HW_{\text{loc}}^{1,p\left(\cdot\right)} \left(\Omega\right)$ be a local weak solution of 
		\begin{align}\label{p(x) harmonic}
			\operatorname{div}_{\mathbb{H}}( \lvert \X w \rvert^{p\left(\cdot\right)-2} \X w) )  = 0   &&\text{ in } \Omega,
		\end{align} where $p:\Omega \rightarrow [\gamma_{1}, \gamma_{2}]$ satisfies \eqref{log holder condition thm}. Let $\tilde{\Omega} \subset \subset \Omega$ be an open subset and set 
		\begin{align}\label{energy bound 2}
			K_{0}:= \int_{\tilde{\Omega}} \left\lvert \X w \right\rvert^{p\left(x\right)}\ \mathrm{d}x +1.
		\end{align}
		Then there exist constants $c_{h} \equiv c_{h} \left(Q, \gamma_1, \gamma_2, L \right)>0,$ and an exponent $\sigma_{0} \equiv \sigma_{0}\left(Q, \gamma_1, \gamma_2, L \right) \in (0,1)$ and a radius $0<R_{0} = R_{0}\left( Q, \gamma_1, \gamma_2, L, K_{0}\right) <1/K_0$ such that for any ball $B_{R}\subset \subset \tilde{\Omega}$ with $0 < R \leq R_{0}$, any $0 < \sigma \leq \sigma_{0}$  and any point $\bar{x} \in B_{R},$ we have 
		\begin{multline}\label{higher integrability estimate RHS 0}
			\max \left\lbrace \left( \fint_{B_{R/2}}\left\lvert \X w \right\rvert^{p\left(x\right)\left(1 + \sigma\right)}\ \mathrm{d}x\right)^{\frac{1}{1+\sigma}}, 	\left( \fint_{B_{R/2}}\left\lvert \X w \right\rvert^{p\left(\bar{x}\right)\left(1 + \sigma\right)}\ \mathrm{d}x\right)^{\frac{1}{p\left(\bar{x}\right)\left(1 + \sigma\right)}} \right\rbrace 
			\\ \leq c_{h} \left[  1 + \left( \fint_{B_{R}}\left\lvert \X w \right\rvert^{\gamma}\ \mathrm{d}x \right)^{\frac{1}{\gamma}} \right]  . 
		\end{multline} 
	\end{theorem}
	\section{H\"{o}lder continuity}	\label{Holder continuity section}
	In this section we prove Theorem \ref{Theorem holder continuity}. Our techniques here are adaptations of \cite{Acerbi_Mingione_C1alpha}. \bigskip 
	
	\textbf{Proof of Theorem \ref{Theorem holder continuity}:} 
	Note that the Euler-Lagrange equation for the functional $J$ is given by 
	\begin{align}\label{EL eqn for J}
		\operatorname{div}_{\mathbb{H}} \left( a\left(x\right)\left\lvert \X u\right\rvert^{p\left(x\right)-2}\X u\right) &= \operatorname{div}_{\mathbb{H}} F &&\text{ in } \Omega.
	\end{align}	
	Thus, every local minimizer of the $J$ is a local weak solution for \eqref{EL eqn for J}. Then the conclusions of part $(a)$ is just Theorem \ref{higher integrability}. 
	We fix $x_{0} \in \Omega$ and choose a radius $\bar{R}>0$ such that $B_{8\bar{R}}\left(x_{0}\right) \subset \subset \Omega$ and set \begin{align*}
		K_{0}:= \int_{B_{8\bar{R}}\left(x_{0}\right)} \left\lvert \X u \right\rvert^{p\left(x\right)}\ \mathrm{d}x +1.
	\end{align*} Now we use Thoerem \ref{higher integrability} with $\tilde{\Omega}$ replaced by $B_{8\bar{R}}\left(x_{0}\right)$ to determine the constants $R_{0}, c$ and $\sigma_{0}.$ Now we choose $\sigma \leq \min \left\lbrace \gamma_{1} -1, \sigma_{0}/2\right\rbrace $ and determine $R_{1}>0$ by the conditions 
	$\omega_p \left(8R_{1}\right) < \sigma /4 $ and $R_{1} < \min \left\lbrace R_{0}/16, 1/16 \right\rbrace.$ Now we choose a radius $0 < \tilde{R} < R_{1}/16$ small enough such that $B_{16R}\left(x_{c}\right)\equiv B_{16R} \subset \subset B_{R_{1}/4}\left(x_{0}\right)$ for any $0 < R < \tilde{R}/32 $ and $x_{c} \in \overline{B_{\tilde{R}}\left(x_{0}\right)}.$ We shall further reduce $\tilde{R}$ in subsequent steps. Set 
	\begin{align*}
		p_{1}:= \min_{\overline{B_{4R}}\left(x_c\right)}p\left(x\right), \quad  \qquad p_{2}:= \max\limits_{\overline{B_{4R}}\left(x_c\right)}p\left(x\right) \quad \text{ and } \quad p_{m}:= \max\limits_{\overline{B_{R_{1}}}\left(x_{0}\right)}p\left(x\right). 
	\end{align*}
	Note that by our choice of parameters, since $\sigma \leq p_{1} - 1$, we have 
	\begin{align}\label{p2 by px}
		p_{2} \left( 1 + \frac{\sigma}{4}\right) &\leq p_{1} \left(  1+ \frac{\sigma}{4} + \omega_p\left(4R\right) \right)  \notag \\ &\leq p\left(x\right)\left(  1+ \frac{\sigma}{4} + \omega_p\left(4R\right) \right) \leq p\left(x\right) \left(  1+ \sigma \right),
	\end{align}  for all $x \in \overline{B_{4R}}$. Since we also have $p_{m}\left( 1+ \frac{\sigma}{4} \right) \leq p\left(x\right)\left(1+ \sigma\right)$ in $B_{R_{1}}\left(x_{0}\right), $ by Theorem \ref{higher integrability}, we have $\X u \in L^{p_{m}}\left(B_{R_{1}}\left(x_{0}\right)\right).$ Consequently, we see that  $u \in HW^{1, p_{m}}\left( B_{R_{1}/2}\left(x_{0}\right)\right)$ is a weak solution of the following 
	\begin{align}\label{eq for u in Xu holder proof}
		\operatorname{div}_{\mathbb{H}}( a(x) \lvert \X u \rvert^{p\left(\cdot\right)-2} \X u) )  = \operatorname{div}_{\mathbb{H}}F
	\end{align} in $B_{R_{1}/2}\left(x_{0}\right).$ We set \begin{align*}
		K:= \int_{B_{R_{1}/4}\left(x_{0}\right)} \left\lvert \X u \right\rvert^{p_{m}} \qquad \text{ and } \qquad  \mathfrak{p}\left( \rho\right):= \max\limits_{\overline{B_{4\rho}\left( x_{c}\right)}}p\left(x\right).
	\end{align*} 
	Since \eqref{vanishing log holder condition thm} implies \eqref{log holder condition thm}, we may assume \eqref{log holder condition thm} with $\Lambda= L$. Now we claim the following.  
	\begin{claim}[Morrey bound]\label{Morrey bound claim}
		Assume $a$ is continuous,  $p$ satisfies \eqref{vanishing log holder condition thm} and $F \in {\rm{BMO}}_{\rm{loc}}\left(\Omega; \mathbb{R}^{2n}\right)$. Then for any $0 < \tau <Q,$ there exists a positive constant $$ C = C\left( Q, \gamma_{1}, \gamma_{2}, \gamma,  L, \omega_p, \omega_{a} ,\left[ F \right]_{{\rm{BMO}} \left( B_{R_{1}}\left(x_{0}\right);\mathbb{R}^{2n}\right)}, K_{0},K, R_1, \tau \right)$$ and a small enough choice for the radius $\tilde{R},$ depending only on $Q, \gamma_{1}, \gamma_{2}, \gamma,$ $  \omega_p, \omega_{a} , L,$ $ K_{0},K, R_1, \tau$ such that   
		whenever $0 < \rho<  \tilde{R}/\left(64\right)^{2}$ and $x_{c} \in \overline{B_{\tilde{R}}\left(x_{0}\right)},$ we have  
		\begin{align*}
			\int_{B_{\rho}\left(x_{c}\right)} \left\lvert \X u \right\rvert^{\mathfrak{p}\left(\rho\right)} \leq c \rho^{Q-\tau}. 
		\end{align*}
	\end{claim} 
	Part (b) easily follows from the claim via a standard covering argument and Proposition \ref{Morrey's theorem}. We start with $0 < R < \tilde{R}/32 $.  To prove the claim, note that there exists a point $x_{\text{max}} \in \overline{B_{4R}}\left(x_c\right)$ such that 
	$ p\left(x_{\text{max}}\right) = p_2.$ Since  $u \in HW^{1,p_{m}}\left(B_{R_{1}/2}\left(x_{0}\right)\right)$ and thus consequently $u \in HW^{1,p_{2}}\left(B_R\left(x_c\right)\right),$ by Proposition \ref{existence of minimizers for constant exponent}, we find $v \in HW^{1, p_{2}}\left( B_{R}\left(x_c\right)\right)$  such that 
	\begin{align}\label{frozen system for comparison}
		\left\lbrace \begin{aligned}
			\operatorname{div}_{\mathbb{H}}( a(x_{\text{max}}) \lvert \X v \rvert^{p\left(x_{\text{max}}\right)-2} \X v ) ) &=0 &&\text{ in }  B_{R} \left(x_c\right), \\
			v &=u &&\text{ on } \origpartial B_{R}\left(x_c\right). 
		\end{aligned}\right. 
	\end{align} 
	For the rest of the proof, $B_\rho$ and $B_R$ denotes balls centred at $x_c$. Observe that it is quite possible that $x_{\text{max}} \notin B_{R},$ but this would not be a problem. Since $v$ minimizes the functional 
	\begin{align*}
		v \mapsto \int_{B_{R}} a\left( x_{\text{max}} \right) \left\lvert \X v \right\rvert^{p_{2}} \quad \text{ in } u + HW_{0}^{1,p_{2}}\left( B_{R}\right),  
	\end{align*}
	by minimality, \eqref{comparison algebraic ineq} and the estimate \eqref{sup estimate constant homogeneous}, for any $0 < \rho < R/2,$ we deduce  
	\begin{align}\label{decay for Morrey bound}
		\int_{B_{\rho}} \left\lvert\X u  \right\rvert^{p_{2}} \notag &\leq c \int_{B_{\rho}} \left\lvert \X v  \right\rvert^{p_{2}} + c \int_{B_{\rho}} \left( \left\lvert \X u\right\rvert^{2} + \left\lvert \X v\right\rvert^{2} \right)^{\frac{p_{2}-2}{2}}\left\lvert \X u - \X v\right\rvert^{2} \notag \\
		&\stackrel{\eqref{sup estimate constant homogeneous}}{\leq} c \left(\frac{\rho}{R}\right)^{Q}\int_{B_{R}} \left\lvert \X v \right\rvert^{p_{2}} + c \int_{B_{\rho}} \left( \left\lvert \X u\right\rvert^{2} + \left\lvert \X v\right\rvert^{2} \right)^{\frac{p_{2}-2}{2}}\left\lvert \X u - \X v\right\rvert^{2} \notag \\
		&\leq c \left(\frac{\rho}{R}\right)^{Q}\int_{B_{R}} \left\lvert \X u \right\rvert^{p_{2}} + c \int_{B_{R}} \left( \left\lvert \X u\right\rvert^{2} + \left\lvert \X v\right\rvert^{2} \right)^{\frac{p_{2}-2}{2}}\left\lvert \X u - \X v\right\rvert^{2}. 
	\end{align}
	Clearly, $v\in HW^{1,p(\cdot)} \left(B_R\right).$ So, the zero extension of the function $u-v$ belongs to $HW_{0}^{1, p(\cdot)}\left(B_{R_{1}/2}\left(x_{0}\right)\right)$ (cf. \cite{Freeman_Thesis_pxonCarnot}). Thus to estimate the last term on the right, we plug $u - v$ as a test function in the weak formulation of \eqref{eq for u in Xu holder proof} and \eqref{frozen system for comparison} and use standard manipulations to arrive at  
	\begin{align}\label{estimate 1}
		\gamma \int_{B_{R}} &\left( \lvert \X u \rvert^{2} + \lvert \X v \rvert^{2}\right)^{\frac{p_{2}-2}{2}}\left\lvert \X u - \X v\right\rvert^{2} \notag \\&\leq \int_{B_{R}}\left\langle   a(x_{\text{max}})\left(  \lvert \X u \rvert^{p_{2}-2} \X u- \lvert \X v \rvert^{p_{2}-2} \X v \right), \X u - \X v \right\rangle \notag 
		\\&= \int_{B_{R}}\left\langle   a(x_{\text{max}})\lvert \X u \rvert^{p_{2}-2} \X u, \X u - \X v \right\rangle :=I_{1} + I_{2} + I_{3},
	\end{align}
	where 
	\begin{align*}
		I_{1}&:= \int_{B_{R}}\left\langle   \left[ a(x_{\text{max}}) - a\left(x\right)\right] \left\lvert \X u \right\rvert^{p_{2}-2} \X u, \X u - \X v \right\rangle, \\ I_{2} &:=  \int_{B_{R}}\left\langle   a\left(x\right)\left\lvert \X u \right\rvert^{p\left(x\right)-2} \X u, \X u - \X v \right\rangle, 
		\\ 
		I_{3}&:= \int_{B_{R}}\left\langle   a\left(x\right)\left[ \lvert \X u \rvert^{p_{2}-2}\X u - \left\lvert \X u \right\rvert^{p\left(x\right)-2} \X u \right], 
		\X u - \X v \right\rangle. 
	\end{align*}
	Using the minimality of $v$, we can easily estimate $I_{1}$ to obtain 
	\begin{align}\label{estimate of a term}
		\left\lvert I_{1}\right\rvert &\leq \int_{B_{R}}   \left\lvert  a\left(x\right) - a(x_{\text{max}})\right\rvert \left\lvert \Xu  \right\rvert^{p_{2}-1}\left\lvert \X u  - \X v \right\rvert \leq c\omega_{a}\left(4R\right)  \int_{B_{R}}\left\lvert \X u  \right\rvert^{p_{2}}. 
	\end{align}
	The weak formulation for \eqref{eq for u in Xu holder proof} and Young's inequality with $\varepsilon>0$ implies  
	\begin{align}\label{estimate of F term}
		\left\lvert I_{2}\right\rvert &= \left\lvert \int_{B_{R}}\left\langle   a\left(x\right)\left\lvert \X u \right\rvert^{p\left(x\right)-2} \X u, \X u - \X v \right\rangle \right\rvert \notag \\&= \left\lvert \int_{B_{R}}\left\langle F - \left(F\right)_{R}, \X u - \X v\right\rangle \right\rvert \leq \varepsilon \int_{B_{R}} \left\lvert \X u  \right\rvert^{p_{2}} + C_{\varepsilon} \int_{B_{R}} \left\lvert F - \left(F\right)_{R} \right\rvert^{\frac{p_{2}}{p_{2}-1}},
	\end{align}
	for any $\varepsilon>0.$ For $I_{3},$ using the elementary inequality $\left \lvert e^x - 1 \right \rvert \leq \lvert x \rvert e^{\lvert x\rvert}$, we have 
	\begin{align}\label{estimate of the log term}
		\left\lvert I_{3}\right\rvert &\leq L\int_{B_{R}}\left\lvert \left\lvert \X u  \right\rvert^{p\left(x\right)-2} \X u  - \lvert \X u  \rvert^{p_{2}-2} \X u  \right\rvert\left\lvert \X u  - \X v \right\rvert \notag \\
		&\leq cL\omega_p\left(4R\right)\int_{B_{R}}\left\lvert  \X u \right\rvert^{p_{2}-1}\left\lvert \log \right\rvert \left( \left\lvert \X u  \right\rvert\right)\rvert\left\lvert \X u  - \X v\right\rvert \notag \\
		&\leq cL \omega_p\left(4R\right) \left( \int_{B_{R}} \left\lvert \X u  -\X v \right\rvert^{p_{2}} \right)^{\frac{1}{p_{2}}}\left(  \int_{B_{R}} \left\lvert \X u  \right\rvert^{p_{2}}\left\lvert \log \right\rvert^{\frac{p_{2}}{p_{2}-1}} \left( \left\lvert \X u \right\rvert\right)\right)^{\frac{p_{2} -1}{p_{2}}} \notag \\ 
		&\leq cL \omega_p\left(4R\right)\int_{B_{R}} \left\lvert \X u\right\rvert^{p_{2}} + \omega_p\left(4R\right) \int_{B_{R}} \left\lvert \X u \right\rvert^{p_{2}}\left\lvert \log \right\rvert^{\frac{p_{2}}{p_{2}-1}} \left( \left\lvert \X u  \right\rvert\right),
	\end{align}
	where we have used the minimality of $v$ in the last line. Now observe that by \eqref{less than e}, if $\left\lvert \X u \right\rvert < e,$ then the last term on the right can be estimated by $cR^{Q}.$ Thus, we have 
	\begin{align}\label{est -4}
		\int_{B_{R}} & \left\lvert \X u \right\rvert^{p_{2}} \left\lvert \log \right\rvert^{\frac{p_{2}}{p_{2}-1}} \left( \left\lvert \X u \right\rvert\right) \notag \\
		&\leq c\int_{B_{R}\cap \left\lbrace \left\lvert \X u  \right\rvert \geq e\right\rbrace} \left\lvert \X u  \right\rvert^{p_{2}}\left\lvert \log \right\rvert^{\frac{p_{2}}{p_{2}-1}} \left( \left\lvert \X u \right\rvert\right) + cR^{Q} \notag \\
		&\leq cR^{Q}\fint_{B_{R}} \left\lvert \X u \right\rvert^{p_{2}}\left\lvert \log \right\rvert^{\frac{p_{2}}{p_{2}-1}} \left( e + \left\lvert \X u \right\rvert^{p_{2}}\right) + cR^{Q} \notag \\
		&\begin{aligned}[b]
			\stackrel{\eqref{concavity of log}}{\leq} cR^{Q}\fint_{B_{R}} &\left\lvert \X u \right\rvert^{p_{2}}\left\lvert \log \right\rvert^{\frac{p_{2}}{p_{2}-1}} \left( e + \frac{\left\lvert \X u \right\rvert^{p_{2}}}{\fint_{B_{R}}\left\lvert \X u \right\rvert^{p_{2}}}\right) \\&+ cR^Q\log^{\left( \frac{p_{2}}{p_{2}-1}\right)} \left( e + \fint_{B_{R}}\left\lvert \X u  \right\rvert^{p_{2}}\right)\fint_{B_{R}} \left\lvert \X u \right\rvert^{p_{2}} + cR^{Q}.
		\end{aligned}
	\end{align} 
	Now, thanks to \eqref{LlogbetaL estimate}, we have 
	\begin{align}
		cR^Q\fint_{B_{R}} \left\lvert \X u \right\rvert^{p_{2}}&\left\lvert \log \right\rvert^{\frac{p_{2}}{p_{2}-1}} \left( e + \frac{\left\lvert \X u \right\rvert^{p_{2}}}{\fint_{B_{R}}\left\lvert \X u  \right\rvert^{p_{2}}}\right) \notag \\&\leq c\sigma^{-\frac{p_{2}}{p_{2}-1}}R^{Q}\left( \fint_{B_{R}} \left\lvert \X u  \right\rvert^{p_{2}\left(1+ \sigma/4\right)}\right)^{\frac{1}{1+ \sigma/4}}. \label{est -5}
	\end{align}
	Note that using \eqref{p2 by px}, we can trivially estimate 
	\begin{align*}
		\fint_{B_{R}} \left\lvert \X u \right\rvert^{p_{2}\left( 1+ \frac{\sigma}{4}\right)} 
		&\stackrel{\eqref{p2 by px}}{\leq}   1+  \frac{1}{\left\lvert B_{R}\right\rvert}	\int_{B_{R}\cap \left\lbrace \left\lvert \X u \right\rvert \geq 1 \right\rbrace } \left\lvert \X u \right\rvert^{p\left(x\right)\left(  1+ \frac{\sigma}{4} + \omega_p\left(4R\right) \right)}  \\
		&\leq 1 + \fint_{B_{R}} \left\lvert \X u \right\rvert^{p\left(x\right)\left(  1+ \frac{\sigma}{4} + \omega_p\left(4R\right) \right)}.
	\end{align*}
	Again using \eqref{p2 by px} and replacing $R$ by $2R$ and taking $\xi= \left(F\right)_{2R}$ in \eqref{higher integrability estimate}, we deduce 
	\begin{multline*}
		\fint_{B_{R}} \left\lvert \X u \right\rvert^{p\left(x\right)\left(  1+ \frac{\sigma}{4} + \omega_p\left(4R\right) \right)} 
		\leq c   + c \left(  \fint_{B_{2R}}\left\lvert F   -  \left(F\right)_{2R}\right\rvert^{\frac{\gamma_{1}\left( 1+\sigma\right)}{\gamma_{1}-1}}\right)^{\frac{1+ \frac{\sigma}{4} + \omega_p\left(4R\right)}{1 + \sigma}} \\ + c\left( \fint_{B_{2R}}\left\lvert \X u \right\rvert^{p\left(x\right)}\right)^{\left(  1+ \frac{\sigma}{4} + \omega_p\left(4R\right) \right)} .
	\end{multline*}
	Thus, we deduce 
	\begin{align*}
		&\left( \fint_{B_{R}} \left\lvert \X u  \right\rvert^{p_{2}\left(1+ \sigma/4\right)}\right)^{\frac{1}{1+ \sigma/4}} \\&\qquad \leq c + c \left( \fint_{B_{2R}}\left\lvert \X u \right\rvert^{p\left(x\right)}\right)^{\frac{\left(  1+ \frac{\sigma}{4} + \omega_p\left(4R\right)\right)}{1+ \sigma/4}} + c \left[ F\right]_{\rm{BMO}}^{\frac{\gamma_{1}\left( 1+ \frac{\sigma}{4} + \omega_p\left(4R\right)\right)}{\left( \gamma_{1}-1\right)\left(1+ \sigma/4\right) }} \\
		&\qquad \leq c + c \left( \fint_{B_{2R}}\left\lvert \X u \right\rvert^{p\left(x\right)}\right)^{\frac{\omega_p\left(4R\right)}{1+ \sigma/4}}\left( \fint_{B_{2R}}\left\lvert \X u \right\rvert^{p\left(x\right)}\right) +  c \left[ F\right]_{\rm{BMO}}^{\frac{\gamma_{1}\left( 1+ \frac{\sigma}{4} + \omega_p\left(4R\right)\right)}{\left( \gamma_{1}-1\right)\left(1+ \sigma/4\right) }} . 
	\end{align*}
	But now it is easy to see that we have 
	\begin{align*}
		\left( \fint_{B_{2R}}\left\lvert \X u \right\rvert^{p\left(x\right)}\right)^{\frac{\omega_p\left(4R\right)}{1+ \sigma/4}} \leq  cR^{-\frac{\omega_p\left(4R\right)Q}{1+ \sigma/4}}\left(\int_{B_{2R}}\left\lvert \X u \right\rvert^{p\left(x\right)}\right)^{\frac{\omega_p\left(4R\right)}{1+ \sigma/4}} \stackrel{\eqref{log holder condition thm}}{\leq} C(L, K_{0}). 
	\end{align*}
	This implies 
	\begin{align*}
		R^{Q}&\left( \fint_{B_{R}} \left\lvert \X u  \right\rvert^{p_{2}\left(1+ \sigma/4\right)}\right)^{\frac{1}{1+ \sigma/4}} \notag \\ &\leq cR^{Q} + cR^{Q}\left[ F\right]_{\rm{BMO}}^{\frac{\gamma_{1}\left( 1+ \frac{\sigma}{4} + \omega_p\left(4R\right)\right)}{\left( \gamma_{1}-1\right)\left(1+ \sigma/4\right) }} + c\left( \int_{B_{2R}}\left\lvert \X u \right\rvert^{p\left(x\right)}\right) \notag \\ &\leq cR^{Q}  + c\left( \int_{B_{2R}}\left( 1 + \left\lvert \X u \right\rvert^{p_{2}}\right) \right). 
	\end{align*}
	Combining with \eqref{est -5}, we arrive at 
	\begin{multline}\label{est -2}
		cR^Q\fint_{B_{R}} \left\lvert \X u \right\rvert^{p_{2}}\left\lvert \log \right\rvert^{\frac{p_{2}}{p_{2}-1}} \left( e + \frac{\left\lvert \X u \right\rvert^{p_{2}}}{\fint_{B_{R}}\left\lvert \X u  \right\rvert^{p_{2}}}\right)  \\ \leq cR^{Q}  + c\left( \int_{B_{2R}}\left( 1 + \left\lvert \X u \right\rvert^{p_{2}}\right) \right). 
	\end{multline}
	For the second term  in \eqref{est -4}, note that we have 
	\begin{align}\label{est -6}
		R^{Q}&\log^{\left( \frac{p_{2}}{p_{2}-1}\right)} \left( e + \fint_{B_{R}}\left\lvert \X u  \right\rvert^{p_{2}}\right)\fint_{B_{R}} \left\lvert \X u \right\rvert^{p_{2}} \notag \\&=  R^{Q}\log \left( e + \fint_{B_{R}}\left\lvert \X u  \right\rvert^{p_{2}}\right)\log^{\left( \frac{1}{p_{2}-1}\right)} \left( e + \fint_{B_{R}}\left\lvert \X u  \right\rvert^{p_{2}}\right)\fint_{B_{R}} \left\lvert \X u \right\rvert^{p_{2}}\notag \\ 
		&\leq R^{Q}\log \left( e + \fint_{B_{R}}\left\lvert \X u  \right\rvert^{p_{2}}\right)\log^{\left( \frac{1}{\gamma_{1}-1}\right)} \left( e + \fint_{B_{R}}\left\lvert \X u  \right\rvert^{p_{2}}\right)\left( e + \fint_{B_{R}} \left\lvert \X u \right\rvert^{p_{2}} \right) \notag \\ &\stackrel{\eqref{growth of log}}{\leq} c\left(\gamma_{1}\right)\sigma^{-\frac{1}{\gamma_{1}-1}}R^{Q}\log \left( e + \fint_{B_{R}}\left\lvert \X u  \right\rvert^{p_{2}}\right)\left( e + \fint_{B_{R}} \left\lvert \X u \right\rvert^{p_{2}} \right)^{1+ \frac{\sigma}{4}}  \notag \\
		&\leq cR^{Q}\log \left( e + \fint_{B_{R}}\left\lvert \X u  \right\rvert^{p_{2}}\right) \left( e^2 + \fint_{B_{R}} \left\lvert \X u \right\rvert^{p_{2}\left( 1 + \frac{\sigma}{4}\right)} \right)\notag \\
		&\leq c\log \left( e + \fint_{B_{R}}\left\lvert \X u  \right\rvert^{p_{2}}\right) \left( e^2R^{Q} + \int_{B_{R}} \left\lvert \X u \right\rvert^{p_{2}\left( 1 + \frac{\sigma}{4}\right)} \right). 
	\end{align}
	Now we use \eqref{higher integrability estimate} with $R$ replaced by $R_1$ and with $\xi= \left(F\right)_{_{R_1}(x_0)}$  to obtain 
	\begin{multline*}
		\int_{B_{\frac{R_1}{2}}(x_0)}\left\lvert \X u \right\rvert^{p\left(x\right)\left(1 + \sigma\right)}\ \mathrm{d}x
		\leq c R_1^{-Q\sigma}\left( \int_{B_{R_1}(x_0)}\left\lvert \X u \right\rvert^{p\left(x\right)}\ \mathrm{d}x + 1\right)^{1+\sigma} \\ + c \left( 1+  \int_{B_{R_1}(x_0)}\left\lvert F - \left(F\right)_{_{R_1}(x_0)}\right\rvert^{\frac{\gamma_{1}\left( 1+\sigma\right)}{\gamma_{1}-1}}\ \mathrm{d}x \right).
	\end{multline*}
	This together with \eqref{p2 by px} and the facts that $B_{4R} \subset \subset B_{\frac{R_1}{2}}(x_0)$, $0<\sigma\leq \sigma_0$ implies 
	\begin{align}\label{est -1}
		\int_{B_{R}} \left\lvert \X u  \right\rvert^{p_{2}\left(1+ \sigma/4\right)} &\leq c\int_{B_{R}} \left\lvert \X u  \right\rvert^{p(x)\left(1+ \sigma\right)} + cR^{Q} \notag \\
		&\leq C\left( K_{0}, R_{1},  \left[ F\right]_{\rm{BMO}}\right) + cR^{Q}.
	\end{align}
	Combining \eqref{est -6} and \eqref{est -1}, we have 
	\begin{align}\label{est -3}
		R^{Q}\log^{\left( \frac{p_{2}}{p_{2}-1}\right)} \left( e + \fint_{B_{R}}\left\lvert \X u  \right\rvert^{p_{2}}\right)\fint_{B_{R}} \left\lvert \X u \right\rvert^{p_{2}} \leq c\log \left( e + \fint_{B_{R}}\left\lvert \X u  \right\rvert^{p_{2}}\right). 
	\end{align}
	But using \eqref{est -1} again, we deduce 
	\begin{align}\label{est 0}
		\log \left( e + \fint_{B_{R}}\left\lvert \X u  \right\rvert^{p_{2}}\right) &\leq c \log \left( e + c\left( K_{0}, R_{1},  \left[ F\right]_{\rm{BMO}}\right)\left(4R\right)^{-Q}\right) \notag \\&\leq c \log \left( \frac{1}{4R}\right) \leq c \log \left( \frac{1}{4R}\right) \left( \int_{B_{2R}}\left( 1 + \left\lvert \X u \right\rvert^{p_{2}}\right) \right). 
	\end{align}
	Combining \eqref{est -3} and \eqref{est 0}, we conclude 
	\begin{align}\label{est 1}
		R^{Q}\log^{\left( \frac{p_{2}}{p_{2}-1}\right)} \left( e + \fint_{B_{R}}\left\lvert \X u  \right\rvert^{p_{2}}\right)\fint_{B_{R}} \left\lvert \X u \right\rvert^{p_{2}} \leq c \log \left( \frac{1}{4R}\right) \left( \int_{B_{2R}}\left( 1 + \left\lvert \X u \right\rvert^{p_{2}}\right) \right). 
	\end{align}	
	Now plugging \eqref{est 1} and \eqref{est -2} into \eqref{est -4}, and plugging the resulting estimate back into  \eqref{estimate of the log term}, we arrive at 
	\begin{align}\label{final esti log term}
		\left\lvert I_{3}\right\rvert &\leq cR^{Q} + c\omega_p\left(4R\right)\log \left(\frac{1}{4R}\right)\left( \int_{B_{2R}}\left( 1 + \left\lvert \X u \right\rvert^{p_{2}}\right) \right).  
	\end{align}
	Now, using \eqref{decay for Morrey bound},\eqref{estimate 1}, \eqref{estimate of a term}, \eqref{estimate of F term} and \eqref{final esti log term}, we obtain 
	\begin{align}\label{final estimate for Morrey bound}
		\int_{B_{\rho}} &\left\lvert\X u  \right\rvert^{p_{2}} \notag \\&\begin{multlined}[b]
			\leq c \left[ \left(\frac{\rho}{R}\right)^{Q} + \varepsilon + \omega_{a}\left(4R\right)+ \omega_p\left(4R\right)\log \left(\frac{1}{4R}\right) \right] \left( \int_{B_{2R}}\left( 1 + \left\lvert \X u \right\rvert^{p_{2}}\right) \right) \\ + cR^{Q}+c \int_{B_{R}} \left\lvert F - \left(F\right)_{R} \right\rvert^{\frac{p_{2}}{p_{2}-1}}.
		\end{multlined}
	\end{align} 
	Using the fact that $\rho \mapsto \mathfrak{p}\left(\rho\right)$ is nondecreasing and $\mathfrak{p}(R)=p_{2}$, we have 
	\begin{align*}
		\int_{B_{\rho}} \left(  \left\lvert \X u  \right\rvert^{\mathfrak{p}\left( \rho\right)} + 1\right)\ \mathrm{d}x \leq c\int_{B_{\rho}} \left(  \left\lvert \X u  \right\rvert^{p_{2}} + 1\right)\ \mathrm{d}x.
	\end{align*}
	Combining this with \eqref{final estimate for Morrey bound} and setting $\Phi \left( \rho\right):= \int_{B_{\rho}} \left(  \left\lvert \X u \right\rvert^{\mathfrak{p}\left( \rho\right)} + 1\right)\ \mathrm{d}x, $ we arrive at 
	\begin{align*}
		\Phi \left( \rho\right) \leq c \left[ \left(\frac{\rho}{R}\right)^{Q} + \varepsilon + \omega_{a}\left(4R\right)+ \omega_p\left(4R\right)\log \left(\frac{1}{4R}\right) \right] 	\Phi \left( 2R\right) + cR^{Q}.  
	\end{align*}
	Now we make use of \eqref{vanishing log holder condition thm} and the standard iteration lemma (see Lemma 5.13 in \cite{giaquinta-martinazzi-regularity}) and choose $\varepsilon >0$ and $R_{1}$ sufficiently small to establish our  claim. \smallskip
	
	Now we prove part $(c).$ We set $\alpha := \min \left\lbrace \alpha_{1}, \alpha_{2}, \alpha_{3}/2\right\rbrace $. Since everything is a local estimate, we would show that $\X u$ is H\"{o}lder continuous in a ball of small enough radius, but otherwise arbitrary. Now we prove the H\"{o}lder continuity of the $\X u$ in $B_{\tilde{R}}\left(x_{0}\right),$ where $\tilde{R}>0$ is the radius given by Claim \ref{Morrey bound claim}. Given $\theta>0$, we choose $0 < \rho_{0} < 1$ such that $\left(2\rho_{0}\right)^{\frac{1}{1+\theta}} < \tilde{R}/\left(64\right)^{2}.$ Precise choice of $\theta$ will be made later. Set 
	$$ r= r \left( \rho \right) := \left(2\rho\right)^{\frac{1}{1+\theta}} \qquad \text{ and } \qquad \tau := \frac{\alpha \beta}{8\left( Q+\beta\right)}. $$ Clearly, for any $0 < \rho < \rho_{0},$ we have 
	$r\left( \rho \right) < \tilde{R}/\left( 64\right)^{2}.$ Thus, by Claim \ref{Morrey bound claim}, we have 
	\begin{align}\label{Morrey decay}
		\int_{B_{r \left( \rho \right)}\left(x_{c}\right)}\left\lvert \X u  \right\rvert^{\mathfrak{p}\left(r \left( \rho \right)\right)} \leq c \left(r \left( \rho \right)\right)^{Q - \tau }.
	\end{align}
	As before, we find $v \in HW^{1, \mathfrak{p}\left(r \left( \rho \right)\right)}\left( B_{\left(r \left( \rho \right)\right)}\right)$ such that 
	\begin{align}\label{frozen system for comparison Holder}
		\left\lbrace \begin{aligned}
			\operatorname{div}_{\mathbb{H}}( a(\bar{x}_{\text{max}}) \lvert \X v  \rvert^{p\left(\bar{x}_{\text{max}}\right)-2} \X v ) ) &=0 &&\text{ in }  B_{r \left( \rho \right)}, \\
			v&=u &&\text{ on } \origpartial B_{r \left( \rho \right)}, 
		\end{aligned}\right. 
	\end{align} 
	where $\bar{x}_{\text{max}} \in \overline{B_{4r \left( \rho \right)}}$ is a point where $p\left(\bar{x}_{\text{max}}\right)= \mathfrak{p}\left(r \left( \rho \right)\right).$ Hence, as before in \eqref{estimate 1}, we have 
	\begin{align}\label{estimate 1 for holder}
		\gamma \int_{B_{r}} \left( \lvert \X u \rvert^{2} + \lvert \X v \rvert^{2}\right)^{\frac{\mathfrak{p}\left(r\right)-2}{2}}\left\lvert \X u - \X v\right\rvert^{2} \leq I_{1} + I_{2} + I_{3}. 
	\end{align}
	As before in \eqref{estimate of a term},  we have 
	\begin{align}\label{estimate of I1 holder}
		\left\lvert I_{1}\right\rvert &\leq c\omega_{a}\left(4r\right)  \int_{B_{r}}\left\lvert \X u  \right\rvert^{\mathfrak{p}\left(r\right)} \stackrel{\eqref{Morrey decay}}{\leq} cr^{\left(\alpha + Q -\tau \right) }. 
	\end{align}
	We estimate $I_{2}$ slightly differently now. We deduce 
	\begin{align}\label{estimate of F term holder}
		\left\lvert I_{2}\right\rvert &\leq \int_{B_{r}}\left\lvert F - \left(F\right)_{r}\right\rvert\left\lvert \X u  - \X v\right\rvert \notag\\
		&\leq c \left( \int_{B_{r}} \left\lvert F - \left(F\right)_{r} \right\rvert^{\frac{\mathfrak{p}\left(r\right)}{\mathfrak{p}\left(r\right)-1}}\right)^{\frac{\mathfrak{p}\left(r\right)-1}{\mathfrak{p}\left(r\right)}}\left( \int_{B_{r}} \left\lvert \X u  \right\rvert^{\mathfrak{p}\left(r\right)} \right)^{\frac{1}{\mathfrak{p}\left(r\right)}} \notag\\
		&\stackrel{\eqref{Morrey decay}}{\leq} 
		c \left( \int_{B_{r}} \left\lvert F - \left(F\right)_{r} \right\rvert^{\frac{\mathfrak{p}\left(r\right)}{\mathfrak{p}\left(r\right)-1}}\right)^{\frac{\mathfrak{p}\left(r\right)-1}{\mathfrak{p}\left(r\right)}}r^{\frac{Q-\tau}{\mathfrak{p}\left(r\right)}} \notag\\
		&\leq c\left[ F\right]_{C^{0, \alpha}} r^{\left(\alpha + Q - \frac{\tau}{\mathfrak{p}\left(r\right)}\right)} \leq c\left[ F\right]_{C^{0, \alpha}} r^{\left(\alpha + Q - \tau \right)} . 
	\end{align}
	For $I_{3},$ once again as in \eqref{final esti log term}, we have 
	\begin{align}\label{estimate of I11 holder}
		\left\lvert I_{3}\right\rvert \leq cr^{Q} + c \omega_p\left(4r\right)\log \left(\frac{1}{4r}\right)\int_{B_{2r}} 1+\left\lvert \X u\right\rvert^{\mathfrak{p}\left(r\right)}\stackrel{\eqref{Morrey decay}}{\leq} cr^{\left(\alpha + Q - \tau \right)}. 
	\end{align}	
	Thus, combining \eqref{estimate of I1 holder}, \eqref{estimate of F term holder} and \eqref{estimate of I11 holder}, we obtain
	\begin{align}\label{estimate of I holder}
		\int_{B_{r}} \left( \lvert \X u \rvert^{2} + \lvert \X v \rvert^{2}\right)^{\frac{\mathfrak{p}\left(r\right)-2}{2}}\left\lvert \X u - \X v\right\rvert^{2} \leq cr^{\left(\alpha + Q -\tau \right)}. 
	\end{align}
	Now if $\mathfrak{p}\left(r\right) \geq 2,$ we have 
	\begin{align*}
		\int_{B_{r}} \left\lvert \X u - \X v \right\rvert^{\mathfrak{p}\left(r\right)} &\leq c \int_{B_{r}} \left( \lvert \X u \rvert^{2} + \lvert \X v \rvert^{2}\right)^{\frac{\mathfrak{p}\left(r\right)-2}{2}}\left\lvert \X u - \X v \right\rvert^{2} \leq cr^{\left(\alpha + Q -\tau \right)}. 
	\end{align*}
	On the other hand, if $1 < \mathfrak{p}\left(r\right)< 2,$ we have 
	\begin{align*}
		\int_{B_{r}} &\left\lvert \X u - \X v\right\rvert^{\mathfrak{p}\left(r\right)} \\
		&\leq \left(\int_{B_{r}} \left( \lvert \X u \rvert^{2} + \lvert \X v \rvert^{2}\right)^{\frac{\mathfrak{p}-2}{2}}\left\lvert \X u - \X v\right\rvert^{2}\right)^{\frac{\mathfrak{p}}{2}}\left( \int_{B_{r}} \left( \lvert \X u \rvert^{2} + \lvert \X v \rvert^{2}\right)^{\frac{\mathfrak{p}}{2}}\right)^{\frac{2-\mathfrak{p}}{2}} \\
		&\leq cr^{\frac{\mathfrak{p}}{2}\left(\alpha + Q -\tau \right)}\left( \int_{B_{r}} \lvert \X u  \rvert^{\mathfrak{p}(r)}\right)^{\frac{2-\mathfrak{p}(r)}{2}} \leq c r^{\left(\frac{\alpha}{2} + Q -\tau \right)}. 
	\end{align*}
	Thus, in both cases, we arrive at 
	\begin{align}\label{comparison estimate holder}
		\int_{B_{r}} \left\lvert \X u - \X v\right\rvert^{\mathfrak{p}\left(r\right)} \leq c r^{\left(\frac{\alpha}{2} + Q -\tau \right)}. 
	\end{align}
	Thus, using minimality of integral means, we deduce 
	\begin{align*}
		\int_{B_{\rho}}& \left\lvert \X u  - \left( \X u \right)_{B_{\rho}} \right\rvert^{\mathfrak{p}\left(r\right)} \\
		&\leq c\int_{B_{\rho}} \left\lvert \X v - \left( \X v \right)_{B_{\rho}} \right\rvert^{\mathfrak{p}\left(r\right)} + c\int_{B_{\rho}} \left\lvert \X u  - \X v \right\rvert^{\mathfrak{p}\left(r\right)} \\
		&\stackrel{\eqref{oscillation estimate constant homogeneous}}{\leq} c\left( \frac{\rho}{r}\right)^{\beta \mathfrak{p}\left(r\right)}\rho^{Q} \left( \fint_{B_{r}} \lvert \X u  \rvert^{\mathfrak{p}}\right) + c\int_{B_{\rho}} \left\lvert \X u  - \X v \right\rvert^{\mathfrak{p}\left(r\right)} \\
		&\stackrel{\eqref{comparison estimate holder}}{\leq}c\left( \frac{\rho}{r}\right)^{\beta \mathfrak{p}\left(r\right)}\rho^{Q} \left( \fint_{B_{r}} \lvert \X u  \rvert^{\mathfrak{p}}\right) + c r^{\left(\frac{\alpha}{2} + Q -\tau \right)} \\
		&\stackrel{\eqref{Morrey decay}}{\leq}c\left( \frac{\rho}{r}\right)^{\beta \mathfrak{p}\left(r\right)}\rho^{Q} r^{-\tau}+ c r^{\left(\frac{\alpha}{2} + Q -\tau \right)} \leq c\left( \frac{\rho}{r}\right)^{\beta}\rho^{Q} r^{-\tau}+ c r^{\left(\frac{\alpha}{2} + Q -\tau \right)}.  
	\end{align*}
	Now choose $\theta = \alpha/2\left(Q+\beta\right)$ and observe that this makes the powers of $\rho$ the same in both terms, which is equal to $Q + \mu,$ for some (explicitly computable) $\mu >0.$ Thus we have 
	\begin{align*}
		\int_{B_{\rho}} \left\lvert \X u  - \left( \X u \right)_{B_{\rho}} \right\rvert^{\gamma_{1}} &\leq c\rho^{Q}\left( \fint_{B_{\rho}} \left\lvert \X u  - \left( \X u \right)_{B_{\rho}} \right\rvert^{p_{2}\left(r\right)} \right)^{\frac{\gamma_{1}}{\mathfrak{p}\left(r\right)}} \\
		&\leq c\rho^{Q}\left( c\rho^{\mu} \right)^{\frac{\gamma_{1}}{\mathfrak{p}\left(r\right)}} \leq c\rho^{\left( Q + \frac{\gamma_{1}\mu}{\gamma_{2}}\right)},
	\end{align*}
	as $\mathfrak{p}\left(r\right) < \gamma_{2}$ and $\rho_{0} < 1.$ Since this holds for any $x_{c} \in \overline{B_{\tilde{R}}\left(x_{0}\right)}$ and any $0 < \rho < \rho_{0},$ by Campanato's characterization (see Theorem 5.5, Page 78 in \cite{giaquinta-martinazzi-regularity}), this implies the H\"{o}lder continuity of $\X u $ and completes the proof. \qed 
	
	\section{Borderline continuity}\label{continuity section}
	\subsection{Basic setup}\label{basic setup for comparison}

	Let $p$ be as in Theorem \ref{main theorem} and $\gamma:= \min \{\gamma_1, \gamma_2', 2Q'  \}.$ Note that by Proposition \ref{Lorenz-Zygmund-Sobolev_implies_logholder}, this implies $p$ is log-H\"{o}lder continuous and thus, we can use Theorem \ref{higher integrability 2}. 
	We would use the following shorthands. 
	\begin{align}
		\mathfrak{P}_{R, \theta}&:= R \log \left(\frac{1}{R}\right)\left(  \fint_{B_{R}}\left\lvert  \X p \right\rvert^{\theta} \right)^{\frac{1}{\theta}},  \label{mathfrak p} 
	\end{align}
	Now set 
	\begin{align}\label{q and theta}
		\theta = \frac{2Q\left(1 + \sigma_{0}\right)}{\sigma_{0}\left(Q+2\right)+2}, 
	\end{align}
	where $\sigma_0$ is given by Theorem \ref{higher integrability 2}.
	Note that this implies 
	\begin{align*}
		1 < \theta < Q \qquad \text{ and } \qquad \theta^{\ast}:= \frac{Q\theta}{Q-\theta} = \frac{2\left(1 + \sigma_{0}\right)}{\sigma_{0}}. 
	\end{align*}
	We set 
	\begin{align}
		\bar{\sigma}:= \min \left\lbrace \gamma_1-1,  \sigma_{1} \right\rbrace. \label{sigma bar}
	\end{align}
	Let $R_0$ be as in Theorem \ref{higher integrability 2}. Now, $p$ is log-H\"{o}lder continuous, we can choose $0<\bar{R}\leq R_0/4$ so small such that 
	\begin{align}\label{smallness of wp}
		\omega_p\left(\bar{4R}\right) \leq \min \left\lbrace \frac{ \bar{\sigma}}{4}, \frac{\gamma \bar{\sigma}}{4\gamma^{'}\left(2 + \bar{\sigma}\right)}, \frac{\gamma_1}{2Q}\right\rbrace 
	\end{align}
	and we can use Proposition \ref{def of Lorentz Zygmund sum} to further reduce $\bar{R}$ if necessarily to have  
	\begin{align}\label{smallness of mathfrak a+mathfrak p}
		\mathfrak{P}_{R, \theta} \leq 1 \qquad \text{ for all } 0 < R \leq \frac{\bar{R}}{4}.
	\end{align}
	Now fix a point $x_{0} \in \Omega$ such that $B_{4\bar{R}}(x_{0}) \subset  \Omega.$ Clearly $\bar{R}$ depends on $Q, \gamma_{1},$ $\gamma_{2}, \nu, L$ and $K_0$.  From now on, all balls would be centered at $x_{0}$ if the center is not specified and we omit writing the center of the balls. Now the smallness in \eqref{smallness of wp} implies that for any $0<R \leq \bar{R}/4$
	$$p_{R} - \inf_{B_{2R}} p (\cdot) \leq \omega_p(\bar{2R}) \leq  \bar{\sigma} \leq \gamma_1 \bar{\sigma} \implies p_{R} \leq p(x) (1+\bar{\sigma}), \mbox{ for all } x\in \overline{B_{2R}}.$$  
	Here $p_{R}:= \fint_{B_{R}} p. $
	Let $w$ be as in Theorem \ref{main theorem}. Then by Theorem \ref{higher integrability 2}, we have $w\in HW^{1,p_{R}}\left(B_{R}\right)$.  Hence, let $v \in w+HW_{0}^{1,p_{R}}\left(B_{R}\right)$ be the unique minimizer of 
	\begin{align*}
		\inf \left\lbrace \int_{B_{R}} \frac{1}{p_{R}}\left\lvert \X v\right\rvert^{p_{R}}\ \mathrm{d}x: v \in w + HW_{0}^{1,p_{R}}\left(B_{R}\right) \right\rbrace.  
	\end{align*}
	Note that $v$ is a weak solution to 
	\begin{align}\label{equation double frozen}
		\operatorname{div}_{\mathbb{H}}\left(\left\lvert \X v \right\rvert^{p_{R}-2} \X v \right) &= 0 &&\text{ in } B_{R}
	\end{align}	
	and $v-w$ is an admissible test function. Also, 
	$w$ is a weak solution to 
	\begin{align}\label{equation frozen}
		\operatorname{div}_{\mathbb{H}}\left(\left\lvert \X w \right\rvert^{p(x) -2} \X w \right) &= 0 &&\text{ in } B_{R}. 
	\end{align}	
	\subsection{Comparison estimates}
	\begin{lemma}\label{comparison lemma 2}
		Let $w$,  $p_{R}$ and $v$ be as defined in Section \ref{basic setup for comparison} and  let 
		\begin{align}\label{def lambda1}
			\lambda_{1}&:= 1 + \left(  \fint_{B_{4R}}\left\lvert  \X w \right\rvert^{\gamma} \right)^{\frac{1}{\gamma}}. 
		\end{align}  Then we have the following estimates.
		\begin{enumerate}[(i)]
			\item \begin{align}\label{comparison lemma 2 main estimate}
				\fint_{B_{R}} \left\lvert V_{p_{R}}\left( \X w\right) - V_{p_{R}}\left( \X v\right)\right\rvert^{2} \leq 	C \mathfrak{P}_{R, \theta}^{2} \lambda_{1}^{p_{R}}.
			\end{align}
			\item If $p_{R} >2,$ we have 
			\begin{align}\label{comparison lemma 2 bar p bigger than 2}
				\fint_{B_{R}} \left\lvert  \X w - \X v\right\rvert^{p_{R}} \leq  C \mathfrak{P}_{R, \theta} ^{2} \lambda_{1}^{p_{R}}.
			\end{align}
			\item If $p_{R} \leq 2,$ we have 
			\begin{align}\label{comparison lemma 2 bar p smaller than 2}
				\fint_{B_{R}} \left\lvert  \X w - \X v\right\rvert^{p_{R}} \leq C \mathfrak{P}_{R, \theta} ^{p_{R}} \lambda_{1}^{p_{R}}.
			\end{align}
		\end{enumerate}
		Here $ C = C \left(n, \gamma_{1}, \gamma_{2}, L \right)$ is a positive constant.
	\end{lemma}
	\begin{proof}
		Clearly, $w-v \in HW^{1,p_{R}}_0(B_{R})$. On the other hand, it follows from \eqref{smallness of wp} that  $p_2(1+\bar{\sigma}/4)< p(x)(1+\bar{\sigma})$, for all $x\in B_{R}$, where $p_2= \sup_{B_{R}} p(\cdot)$. This implies $w\in HW^{1,p_2(1+\bar{\sigma})}(B_{2R})$, since $w\in HW^{1,p(\cdot)(1+\bar{\sigma})}(B_{2R})$. Thus it follows from a well-known global higher integrability result (see Thoerem 6.8 in \cite{Giusti_DCV}, Theorem 4.4 in \cite{Kinnunen_higherintegrability_metricmeasure}) that $v\in HW^{1,p(\cdot)}(B_{R})$. As a consequence we have $w-v\in HW_0^{1, p(\cdot)}(B_{R})$. Hence, plugging in $w-v$ in the weak formulations \eqref{equation frozen}, \eqref{equation double frozen}, we deduce 
		\begin{align*}
			\int_{B_{R/2}} \left\langle \left\lvert \X w \right\rvert^{p\left(x\right)-2} \X w - \left\lvert \X v \right\rvert^{p_{R/2}-2} \X v, \X w -\X v \right\rangle  = 0. 
		\end{align*}	
		Thus we have 
		\begin{align*}
			\fint_{B_{R}} &\left\lvert V_{p_{R}}\left(\X w \right) - V_{p_{R}}\left(\X v \right) \right\rvert^{2} \\
			&\leq c  \fint_{B_{R}}  \left\langle \left\lvert \X w \right\rvert^{p_{R/2}-2} \X w - \left\lvert \X v \right\rvert^{p_{R/2}-2} \X v, \X w -\X v \right\rangle \\
			&= c\fint_{B_{R}}   \left\langle  \left\lvert \X w \right\rvert^{p\left(x\right)-2} \X w -\left\lvert \X w \right\rvert^{p_{R/2}-2} \X w, \X v -\X w \right\rangle 
			\\&\leq c\fint_{B_{R}}\left\lvert \left\lvert \X w \right\rvert^{p\left(x\right)-2} \X w - \left\lvert \X w \right\rvert^{p_{R/2}-2} \X w \right\rvert \left\lvert \X w -\X v \right\rvert 
			\\&=\fint_{B_{R}}\left\lvert \left\lvert \X w \right\rvert^{p\left(x\right)-p_{R}} -1 \right\rvert \left\lvert \X w \right\rvert^{p_{R/2}-1}  \left\lvert \X w -\X v \right\rvert:= I_{1}.  
		\end{align*}
		Now, for a.e. $x \in B_{R},$ with $\left\lvert \X w \right\rvert \neq 0,$ we have 
		\begin{align*}
			\left\lvert \left\lvert \X w \right\rvert^{p\left(x\right)-p_{R}} -1 \right\rvert = \left\lvert p\left(x\right)-p_{R} \right\rvert\left\lvert \X w \right\rvert^{\theta_{x}\left( p\left(x\right)-p_{R}\right)}  \left\lvert \log \left\lvert \X w \right\rvert \right\rvert
		\end{align*}
		for some $\theta_{x} \in (0,1).$	Now if $p_{R} \geq 2,$ we have 
		\begin{align*}
			\left\lvert \X w -\X v \right\rvert \leq c \left\lvert \X w \right\rvert^{\frac{2-p_{R}}{2}}\left\lvert V_{p_{R}}\left(\X w \right) - V_{p_{R}}\left(\X v \right) \right\rvert. 
		\end{align*}
		We use the shorthand $\tilde{p} = p_{R}.$ Now using Young's inequality, we have 
		\begin{align*}
			I_{1} &\leq c \fint_{B_{R}}\left\lvert p\left(x\right)-\tilde{p} \right\rvert\left\lvert \X w \right\rvert^{\frac{\tilde{p}}{2} + \theta_{x}\left( p\left(x\right)-\tilde{p}\right)}  \left\lvert \log \left\lvert \X w \right\rvert \right\rvert \left\lvert V_{\tilde{p}}\left(\X w \right) - V_{\tilde{p}}\left(\X v \right) \right\rvert 
			\\&\begin{aligned}
				\leq  c \fint_{B_{R}}\left\lvert p\left(x\right)-\tilde{p} \right\rvert^{2}\left\lvert \X w \right\rvert^{\tilde{p} + 2\theta_{x}\left( p\left(x\right)-\tilde{p}\right)} & \left\lvert \log \left\lvert \X w \right\rvert \right\rvert^{2} \\&+ \varepsilon\fint_{B_{R}}\left\lvert V_{\tilde{p}}\left(\X w \right) - V_{\tilde{p}}\left(\X v \right) \right\rvert^{2}
			\end{aligned}
		\end{align*}
		Now if $p_{R} < 2,$ then we have 
		\begin{align*}
			I_{1} &\begin{multlined}[t]
				\stackrel{\eqref{v estimate p less 2}}{\leq} c \fint_{B_{R}}\left\lvert p\left(x\right)-\tilde{p} \right\rvert\left\lvert \X w \right\rvert^{\tilde{p} -1 + \theta_{x}\left( p\left(x\right)-\tilde{p}\right)}  \left\lvert \log \left\lvert \X w \right\rvert \right\rvert \left\lvert V_{\tilde{p}}\left(\X w \right) - V_{\tilde{p}}\left(\X v \right) \right\rvert^{\frac{2}{\tilde{p}}}  \\
				+ c\fint_{B_{R}}\left\lvert p\left(x\right)-\tilde{p} \right\rvert\left\lvert \X w \right\rvert^{\frac{\tilde{p}}{2} + \theta_{x}\left( p\left(x\right)-\tilde{p}\right)}  \left\lvert \log \left\lvert \X w \right\rvert \right\rvert \left\lvert V_{\tilde{p}}\left(\X w \right) - V_{\tilde{p}}\left(\X v \right) \right\rvert
			\end{multlined}
			\\&\begin{aligned}
				\leq  c \fint_{B_{R}}\left\lvert p\left(x\right)-\tilde{p} \right\rvert^{{\tilde{p}}^{'}}&\left\lvert \X w \right\rvert^{\tilde{p} + {\tilde{p}}^{'}\theta_{x}\left( p\left(x\right)-\tilde{p}\right)}  \left\lvert \log \left\lvert \X w \right\rvert \right\rvert^{{\tilde{p}}^{'}}\\&\begin{multlined}
					+c \fint_{B_{R}}\left\lvert p\left(x\right)-\tilde{p} \right\rvert^{2}\left\lvert \X w \right\rvert^{\tilde{p} + 2\theta_{x}\left( p\left(x\right)-\tilde{p}\right)}  \left\lvert \log \left\lvert \X w \right\rvert \right\rvert^{2} \\+ \varepsilon\fint_{B_{R/2}}\left\lvert V_{\tilde{p}}\left(\X w \right) - V_{\tilde{p}}\left(\X v \right) \right\rvert^{2}.
				\end{multlined}
			\end{aligned}
		\end{align*} 
		Now note that we have 
		\begin{align*}
			\left\lvert p\left(x\right)-\tilde{p}\right\rvert  &\leq \omega_p\left(R\right) \leq \omega_p\left(4\bar{R}\right) \intertext{ and }
			\left\lvert t\theta_x \left( p\left(x\right)-\tilde{p}\right)\right\rvert  &\leq t\omega_p\left(R\right) \leq \max \left\lbrace 2, \gamma_{1}^{'} \right\rbrace\omega_p\left(4\bar{R}\right) \leq \frac{\tilde{p}}{2},  
		\end{align*}
		for $t=2, {\tilde{p}}^{'}.$ Thus, if $\left\lvert \X w (x)\right\rvert < e,$ we deduce 
		\begin{align*}
			\left\lvert p\left(x\right)-\tilde{p} \right\rvert^{t}&\left\lvert \X w (x)\right\rvert^{\tilde{p} + t\theta_{x}\left( p\left(x\right)-\tilde{p}\right)}  \left\lvert \log \left\lvert \X w (x)\right\rvert \right\rvert^{t} \leq c
		\end{align*}
		for $t=2, {\tilde{p}}^{'}.$ Hence, we deduce 
		\begin{align*}
			\int_{B_{R}\cap \left\lbrace \left\lvert \X w (x)\right\rvert < e \right\rbrace} 	\left\lvert p\left(x\right)-\tilde{p} \right\rvert^{t}&\left\lvert \X w \right\rvert^{\tilde{p} + t\theta_{x}\left( p\left(x\right)-\tilde{p}\right)}  \left\lvert \log \left\lvert \X w \right\rvert \right\rvert^{t} \\&\leq c \left\lvert B_{R}\cap \left\lbrace \left\lvert \X w \right\rvert < e \right\rbrace\right\rvert \leq c \left\lvert B_{R}\right\rvert. 
		\end{align*}
		On the other hand, setting $E_{R} := B_{R}\cap \left\lbrace \left\lvert \X w (x)\right\rvert \geq  e \right\rbrace,$ we estimate  
		\begin{align*}
			\int_{E_{R}}& 	\left\lvert p\left(x\right)-\tilde{p} \right\rvert^{t}\left\lvert \X w \right\rvert^{\tilde{p} + t\theta_{x}\left( p\left(x\right)-\tilde{p}\right)}  \left\lvert \log \left\lvert \X w \right\rvert \right\rvert^{t} \\ &\leq 	c \int_{E_{R}} 	\left\lvert p\left(x\right)-\tilde{p} \right\rvert^{t}\left( 1 + \left\lvert \X w \right\rvert\right)^{\tilde{p} + t\omega_p\left(R\right)}  \left\lvert \log \left( e + \left( 1 + \left\lvert \X w \right\rvert \right)^{\tilde{p} + t\omega_p\left(R\right)}\right)\right\rvert^{t}.
		\end{align*}
		Now we observe that 
		\begin{align*}
			\frac{1}{\left\lvert B_{R}\right\rvert}&\int_{E_{R}} 	\left\lvert p\left(x\right)-\tilde{p} \right\rvert^{t}\left( 1 + \left\lvert \X w \right\rvert\right)^{\tilde{p} + t\omega_p\left(R\right)}  \left\lvert \log \left( e + \left( 1 + \left\lvert \X w \right\rvert \right)^{\tilde{p} + t\omega_p\left(R\right)}\right)\right\rvert^{t}  \\&\leq c \fint_{B_{R}} 	\left\lvert p\left(x\right)-\tilde{p} \right\rvert^{t}\left( 1 + \left\lvert \X w \right\rvert\right)^{\tilde{p} + t\omega_p\left(R\right)}  \left\lvert \log \left( e + \left( 1 + \left\lvert \X w \right\rvert \right)^{\tilde{p} + t\omega_p\left(R\right)}\right)\right\rvert^{t}. 
		\end{align*}
		Set $\vartheta := \omega_p (R)$. Using Holder inequality, we deduce 
		\begin{align*}
			\fint_{B_{R}}& 	\left\lvert p\left(x\right)-\tilde{p} \right\rvert^{t}\left( 1 + \left\lvert \X w \right\rvert\right)^{\tilde{p} + t\vartheta}  \left\lvert \log \left( e + \left( 1 + \left\lvert \X w \right\rvert \right)^{\tilde{p} + t\vartheta}\right)\right\rvert^{t} \\
			&\begin{multlined}
				\leq R^{t}\left(\fint_{B_{R}} \left\lvert \frac{p\left(x\right)-\tilde{p}}{R}\right\rvert^{t(1 + \frac{4}{\bar{\sigma}})} \right)^{\frac{\bar{\sigma}}{ 4+ \bar{\sigma}}} \\
				\times \left( \fint_{B_{R}}\left( 1 + \left\lvert \X w \right\rvert\right)^{\left(\tilde{p} + t\vartheta\right)\left(1 + \frac{\bar{\sigma}}{4}\right)}   \log^{t\left(1 + \frac{\bar{\sigma}}{4}\right)} \left( e + \left( 1 + \left\lvert \X w \right\rvert \right)^{\tilde{p} + t\vartheta}\right) \right)^{\frac{4}{4+\bar{\sigma}}}. 
			\end{multlined}
		\end{align*}
		Now we apply Lemma \ref{Llog beta L lemma} with the choices 
		\begin{align*}
			\begin{aligned}
				f = \left( 1 + \left\lvert \X w \right\rvert\right)^{\left( \tilde{p} + t\vartheta\right)\left(1 + \frac{\bar{\sigma}}{4}\right)}, \ \beta = t\left( 1 + \frac{\bar{\sigma}}{4}\right),\\  \sigma = \frac{1}{\left( 1 + \frac{\bar{\sigma}}{4}\right)},\  \tau = \left(Q+1\right)\left( 1 + \frac{\bar{\sigma}}{2}\right) \mbox{ and } \zeta = \frac{1+ \frac{\bar{\sigma}}{2}}{1+ \frac{\bar{\sigma}}{4}},
			\end{aligned}
		\end{align*}
		to the last integral on the right above to deduce  
		\begin{align}
			\fint_{B_{R}}	\left\lvert p\left(x\right)-\tilde{p} \right\rvert^{t}&\left( 1 + \left\lvert \X w \right\rvert\right)^{\tilde{p} + t\vartheta}  \left\lvert \log \left( e + \left( 1 + \left\lvert \X w \right\rvert \right)^{\tilde{p} + t\vartheta}\right)\right\rvert^{t} \notag \\
			&\begin{multlined}[b]
				\leq cR^{t}\log^{t} \left( \frac{1}{R}\right)\left(\fint_{B_{R}} \left\lvert \frac{p\left(x\right)-\tilde{p}}{R}\right\rvert^{t(1 + \frac{4}{\bar{\sigma}})} \right)^{\frac{\bar{\sigma}}{ 4+ \bar{\sigma}}} \\
				\times \left( 1 + R^{\left(Q+1\right)\left( 1 + \frac{\bar{\sigma}}{2}\right)}\fint_{B_{R}}\left( 1 + \left\lvert \X w \right\rvert\right)^{\left(\tilde{p} + t\vartheta\right)\left(1 + \frac{\bar{\sigma}}{4}\right)}\right)^{t} \\
				\times \left( \fint_{B_{R}}\left( 1 + \left\lvert \X w \right\rvert\right)^{\left(\tilde{p} + t\vartheta \right)\left(1 + \frac{\bar{\sigma}}{2}\right)}   \right)^{\frac{2}{2+\bar{\sigma}}}. 
			\end{multlined}
		\end{align}
		Now we have  
		\begin{align}
			R^{\left(Q+1\right)\left( 1 + \frac{\bar{\sigma}}{2}\right)}\fint_{B_{R}}\left( 1 + \left\lvert \X w \right\rvert\right)^{\left(\tilde{p} + t\vartheta\right)\left(1 + \frac{\bar{\sigma}}{4}\right)} &\leq c, \label{est p-tilde wp 1} \\
			\left( \fint_{B_{R}}\left( 1 + \left\lvert \X w \right\rvert\right)^{\left(\tilde{p} + t\vartheta\right)\left(1 + \frac{\bar{\sigma}}{2}\right)}   \right)^{\frac{2}{2+\bar{\sigma}}} &\leq c\lambda_{1}^{\tilde{p}} \label{est p-tilde wp 2}.
		\end{align}
		Indeed, using the shorthand $\varTheta: = \gamma' (1+\bar{\sigma}/4)$, we note that 
		\begin{align}\label{est p-tilde wp}
			t(1+\frac{\bar{\sigma}}{4}) \vartheta \leq \varTheta \omega_p(R) \stackrel{\eqref{smallness of wp}}{\leq} \frac{\gamma \bar{\sigma}}{4}.
		\end{align}
		On the other hand it is easy to see that $\tilde{p}(1+\frac{\bar{\sigma}}{4}) \leq p(x)(1+\frac{\bar{\sigma}}{4} + \omega_p(2R))$ for any $x\in \overline{B_R}.$ Again using \eqref{smallness of wp} we conclude that $\left(1+\varTheta\right) \omega_{p}(2R) \leq \frac{\bar{\sigma}}{4}$.
		So we estimate using \eqref{higher integrability estimate RHS 0}
		\begin{align*}
			R^{\left(Q+1\right)\left( 1 + \frac{\bar{\sigma}}{2}\right)}  &\fint_{B_{R}}\left( 1 + \left\lvert \X w \right\rvert\right)^{\left(\tilde{p} + t\vartheta\right)\left(1 + \frac{\bar{\sigma}}{4}\right)} \\
			&\leq R^{\left(Q+1\right)\left( 1 + \frac{\bar{\sigma}}{2}\right)} \fint_{B_{R}}\left( 1 + \left\lvert \X w \right\rvert\right)^{p(x)\left(1 + \frac{\bar{\sigma}}{4}+ \left(1+\varTheta \right)\omega_p(2R) \right)}\\
			&\leq c R^{\left(Q+1\right)\left( 1 + \frac{\bar{\sigma}}{2}\right)}  \left(1+ \frac{1}{R^Q} \left \lVert \X w \right \rVert _{L^\gamma\left(B_{2R}\right)}\right)^{\left(1+\frac{\bar{\sigma}}{4} +\ \left(1+\varTheta \right)\omega_p(2R)\right)}  \\
			&\leq c R^{\left(Q+1\right)\left( 1 + \frac{\bar{\sigma}}{2}\right)}  \left(\frac{1}{R^Q} K_0\right)^{\left(1+\frac{\bar{\sigma}}{4} + \left(1+\varTheta \right) \omega_p(2R)\right)} \\ 
			&\leq c R^{\left(Q+1\right)\left( 1 + \frac{\bar{\sigma}}{2}\right)}  \left(\frac{1}{R^{Q+1}} \right)^{\left(1+\frac{\bar{\sigma}}{4} + \left(1+\varTheta \right) \omega_p(2R)\right)} \leq c.
		\end{align*}
		This proves \eqref{est p-tilde wp 1}. Similarly, we can prove \eqref{est p-tilde wp 2}. Now, using \eqref{est p-tilde wp 1}, \eqref{est p-tilde wp 2} and the fact that $t \geq 2$ in either case (since $\tilde{p}^{'} \geq 2$ if $\tilde{p} \leq 2$), we arrive at \eqref{comparison lemma 2 main estimate}. 
		Now \eqref{comparison lemma 2 bar p bigger than 2} follows from the fact $p_{R}>2$, \eqref{constant cv} and \eqref{comparison lemma 2 main estimate}. On the other hand, if $\tilde{p}\leq 2,$ we have 
		\begin{align}\label{penultimate est}
			\fint_{B_{R}}& \left\lvert  \X v - \X w\right\rvert^{\tilde{p}} \notag\\
			&\leq c\fint_{B_{R}}\left( \left\lvert  \X v\right\rvert^{2} + \left\lvert  \X w\right\rvert^{2} \right)^{\frac{\tilde{p}\left(\tilde{p}-2\right)}{2}} \left\lvert  \X v - \X w\right\rvert^{\tilde{p}}\left( \left\lvert  \X v\right\rvert^{2} + \left\lvert  \X w\right\rvert^{2}\right)^{\frac{\tilde{p}\left(2- \tilde{p}\right)}{2}} \notag\\
			&\leq c\fint_{B_{R}}\left\lvert V_{\tilde{p}}\left( \X v\right) - V_{\tilde{p}}\left( \X w\right)\right\rvert^{\tilde{p}}\left( \left\lvert  \X v\right\rvert^{2} + \left\lvert  \X w\right\rvert^{2}\right)^{\frac{\tilde{p}\left(2- \tilde{p}\right)}{2}} \notag\\
			&\leq c\left( \fint_{B_{R}}\left\lvert V_{\tilde{p}}\left( \X v\right) - V_{\tilde{p}}\left( \X w\right)\right\rvert^{2}\right)^{\frac{\tilde{p}}{2}}\left( \fint_{B_{R}}\left( \left\lvert  \X v\right\rvert^{2} + \left\lvert  \X w\right\rvert^{2}\right)^{\frac{\tilde{p}}{2}}\right)^{\frac{2-\tilde{p}}{2}} \notag\\
			&\leq c\left( \fint_{B_{R}}\left\lvert V_{\tilde{p}}\left( \X v\right) - V_{\tilde{p}}\left( \X w\right)\right\rvert^{2}\right)^{\frac{\tilde{p}}{2}}\left( \fint_{B_{R}}\left( \left\lvert  \X v\right\rvert^{\tilde{p}} + \left\lvert  \X w\right\rvert^{\tilde{p}}\right)\right)^{\frac{2-\tilde{p}}{2}}.
		\end{align}
		Now observe that by using Theorem \ref{higher integrability 2}, we estimate 
		\begin{align}\label{estimate of p bar energy of w}
			\fint_{B_{R}} \left\lvert \X w  \right\rvert^{\tilde{p}}
			& \leq \left(\fint_{B_{R}} \left\lvert \X w  \right\rvert^{\tilde{p}\left(1+ \bar{\sigma}\right)}\right)^\frac{1}{1+\bar{\sigma}}  \notag \\
			& \leq c \left(1+ \fint_{B_{2R}}\left\lvert \X w  \right\rvert^{\gamma}\ \mathrm{d}x \right)^{\frac{\tilde{p}}{\gamma}} \stackrel{\eqref{def lambda1} }{\leq} c \lambda_1^{\tilde{p}}. 
		\end{align}
		
		Similarly we can estimate $\fint_{B_{R}} \left\lvert \X v  \right\rvert^{\tilde{p}}$ to conclude 
		\begin{align*}
			\fint_{B_{R}}\left( \left\lvert  \X v\right\rvert^{\tilde{p}} + \left\lvert  \X w\right\rvert^{\tilde{p}}\right) \leq c \lambda_{1}^{\tilde{p}}.
		\end{align*}
		Plugging this in \eqref{penultimate est}, we arrive at \eqref{comparison lemma 2 bar p smaller than 2}. This concludes the proof.
	\end{proof}
	
	\subsection{Comparison with bounds}\label{comparison with bounds}
	Let $\lambda \geq 1$. Fix $x_0\in \Omega$. Consider,  the radius $\bar{R}$ which is defined in Section \ref{basic setup for comparison}.
	Let $\delta \in (0,1/4)$, $s>0$ and for $j=1,2, \dots $, define the following quantities 
	\begin{align}
		r_j: &= \delta^{j-1} \frac{\bar{R}}{16}, \quad B_j:= B_{r_j}\left(x_0\right), \quad sB_j:= B_{sr_j} (x_0) \label{seq of balls} \\
		p_{j}:&= \fint_{B_{r_j}} p. \label{exponents}
	\end{align}
	Also, let $v_j\in w+ HW_0^{1, p_{r_j}}(B_j)$ solve \eqref{equation double frozen} with $R$ being replaced by $r_j$. We have the following linearized comparison.
	\begin{lemma}\label{linearized comparison-hom}
		Let  $\theta$ be as in 
		\eqref{q and theta}. Suppose for $j\geq 2$,  we have  
		\begin{equation}\label{upper bound of w}
			1 + \max \left\lbrace \left( \fint_{4B_{j}} \left \lvert \X w \right \rvert ^\gamma \right)^{\frac{1}{\gamma}}, \left( \fint_{4B_{j-1}} \left \lvert \X w \right \rvert ^\gamma \right)^{\frac{1}{\gamma}} \right\rbrace \leq \lambda 
		\end{equation}
		and for some constants $A, B \geq 1,$ the estimates 
		\begin{equation}\label{both side bounds on v}
			%		\sup\limits_{\frac{1}{2}B_{j}} \left\lvert \X v_{j} \right\rvert \leq A\lambda \qquad \text{ and } \qquad 
			\frac{\lambda}{B} \leq \left\lvert \X v_{j-1} \right\rvert 
			\leq A\lambda \quad \text{ in } B_{j}
		\end{equation}
		hold. Then there exists a constant $C_1 \equiv C_{1}\left( Q, \gamma_1, \gamma_2, L, A, B, \delta \right) >0$ such that 
		\begin{align}\label{grad w -grad v pbigger2}
			\left( \fint_{B_{j}} \left\lvert \X  w - \X  v_{j} \right\rvert^{\gamma}\right)^{\frac{1}{\gamma}} \leq C_{1} \mathfrak{P}_{r_{j-1}, \theta} \lambda.
		\end{align}
	\end{lemma}
	\begin{remark}\label{Baroni error}
		This lemma here is similar in spirit to \cite[Lemma 3.5]{Baroni_pxStein}. However, the proof of \cite[Lemma 3.5]{Baroni_pxStein} is wrong. The point is that the last inequality in the estimate (3.33) \emph{does not} follow from the estimate (3.29). In deducing (3.29), $w$ is the solution to the constant exponent $\bar{p}_{R/4}$-Laplacian equation and in (3.33), $\tilde{w}$ solves a constant exponent $\bar{p}_{\sigma^{-1}R/4}$-Laplacian equation.  So, using (3.29), the estimate we obtain is 
		\begin{multline*}
			\fint_{B_{\sigma^{-1}R/4}} \left\lvert V_{\bar{p}_{\sigma^{-1}R/4}}\left(Du\right) - V_{\bar{p}_{\sigma^{-1}R/4}}\left(D\tilde{w}\right)\right\rvert^{2}\ \mathrm{d}x \\ \leq c \left[ \mathfrak{A}^{2}_{\sigma^{-1}R, q} + \mathfrak{B}^{2}_{\sigma^{-1}R, q}\right]\lambda^{\bar{p}_{\sigma^{-1}R/4}}. 
		\end{multline*}
		Since there is no reason why the averages 
		$$ \bar{p}_{\sigma^{-1}R/4} = \fint_{B_{\sigma^{-1}R/4}} p\left(x\right)\ \mathrm{d}x  \quad \text{ and } \quad  \bar{p}_{R/4} = \fint_{B_{R/4}}p\left(x\right)\ \mathrm{d}x  \text{ are the same,}$$ this is not the claimed estimate, which is  
		\begin{align*}
			\fint_{B_{R/4}} \left\lvert V_{\bar{p}_{R/4}}\left(Du\right) - V_{\bar{p}_{R/4}}\left(D\tilde{w}\right)\right\rvert^{2}\ \mathrm{d}x  \leq c \left[ \mathfrak{A}^{2}_{\sigma^{-1}R, q} + \mathfrak{B}^{2}_{\sigma^{-1}R, q}\right]\lambda^{\bar{p}_{R/4}}, 
		\end{align*}
		as $\bar{p} = \bar{p}_{R/4}$, as defined in the beginning of the proof of Lemma 3.5. Note that just defining $\bar{p} = \bar{p}_{\sigma^{-1}R/4}$ instead can not fix the proof, as then every estimate in the proof where $w$ is involved does not hold anymore, e.g. the estimate of the terms (II), (III), (IV) in the pages 439-440. 
	\end{remark}
	\begin{proof}
		We divide the proof in several cases depending on $p_{j-1}$ and $p_{j}$. Note that, by our choice $\gamma \leq \min\lbrace  p_{j-1}, p'_{j-1}, p_{j}, p'_{j} \rbrace$. Also, by the definitions of $\mathfrak{P}_{r_j, \theta}$ in \eqref{mathfrak p}  we have the following relations
		\begin{align}\label{estimates of j by j-1 for p}
			\mathfrak{P}_{r_j, \theta} \leq  \delta^{\frac{Q-\theta}{\theta}}\mathfrak{P}_{r_{j-1}, \theta}.
		\end{align}
		Since in the case $p_{j}\leq 2$, \eqref{grad w -grad v pbigger2} follows from \eqref{comparison lemma 2 bar p smaller than 2} and H\"{o}lder inequality, we will only focus on the case $p_{j}>2$. The cases of interest are the following.  
		\begin{align}
			&2<p_{j-1} \leq p_{j}. \label{Case I} \tag{Case I}\\
			& 2<p_{j} \leq p_{j-1}. \label{Case II} \tag{Case II}\\ 
			&p_{j-1}\leq 2 < p_{j}. \label{Case III} \tag{Case III}
		\end{align}
		Note that if $p_{j-1}$ and  $p_{j}$ satisfy the condition of \eqref{Case III}, then from the relation $p_{j}-p_{j-1} \leq \omega_p(\bar{R}) < 1/4$ we have $p_{j-1}\leq 2< p_{j}<p_{j-1}+1/4$. This implies $p_{j-1}>7/4$ and 
		$$\frac{1+p_{j}}{p_{j}} \leq  \frac{5/4+p_{j-1}}{2} <p_{j-1}.$$
		Set
		\begin{align} \label{def alpha}
			\alpha: = \frac{1}{2} \left(\frac{1}{p_{j-1}}+ \frac{p_{j}}{1+p_{j}}\right).
		\end{align}
		Then clearly $\alpha$ satisfies 
		\begin{align}\label{bound of alpha}
			\frac{1}{p_{j-1}}< \alpha< \frac{p_{j}}{p_{j}+1}.
		\end{align}
		
		First, we have the following claim. 
		\begin{claim}\label{gen p bar-hom}
			For some $C \equiv C\left( Q, \gamma_1, \gamma_2, L, A, B, \delta \right) >0,$ we have 
			\begin{align}
				\fint_{B_{j}} \left\lvert \X  v_{j-1} - \X  v_{j} \right\rvert^{p_{j}} &\leq C \mathfrak{P}_{r_{j-1}, \theta} \lambda ^{p_{j}}, \qquad \mbox{ if } \eqref{Case I} \mbox{ holds and } \label{grad v_1 -grad v_2 Case I}\\
				\fint_{B_{j}} \left\lvert \X  v_{j-1} - \X  v_{j} \right\rvert^{p_{j}} &\leq C \left(\mathfrak{P}_{r_{j-1}, \theta}\right)^{\alpha p_{j-1}} \lambda ^{p_{j}}, \qquad \mbox{ if } \eqref{Case III} \text{ holds}, \label{grad v_1 -grad v_2 Case III}. 
			\end{align}
			
		\end{claim}
		\emph{Proof of Claim \ref{gen p bar-hom}:} Clearly,
		\begin{align}\label{xvj-xvj-1 gen}
			\fint_{B_{j}} \lvert \X  v_{j-1} - \X  v_{j}\rvert^{p_{j}}  \leq c\fint_{B_{j}}\lvert \X w - \X v_{j-1}\rvert^{p_{j}} + c \fint_{B_{j}}\lvert \X w  - \X v_{j}\rvert^{p_{j}}. 
		\end{align}
		It is enough to estimate the first term. The rest follows by  \eqref{comparison lemma 2 bar p bigger than 2} and the fact that $1< \alpha p_{j-1}<2$.  To prove \eqref{grad v_1 -grad v_2 Case I}, we deduce
		\begin{align}\label{estimate for caseI}
			&\fint_{B_{j}} \left\lvert \X  w - \X  v_{j-1} \right\rvert^{p_{j}} \notag\\ 
			&\leq  c\fint_{B_{j}} \left\lvert \X  w - \X  v_{j-1} \right\rvert^{p_{j-1}/2} \left( \lvert\X  w\rvert + \lvert \X  v_{j-1} \rvert\right)^{p_{j}-p_{j-1}/2}  \notag\\
			&\stackrel{\eqref{both side bounds on v}}{\leq} c\delta^{-Q/2}  \left( \fint_{B_{j-1}} \left\lvert \X w - \X  v_{j-1} \right\rvert^{p_{j-1}}\right)^{\frac{1}{2}} \left( \lambda^{2p_{j}-p_{j-1}}+ \fint_{B_{j}} \left\lvert \X w \right\rvert^{p_{j}+ p_{j}-p_{j-1}}\right)^{\frac{1}{2}} \notag\\ 
			& \stackrel{\eqref{comparison lemma 2 bar p bigger than 2}}{\leq} c \mathfrak{P}_{r_{j-1},\theta} \lambda^{p_{j-1}/2} \left( \lambda^{2p_{j}-p_{j-1}}+ \fint_{B_{j}} \left(1+ \left\lvert \X w \right\rvert\right)^{p_{j}\left( 1+\omega_p(\bar{R})\right)}\right)^{\frac{1}{2}} \notag\\
			& \stackrel{\eqref{higher integrability estimate RHS 0}}{\leq} c \mathfrak{P}_{r_{j-1}1,\theta} \lambda^{p_{j-1}/2} \left( \lambda^{2p_{j}-p_{j-1}}+ \lambda^{p_{j}\left(1+\omega_p(\bar{R})\right)}\right)^{\frac{1}{2}}.
		\end{align}
		For \eqref{grad v_1 -grad v_2 Case III}, we deduce 
		\begin{align}\label{estimate for caseIII}
			& \fint_{B_{j}} \left\lvert \X  w - \X  v_{j-1} \right\rvert^{p_{j}} \notag\\ 
			&\leq  c\fint_{B_{j}} \left\lvert \X  w - \X  v_{j-1} \right\rvert^{\alpha p_{j-1}} \left( \lvert\X  w\rvert + \lvert \X  v_{j-1} \rvert\right)^{p_{j}-\alpha p_{j-1}}  \notag\\
			&\begin{aligned}
				\stackrel{\eqref{both side bounds on v}}{\leq} c\delta^{-\alpha Q}  &\left( \fint_{B_{j-1}} \left\lvert \X w - \X  v_{j-1} \right\rvert^{p_{j-1}}\right)^{\alpha}\\ &\qquad \qquad \left( \lambda^{p_{j}-\alpha p_{j-1}}+ \left(\fint_{B_{j}} \left\lvert \X w \right\rvert^{(p_{j}+ \frac{\alpha}{1-\alpha} (p_{j}-p_{j-1})}\right)^{1-\alpha}\right)
			\end{aligned} \notag\\ 
			& \stackrel{\eqref{bound of alpha}, \eqref{comparison lemma 2 bar p smaller than 2}}{\leq} c \left(\mathfrak{P}_{r_{j-1},\theta} \lambda\right)^{\alpha p_{j-1}} \left( \lambda^{p_{j}-\alpha p_{j-1}}+ \left(\fint_{B_{j}} \left(1+ \left\lvert \X w \right\rvert\right)^{p_{j}\left(1+\omega_p(\bar{R})\right)}\right)^{1-\alpha}\right) \notag\\
			& \stackrel{\eqref{higher integrability estimate RHS 0}}{\leq} c \mathfrak{P}^{\alpha p_{j-1}}_{r_{j-1},\theta} \lambda^{\alpha p_{j-1}} \left( \lambda^{p_{j}-\alpha p_{j-1}}+ \lambda^{p_{j}(1-\alpha)\left(1+\omega_p(\bar{R})\right)}\right).
		\end{align}
		Since $\lambda \geq 1,$ $p_{j-1}\leq p_{j}$ and $\lambda^{\omega_p(\bar{R})} \leq C(Q, \gamma_1, \gamma_2, L, A, B, \delta) $ by the choice of radius,  estimate \eqref{estimate for caseI} and \eqref{estimate for caseIII} implies \eqref{grad v_1 -grad v_2 Case I} and \eqref{grad v_1 -grad v_2 Case III}, respectively. This proves the claim.\smallskip

		\noindent Next,  we estimate 
		\begin{align}\label{est xw-xv2}
			&\left(\fint_{B_{j}} \lvert \X  w - \X  v_{j}\rvert^{p'_j}\right)^{\frac{1}{p'_j}} \notag \\
			&\qquad \stackrel{\eqref{both side bounds on v}}{\leq} c\lambda^{2-p_j} \left(\fint_{B_{j}} \left\lvert \X v_{j-1} \right\rvert^{p'_j(p_j-2)} \left\lvert \X  w - \X  v_{j}\right\rvert^{p'_j}\right)^{\frac{1}{p'_j}} \notag \\ 
			& \qquad \leq c\lambda^{2-p_j} \left(\fint_{B_{j}} \left\lvert \X v_{j-1} -\X v_j \right\rvert^{p'_j(p_j-2)} \left\lvert \X  w - \X  v_{j}\right\rvert^{p'_j}\right)^{\frac{1}{p'_j}} \notag \\
			& \qquad \qquad  + c\lambda^{2-p_j} \left(\fint_{B_{j}} \left\lvert \X v_j \right\rvert^{p'_j(p_j-2)} \left\lvert \X  w - \X  v_{j}\right\rvert^{p'_j}\right)^{\frac{1}{p'_j}} =: I_1+I_2 
		\end{align}
		Now we deduce estimate \eqref{grad w -grad v pbigger2}. We first consider \eqref{Case I}. We have 
		\begin{align}\label{est of I1 for Case I}
			I_{1} &\leq c \lambda^{2-p_j}
			\left( \fint_{B_{j}} \left\lvert \X v_{j-1} - \X v_{j} \right\rvert^{p_j}\right)^{\frac{p_j-2}{p_j}}
			\left( \fint_{B_{j}} \left\lvert \X  w - \X  v_{j} \right\rvert^{p_j}\right)^{\frac{1}{p_j}} \notag \\
			&\stackrel{ \eqref{grad v_1 -grad v_2 Case I}, \eqref{comparison lemma 2 bar p bigger than 2}}{\leq} c\lambda^{2-p_j} \left( \mathfrak{P}_{r_{j-1}, \theta} \lambda^{p_j}\right)^{\frac{p_j-2}{p_j}} \left( \mathfrak{P}^2_{r_{j-1}, \theta} \lambda^{p_j}\right)^{\frac{1}{p_j}} \leq c \mathfrak{P}_{r_{j-1}, \theta} \lambda.
		\end{align}	
		Also, we estimate 
		\begin{align}\label{est of I2 for Case I}
			I_2 & = c\lambda^{2-p_j} \left(\fint_{B_{j}} \left\lvert \X v_{j} \right\rvert^{\left(p_j-2\right)p'_j/2}\left\lvert \X  w - \X  v_{j} \right\rvert^{p'_j} \left\lvert \X v_{j} \right\rvert^{\left(p_j-2\right)p'_j/2}\right)^{\frac{1}{p'_j}} \notag\\
			&\stackrel{\text{H\"{o}lder}}{\leq} c\lambda^{2-p_j}
			\left( \fint_{B_{j}} \left\lvert \X v_{j} \right\rvert^{p_j-2}\left\lvert \X  w - \X  v_{j} \right\rvert^{2} \right)^{\frac{1}{2}}
			\left( \fint_{B_{j}} \left\lvert \X v_{j} \right\rvert^{p_j}\right)^{\frac{p_j-2}{2p_j}}\notag \\
			&\stackrel{\eqref{both side bounds on v}}{\leq} c\lambda^{\frac{2-p_j}{2}} 
			\left( \fint_{B_{j}} \left( \left\lvert \X w \right\rvert + \left\lvert \X v_{j} \right\rvert \right)^{p_j-2}\left\lvert \X  w - \X  v_{j} \right\rvert^{2} \right)^{\frac{1}{2}}\notag \\
			&\leq c\lambda^{\frac{2-p_j}{2}} 
			\left(\fint_{B_j}\left\lvert V_{p_j}\left( \X w\right) - V_{p_j}\left( \X v_j\right)\right\rvert^{2}\right)^{\frac{1}{2}} \stackrel{\eqref{comparison lemma 2 main estimate}}{\leq} C \mathfrak{P}_{r_{j-1}, \theta} \lambda.
		\end{align}
		Plugging in \eqref{est of I1 for Case I} and \eqref{est of I2 for Case I} in \eqref{est xw-xv2} and using H\"older's inequality we arrive at \eqref{grad w -grad v pbigger2} for \eqref{Case I}. Next we derive \eqref{grad w -grad v pbigger2} for \eqref{Case II}. We have
		\begin{align}\label{est of I1 for Case II}
			I_{1} &\leq c \lambda^{2-p_j}
			\left( \fint_{B_{j}} \left\lvert \X v_{j-1} - \X v_{j} \right\rvert^{p_j}\right)^{\frac{p_j-2}{p_j}}
			\left( \fint_{B_{j}} \left\lvert \X  w - \X  v_{j} \right\rvert^{p_j}\right)^{\frac{1}{p_j}} \notag \\
			& \leq c \lambda^{2-p_j}
			\left( \fint_{B_{j}} \left\lvert \X w - \X v_{j-1} \right\rvert^{p_j}\right)^{\frac{p_j-2}{p_j}}
			\left( \fint_{B_{j}} \left\lvert \X  w - \X  v_{j} \right\rvert^{p_j}\right)^{\frac{1}{p_j}} \notag \\
			& \hspace{5cm}+ c \lambda^{2-p_j}  \left( \fint_{B_{j}} \left\lvert \X  w - \X  v_{j} \right\rvert^{p_j}\right)^{\frac{p_j-1}{p_j}} \notag \\
			&\stackrel{ \eqref{comparison lemma 2 bar p bigger than 2}}{\leq} c\lambda^{3-p_j} \left( \mathfrak{P}_{r_{j-1}, \theta}\right)^{\frac{2}{p_j}} \left( \fint_{B_{j}} \left\lvert \X w - \X v_{j-1} \right\rvert^{p_j}\right)^{\frac{p_j-2}{p_j}} + c \lambda \left( \mathfrak{P}_{r_{j-1}, \theta}\right)^{\frac{2}{p'_j}} .
		\end{align}
		Now we use the relation $2<p_j\leq p_{j-1}$ to estimate 
		\begin{align*}
			\fint_{B_{j}} &\left\lvert \X w - \X v_{j-1} \right\rvert^{p_j} \\
			&\leq c  \fint_{B_{j}} \left\lvert \X w - \X v_{j-1} \right\rvert^{2} \left(\lvert \X w\rvert+ \lvert \X v_{j-1} \rvert \right)^{p_j-2} \\ 
			& = 	c  \fint_{B_{j}} \left\lvert \X w - \X v_{j-1} \right\rvert^{2} \left(\lvert \X w\rvert+ \lvert \X v_{j-1} \rvert \right)^{p_{j-1}-2}  \left(\lvert \X w\rvert+ \lvert \X v_{j-1} \rvert \right)^{p_j-p_{j-1}} \\
			& \stackrel{\eqref{both side bounds on v}}{\leq} c \lambda^{p_j-p_{j-1}}	 \fint_{B_{j}} \left\lvert \X w - \X v_{j-1} \right\rvert^{2} \left(\lvert \X w\rvert+ \lvert \X v_{j-1} \rvert \right)^{p_{j-1}-2} \\
			& \leq c \lambda^{p_{j}-p_{j-1}}	 \fint_{B_{1}} \left\lvert V_{p_{j-1}}\left(\X w\right) - V_{p_{j-1}}\left(\X v_{j-1}\right) \right\rvert^{2} \stackrel{\eqref{comparison lemma 2 main estimate}}{\leq} c \mathfrak{P}_{r_{j-1},\theta}^2 \lambda^{p_j}.  
		\end{align*}
		Combing this estimate with \eqref{est of I1 for Case II} we derive 
		\begin{align}\label{final est of I1 for Case II}
			I_1 \leq  c \lambda \left( \mathfrak{P}_{r_{j-1},\theta}\right)^{\frac{2}{p'_j}} \leq c \lambda  \mathfrak{P}_{r_{j-1},\theta}. 
		\end{align}
		The estimation of $I_2$ for \eqref{Case II} is identical to the derivation \eqref{est of I2 for Case I}. Consequently, we establish \eqref{grad w -grad v pbigger2} by combining this two estimates with \eqref{est xw-xv2} and applying of H\"older's inequality. 
		
		Next we move on to prove \eqref{grad w -grad v pbigger2} under the assumption of \eqref{Case III}. Again in this case the estimation of $I_2$ will remain the same. We only estimate $I_1$.  We have 
		\begin{align}\label{est of I1 for Case III}
			I_{1} &\leq c \lambda^{2-p_j}
			\left( \fint_{B_{j}} \left\lvert \X v_{j-1} - \X v_{j} \right\rvert^{p_j}\right)^{\frac{p_j-2}{p_j}}
			\left( \fint_{B_{j}} \left\lvert \X  w - \X  v_{j} \right\rvert^{p_j}\right)^{\frac{1}{p_j}} \notag \\
			&\stackrel{ \eqref{grad v_1 -grad v_2 Case III}, \eqref{comparison lemma 2 bar p bigger than 2}}{\leq} c\lambda^{2-p_j} \left( \mathfrak{P}^{\alpha p_{j-1}}_{r_{j-1}, \theta} \lambda^{p_j}\right)^{\frac{p_j-2}{p_j}} \left( \mathfrak{P}^2_{r_{j-1}, \theta} \lambda^{p_j}\right)^{\frac{1}{p_j}} \leq c \mathfrak{P}_{r_{j-1}, \theta} \lambda.
		\end{align}	
		To derive the last estimate we used the fact $\alpha p_{j-1}>1$ from \eqref{bound of alpha} and the following computation. 
		\begin{align*}
			\alpha p_{j-1} \frac{p_j-2}{p_j}+ \frac{2}{p_j} - 1 = \frac{(\alpha p_{j-1} -1)( p_j-2)}{p_j}>0. 
		\end{align*}
		This completes the proof.  
	\end{proof}
	\subsection{Pointwise bound}
	We set $\texttt{data} = Q, \gamma_{1}, \gamma_{2}, \left\lVert p\right\rVert_{L^{1}\left( \Omega\right)} + \left\lVert \X p \right\rVert_{L^{\left(Q,1\right)}\log L\left( \Omega; \mathbb{R}^{2n}\right)}.$ 
	\begin{theorem}\label{pointwise bound thm} Let $w$ and $p$ be as in Theorem \ref{main theorem}. Then $\X w$ is locally bounded in $\Omega$. Furthermore, there exists a constant $c_{3} = c_{3}\left( \rm{\texttt{data}}  \right) \geq 1$ and a positive radius $$R_{1} = R_{1}\left( \texttt{data}, \left\lVert \left\lvert \X w \right\rvert^{p\left(\cdot\right)}\right\rVert_{L^{1}\left( \Omega\right)} \right) >0$$ such that if $0 < R \leq R_{1},$ then we have 
		\begin{align}\label{sup bound pointwise}
			\left\lvert \X w \left(x_{0}\right) \right\rvert \leq c_{3}\left( 1 + \left( \fint_{B\left(x_{0}, 16 R\right)}\left\lvert \X w \right\rvert^{\gamma} \right)^{\frac{1}{\gamma}} \right),  
		\end{align}
		whenever $B_{16R}(x_{0}) \subset \subset \Omega$ and $x_{0}$ is a Lebesgue point of $\X w.$
	\end{theorem}
	\begin{proof}
		Fix $x_{0} \in \Omega.$ In view of \eqref{oscillation estimate constant homogeneous} and \eqref{excess decay}, we can choose $\delta \in (0,1/4)$ small enough such that we have 
		\begin{align}
			\max \left\lbrace c_{2}\delta^{\beta}, 2C_{\gamma}\delta^{\beta}  \right\rbrace \leq \frac{1}{4^{Q+4}}. 
		\end{align} Also, in view of \eqref{sup estimate constant homogeneous} and \eqref{higher integrability estimate RHS 0}, we set $A = 1 + c_{1}c_{h}$ and set $B=2A.$ Now, we use this choice of $A$ and $B$ in Lemma \ref{linearized comparison-hom} to determine the constant $C_{1}.$ Let $\bar{R}$ is the radius defined in Section \ref{basic setup for comparison}. Finally, we choose a radius $0 < R_{1} < \bar{R}/16$ small enough such that we have 
		\begin{align}\label{smallness of radius for sum p to be small}
			\sum\limits_{j=1}^{\infty} \mathfrak{P}_{r_{j-1}, \theta} \leq \frac{\delta^{4Q}}{4C_{1}10^{2Q+4}}, 
		\end{align} for all $0 < R < R_{1},$ where $r_{j} = \delta^{j-1}R$ and  $B_{r_{j}}$ and $p_{j}$ are defined by \eqref{seq of balls} and \eqref{exponents}. Set 
		\begin{align}
			\lambda := H_{1} \left( 1 + \left( \fint_{B\left(x_{0}, 16 R\right)}\left\lvert \X w \right\rvert^{\gamma} \right)^{\frac{1}{\gamma}} \right), 
		\end{align} where $H_{1} = 10^{10} 16^{4Q}/ \delta^{2Q},$ is a large constant. For notational ease, for any nonnegative integer $j$, we denote 
		\begin{align}\label{notation for excess and average of u}
			a_j:= \lvert (\X w)_{B_j} \rvert \qquad \text{ and } \qquad E_j:=  \left(  \fint_{B_{j}} \left\lvert \X w- (\X w)_{B_j} \right\rvert^{\gamma} \right)^\frac{1}{\gamma}.\end{align}
		For $j \geq 1$ integer, we also define $v_{j} \in w + HW_{0}^{1,p_{j}}\left(B_{j}\right)$ to be the unique minimizer of 
		\begin{align*}
			\inf \left\lbrace \int_{B_{j}} \frac{1}{p_{j}}\left\lvert \X v\right\rvert^{p_{j}}\ \mathrm{d}x: v \in w + HW_{0}^{1,p_{j}}\left(B_{j}\right) \right\rbrace.
		\end{align*} Thus, for every $j\geq1,$ $v_{j}$ solve \eqref{equation double frozen} with $R$ being replaced by $r_j$
		Now we shall prove that we have 
		\begin{align}\label{main control of composites}
			a_{j} + E_{j} \leq \lambda \qquad \text{ for all } j \geqslant 1. 
		\end{align}
		Clearly, \eqref{main control of composites} implies \eqref{sup bound pointwise}. Indeed, when $x_0$ is a Lebesgue point of $\X w$, we have \begin{equation*}
			\lvert \Xu(x_0) \rvert = \lim_{j \to \infty } a_{j} \leq  \lambda.   
		\end{equation*}
		Also, by our choice of $H_{1},$ \eqref{main control of composites} clearly holds for $j=1,2 $ and indeed, we have 
		\begin{align}\label{initial bound}
			a_{j} + E_{j} \leq \frac{\lambda \delta^{(j-1)Q}}{1000} \mbox{ for } j=1,2.
		\end{align} So we just need to show that if \eqref{main control of composites} holds for all $1 \leq j \leq i,$ then it also holds for $j=i+1.$ To show this, we begin by noting that if \eqref{main control of composites} holds for all $ 1 \leq j \leq  i,$ then we have 
		\begin{align}\label{Ind j}
			1 + 	\max \left\lbrace \left ( \fint_{4B_{j-1}} \lvert \X w \rvert^{\gamma} \right)^\frac{1}{\gamma},  \left ( \fint_{4B_{j}} \lvert \X w \rvert^{\gamma} \right)^\frac{1}{\gamma}\right\rbrace \leq \lambda \qquad \text{ for all } 2 \leq j \leq i.  
		\end{align}
		Now, by our choice of $A$ and $\lambda$ and minimality of all $v_{j}$, \eqref{Ind j} implies 
		\begin{align}
			\sup\limits_{2B_{j}} \left\lvert \X v_{j-1} \right\rvert &\stackrel{\eqref{sup estimate constant homogeneous}}{\leq} c_{1} \fint_{B_{j-1}} \left\lvert \X v_{j-1}\right\rvert \notag \\&\leq c_{1} \left( \fint_{B_{j-1}} \left\lvert \X v_{j-1}\right\rvert^{p_{j-1}} \right)^{\frac{1}{p_{j-1}}} \notag \\&\leq c_{1} \left( \fint_{B_{j-1}} \left\lvert \X w\right\rvert^{p_{j-1}} \right)^{\frac{1}{p_{j-1}}} \stackrel{\eqref{higher integrability estimate RHS 0}}{\leq} c_{1}c_{h} \left[ 1 + \left ( \fint_{4B_{j}} \lvert \X w \rvert^{\gamma} \right)^\frac{1}{\gamma}\right] \leq A\lambda. \label{pointwise sup bounds j}
		\end{align}
		Now, by our choice of $\delta$ and \eqref{oscillation estimate constant homogeneous}, \eqref{pointwise sup bounds j} implies  
		\begin{align}\label{pointwise inf bounds j}
			\inf\limits_{B_{j}}	\left\lvert \X v_{j-1} \right\rvert \geq  A\lambda - \frac{1}{2}A\lambda \geq \lambda /2A   \qquad \text{ for all } 2 \leq j \leq i.
		\end{align} 
		Now \eqref{pointwise sup bounds j} and \eqref{pointwise inf bounds j} implies that we can apply Lemma \ref{linearized comparison-hom}. Thus, for every $2\leq j \leq i,$ we have 
		\begin{align}
			E_{j+1} &= \left(  \fint_{B_{j+1}} \left\lvert \X w- (\X w)_{B_{j+1}} \right\rvert^{\gamma} \right)^\frac{1}{\gamma} \notag \\
			&\leq 2\left(  \fint_{B_{j+1}} \left\lvert \X w- (\X v_{j})_{B_{j+1}} \right\rvert^{\gamma} \right)^\frac{1}{\gamma} \notag \\
			&\leq 2\left(  \fint_{B_{j+1}} \left\lvert \X v_{j}- (\X v_{j})_{B_{j+1}} \right\rvert^{\gamma} \right)^\frac{1}{\gamma} + 2 \left( \fint_{B_{j+1}} \left\lvert \X  w - \X  v_{j} \right\rvert^{\gamma}\right)^{\frac{1}{\gamma}} \notag \\
			&\leq \frac{1}{4} \left(  \fint_{B_{j}} \left\lvert \X v_{j}- (\X v_{j})_{B_{j}} \right\rvert^{\gamma} \right)^\frac{1}{\gamma} + 2\delta^{-\frac{Q}{\gamma}}\left( \fint_{B_{j}} \left\lvert \X  w - \X  v_{j} \right\rvert^{\gamma}\right)^{\frac{1}{\gamma}} \notag \\
			&\leq \frac{1}{2} \left(  \fint_{B_{j}} \left\lvert \X v_{j}- (\X w)_{B_{j}} \right\rvert^{\gamma} \right)^\frac{1}{\gamma} + 4\delta^{-\frac{Q}{\gamma}}\left( \fint_{B_{j}} \left\lvert \X  w - \X  v_{j} \right\rvert^{\gamma}\right)^{\frac{1}{\gamma}} \notag \\
			&\leq \frac{1}{2}E_{j} + 4\delta^{-\frac{Q}{\gamma}}C_{1} \mathfrak{P}_{r_{j-1}, \theta} \lambda. \label{excess decay at j}
		\end{align}
		Summing this inequality as $j$ ranges from $1$ to $i,$ we deduce 
		\begin{align*}
			\sum\limits_{j=3}^{i+1} E_{j} \leq \frac{1}{2} \sum\limits_{j=2}^{i} E_{j} + 4\delta^{-\frac{Q}{\gamma}}C_{1}\lambda \left( \sum\limits_{j=1}^{i}\mathfrak{P}_{r_{j-1}, \theta} \right). 
		\end{align*}	
		This implies 
		\begin{align}\label{sum of Ej estimate}
			\sum\limits_{j=2}^{i+1} E_{j} \leq  2E_{2} + 8\delta^{-\frac{Q}{\gamma}}C_{1} \left( \sum\limits_{j=1}^{i}\mathfrak{P}_{r_{j-1}, \theta} \right)\lambda \stackrel{\eqref{smallness of radius for sum p to be small}, \eqref{initial bound}}{\leq} \frac{\lambda \delta^{Q}}{100}. 
		\end{align}
		Thus, writing the difference of averages as a telescoping sum, we have 
		\begin{align*}
			a_{i+1} - a_{2} &= \sum\limits_{j=2}^{i} \left( a_{j+1} -a_{j}\right) \\ &\leq \sum\limits_{j=2}^{i} \fint_{B_{j+1}} \left\lvert \X w - \left( \X w \right)_{B_{j}}\right\rvert \\
			&\leq \delta^{-Q} \sum\limits_{j=1}^{i} \fint_{B_{j}} \left\lvert \X w - \left( \X w \right)_{B_{j}}\right\rvert \leq \delta^{-Q} \sum\limits_{j=1}^{i} E_{j} \stackrel{\eqref{sum of Ej estimate}}{\leq} \frac{\lambda}{100}. 
		\end{align*}
		This clearly implies \begin{align*}
			a_{j+1} + E_{j+1} \leq a_{1} + \frac{\lambda}{100} + E_{j+1} \leq a_{1} + \frac{\lambda}{100} +\sum\limits_{j=1}^{i+1} E_{j}  \stackrel{\eqref{initial bound}, \eqref{sum of Ej estimate}}{\leq} \lambda.  
		\end{align*}
		This proves \eqref{main control of composites} and completes the proof. 
	\end{proof}
	\subsection{Continuity}
	\textbf{Proof of Theorem \ref{main theorem}:} We divide the proof in three steps.\smallskip 
	
	\noindent\textbf{Step 1:} \underline{Boundedness estimate} The first conclusion is just Theorem \ref{pointwise bound thm} and estimate \eqref{sup bound homogeneous} follows from the estimate \eqref{sup bound pointwise} via  a well-known interpolation and covering argument.\smallskip 
	
	\noindent\textbf{Step 2:} \underline{Continuity} Now we prove continuity of $\X w$. We want to show $\X w$ is the locally uniform limit of a net of continuous maps, defined by the averages 
	$$ x \mapsto \left( \X w\right)_{B(x, \rho)}.$$ To do this, we pick any $R>0$ and $x_{0} \in \Omega$ such that $B\left(x_{0}, R\right) \subset \Omega$. Let $A \geq 1$ be a fixed number. We show that for every 
	$\varepsilon > 0,$ there exists a radius $0 < r_{\varepsilon} \leq R/1000 $, depending only on $\rm{\texttt{data}}, \varepsilon$ and $A$ such that if  
	\begin{align*}
		\sup\limits_{B\left(x_{0}, R/2\right)} \left\lvert \X w \right\rvert \leq A\lambda \qquad \text{ for some } \lambda >0,
	\end{align*}
	then for every $x \in B\left(x_{0}, R/4\right),$ the estimate 
	\begin{equation}\label{oscillationVdw}
		\left\lvert \left( \X w\right)_{B(x, \rho)} - \left( \X w\right)_{B(x, \varrho)} \right\rvert \leq \lambda \varepsilon \qquad \text{ holds for every } \rho, \varrho \in 
		(0, r_{\varepsilon}].
	\end{equation}
	This would imply that the sequence of maps $x \mapsto \left( \X w \right)_{B(x, \rho)}$ are uniformly Cauchy and would conclude the continuity of $\X w.$ We fix $\varepsilon >0.$ Now we choose the constants as in the proof of Theorem \ref{pointwise bound thm}, but in the the scale $\varepsilon.$ More precisely, we choose now, 
	we choose $\delta \in (0,\frac{1}{4})$ small enough such that 
	\begin{equation}\label{choiceofsigma1epsilon}
		\max \left\lbrace c_{2}\delta^{\beta}, 2C_{\gamma}\delta^{\beta}  \right\rbrace \leq \frac{\varepsilon}{4^{Q+4}}.
	\end{equation}
	Now, we fix a radius $ 0 < R_{1} < \bar{R}/16$ small enough such that we have 
	\begin{equation}\label{smallnessradius1epsilon} 
		\sum\limits_{j=1}^{\infty} \mathfrak{P}_{r_{j-1}, \theta} \leq \frac{\varepsilon \delta^{4Q}}{4C_{1}10^{2Q+4}} .
	\end{equation}
	With this, we have chosen all the relevant parameters. Fix $r< \min \{R/1000, R_1\}$ and $x \in B\left(x_{0}, R/4\right).$ Set $r_{j} = \delta^{j-1}r$ and define $B_{r_{j}} := B(x, r_{j}).$ Then we define $p_{j},$ $v_{j}, a_{j}, E_{j}$ exactly as in the proof of Theorem \ref{pointwise bound thm}.  Proceeding exactly as in the proof of estimate \eqref{excess decay at j} in the proof of Theorem \ref{pointwise bound thm}, we obtain the following.  
	
	For any $i \geq 1,$ we have 
	\begin{equation}\label{i+1 to i estimate for excess of xw}
		E_{i+1} \leq \frac{\varepsilon}{2} E_{i} +4\delta^{-\frac{Q}{\gamma}}C_{1}\lambda \mathfrak{P}_{r_{j-1}, \theta} . 
	\end{equation}
	This now easily implies that given $\varepsilon \in (0,1),$ there exists a positive radius $R_{3} \equiv R_{3}\left( \texttt{data}, \left\lVert \left\lvert \X w \right\rvert^{p\left(\cdot\right)}\right\rVert_{L^{1}\left( \Omega\right)}, \varepsilon, A \right)$ such that 
	\begin{equation}\label{excess very small for epsilon}
		\sup_{0 < \rho \leq R_{3}} \sup_{x \in B\left(x_{0}, R/4\right)} \left( \fint_{B_{\rho}} \left\lvert \X w - \left( \X w\right)_{B_{\rho}}\right\rvert^{\gamma}\right)^{\frac{1}{\gamma}} \leq \frac{\delta^{4Q}\lambda \varepsilon}{10^{10}}.
	\end{equation}
	We set $$ R_{0} := \frac{\min{\lbrace R/1000, R_{3}, R_{1} \rbrace} }{16}$$
	and again consider the chain of shrinking balls $B_{j} = B\left(x, \delta^{j-1}R_{0}\right)$ with the starting radius $R_0$.
	Now we want to show that given two integers $2 \leq i_{1} < i_{2},$ we have the estimate 
	\begin{equation}\label{integer ball average oscillation}
		\left\lvert (\X w)_{B_{i_{1}}} - (\X w)_{B_{i_{2}}} \right\rvert \leq \frac{\lambda \varepsilon}{10}.
	\end{equation}
	Note that this will complete the proof, since for any $0 < \rho < \varrho \leq \sigma^2 R_0,$ there exist integers such that 
	$ \sigma^{i_{1}+1}R_{0} < \varrho \leq \sigma^{i_{1}}R_{0}$ and $\sigma^{i_{2}+1}R_{0} < \rho \leq \sigma^{i_{2}}R_{0}. $
	Also, we have the easy estimates
	\begin{align*}
		\left\lvert (\X w)_{B_{\varrho}} - (\X w)_{B_{i_{1}+1}} \right\rvert &\leq \fint_{B_{i_{1}+1}} \left\lvert  \X w - ( \X w)_{B_{\varrho}}  \right\rvert \\
		&\leq \frac{\lvert B_{\varrho}\rvert}{\lvert B_{i_{1}+1}\rvert} \fint_{B_{\varrho}} \left\lvert  \X w - (\X w)_{B_{\varrho}}  \right\rvert \\
		&\stackrel{\eqref{excess very small for epsilon}}{\leq} \delta^{-Q} \left( \fint_{B_{\rho}} \left\lvert \X w - \left( \X w\right)_{B_{\varrho}}\right\rvert^{\gamma}\right)^{\frac{1}{\gamma}} \leq  \frac{\lambda \varepsilon}{1000}
	\end{align*}
	and similarly 
	$$  \left\lvert (\X w)_{B_{\rho}} - (\X w)_{B_{i_{2}+1}} \right\rvert \leq \frac{\lambda \varepsilon}{1000}.$$
	These two estimates combined with \eqref{integer ball average oscillation} will establish \eqref{oscillationVdw}. Thus it only remains to establish \eqref{integer ball average oscillation}. But summing up \eqref{i+1 to i estimate for excess of xw} for 
	$i = i_{1}-1$ to $i_{2}-2,$ we obtain 
	$$\sum_{i=i_{1}}^{i_{2}-1}E_{i} \leq 2 E_{i_{1}-1} + 4\delta^{-\frac{Q}{\gamma}}C_{1}\sum_{i=1}^{\infty} \mathfrak{P}_{r_{i-1}, \theta}
	\leq \frac{ \delta^{2Q} \lambda \varepsilon}{50}. $$
	This yields \eqref{integer ball average oscillation} via the elementary estimate 
	\begin{align*}
		\left\lvert (\X w)_{B_{i_{1}}} - (\X w)_{B_{i_{2}}} \right\rvert \leq \sum_{i=i_{1}}^{i_{2}-1} \fint_{B_{i+1}} \left\lvert \X w - (\X w)_{B_{i}} \right\rvert \leq \delta^{-Q} \sum_{i=i_{1}}^{i_{2}-1}E_{i} .
	\end{align*}
	This establishes \eqref{oscillationVdw} and thus, completes the proof of continuity of $\X w.$\smallskip 
	
	\noindent\textbf{Step 3:} \underline{Oscillation estimate} Now we prove the estimate \eqref{osc bound homogeneous}.  Fix $x_{0}\in \Omega$, $A, R$ and $\lambda$ satisfying $\sup\limits_{B\left(x_{0}, R/2\right)} \left\lvert \X w \right\rvert \leq A\lambda.$ Now, as in \eqref{oscillationVdw} above, we deduce that there exists $r_{\varepsilon}= r_{\varepsilon} \left( \texttt{data}, \left\lVert \left\lvert \X w \right\rvert^{p\left(\cdot\right)}\right\rVert_{L^{1}\left( \Omega\right)}, \varepsilon, A \right) \in (0, R/16)$ such that 
	\begin{equation}\label{oscillationVdwFC}
		\sup\limits_{x \in B(x_{0},R/4)}\sup\limits_{0 \leq \rho, \varrho \leq r_{\varepsilon}}\left\lvert \left( \X w\right)_{B(x, \rho)} - \left( \X w\right)_{B(x, \varrho)} \right\rvert \leq \frac{\lambda \varepsilon}{3} 
	\end{equation}
	and 
	\begin{equation}\label{smallness of the excess diniFC}
		\sup\limits_{x \in B(x_{0},R/4)}\sup\limits_{0 \leq \rho \leq r_{\varepsilon}}	\left( \fint_{B_{\rho}} \left\lvert \X w - \left( \X w\right)_{B_{\rho}}\right\rvert^{\gamma}\right)^{\frac{1}{\gamma}} \leq \frac{\lambda \varepsilon}{48^{2Q}}. 
	\end{equation}
	Since $\X w$ is continuous, by letting $\rho \rightarrow 0$ in \eqref{oscillationVdwFC}, we obtain
	\begin{align}\label{oscillationVdwFClimit}
		\sup\limits_{x \in B(x_{0},R/4)}	\left\lvert  \X w\left( x\right) - \left( \X w\right)_{B(x, \varrho)} \right\rvert \leq \frac{\lambda \varepsilon}{3} \qquad \text{ holds for every }  \varrho \in 
		(0, r_{\varepsilon}].
	\end{align}
	Now choose $\tau_1 >0$ sufficiently small such that we have 
	\begin{align*}
		\tau_1 R \leq \frac{r_{\varepsilon}}{64}.
	\end{align*}
	Now for any $ y, x\in B(x_{0} , \tau_1 R),$ we have $ 	B(y, r_{\varepsilon}/8)  \subset B(x, r_{\varepsilon }) \subset  B( x_{0}, R/8 ).$
	So, we estimate  
	\begin{align}\label{diff averrages}
		&\left\lvert  \left( \X w\right)_{B(y,  r_{\varepsilon}/8)} - \left( \X w\right)_{B(x,  r_{\varepsilon}/8)} \right\rvert 	\notag \\&\qquad\leq \fint_{B(y,  r_{\varepsilon}/8)}\left\lvert  \X w - \left( \X w\right)_{B(x,  r_{\varepsilon}/8)} \right\rvert \notag\\
		&\qquad \leq 8^Q \fint_{B(x,  r_{\varepsilon})}\left\lvert  \X w - \left( \X w\right)_{B(x,  r_{\varepsilon})} \right\rvert+ 8^Q \left\lvert \left( \X w\right)_{B(x,  r_{\varepsilon})} - \left( \X w\right)_{B(x,  r_{\varepsilon}/8)} \right\rvert  \notag\\
		&\qquad\leq 2\cdot8^{2Q} \fint_{B(x,  r_{\varepsilon})}\left\lvert  \X w - \left( \X w\right)_{B(x,  r_{\varepsilon})} \right\rvert \notag\\
		&\qquad\leq 16^{2Q} \left(\fint_{B(x,  r_{\varepsilon})}\left\lvert  \X w - \left( \X w\right)_{B(x,  r_{\varepsilon})} \right\rvert^{\gamma} \right)^{\frac{1}{\gamma}} \stackrel{\eqref{smallness of the excess diniFC}}{\leq}\frac{\lambda\varepsilon}{3}. 
	\end{align}
	Thus, by triangle inequality, for any $x, y \in B(x , \tau_1 R),$ we have 
	\begin{align*}
		\left\lvert  \X w\left( x\right) -  \X w\left( y\right) \right\rvert &
		\begin{multlined}[t]
			\leq 	\left\lvert  \X w\left( x\right) - \left( \X w\right)_{B(x, r_\varepsilon/8)} \right\rvert + 	\left\lvert  \X w\left( y\right) - \left( \X w\right)_{B(y, r_\varepsilon/8)} \right\rvert  \\ + \left\lvert  \left( \X w\right)_{B(y,  r_{\varepsilon}/8)} - \left( \X w\right)_{B(x,  r_{\varepsilon}/8)} \right\rvert,
		\end{multlined}\\
		&\stackrel{\eqref{oscillationVdwFClimit}}{\leq} \frac{\lambda\varepsilon}{3} + \frac{\lambda\varepsilon}{3} + \left\lvert  \left( \X w\right)_{B(y,  r_{\varepsilon}/8)} - \left( \X w\right)_{B(x,  r_{\varepsilon}/8)} \right\rvert \stackrel{\eqref{diff averrages}}{\leq} \lambda\varepsilon. 
	\end{align*}
	This proves \eqref{osc bound homogeneous} and completes the proof.
	\section*{Acknowledgment}
	S.Sil's research is supported by the ANRF-SERB MATRICS Project grant MTR/2023/000885.

\end{document}